\documentclass{amsart}

\usepackage{amssymb}
\usepackage{amsmath}
\usepackage{graphicx}

\newcommand{\erase}[1]{}

\newtheorem{theorem}{Theorem}[section]
\newtheorem{lemma}[theorem]{Lemma}
\newtheorem{proposition}[theorem]{Proposition}
\newtheorem{corollary}[theorem]{Corollary}

\newtheorem{_algorithm}[theorem]{Algorithm}
\newenvironment{algorithm}{\begin{_algorithm}\rm}{\hfill \rule{3pt}{6pt}\end{_algorithm}}

\newtheorem{_procedure}[theorem]{Procedure}

\newtheorem{_definition}[theorem]{Definition}
\newenvironment{definition}{\begin{_definition}\rm}{\end{_definition}}

\newtheorem{_remark}[theorem]{\it Remark}
\newenvironment{remark}{\begin{_remark}\rm}{\end{_remark}}

\newtheorem{_example}[theorem]{Example}

\newtheorem{_claim}[theorem]{Claim}
\newenvironment{claim}{\begin{_claim}\rm}{\end{_claim}}

\numberwithin{equation}{section}
\numberwithin{table}{section}
\numberwithin{figure}{section}

\newcommand{\C}{\mathord{\mathbb C}}
\newcommand{\F}{\mathord{\mathbb F}}
\renewcommand{\P}{\mathord{\mathbb  P}}
\newcommand{\Q}{\mathord{\mathbb  Q}}
\newcommand{\R}{\mathord{\mathbb R}}
\newcommand{\Z}{\mathord{\mathbb Z}}

\newcommand{\BBB}{\mathord{\mathcal B}}
\newcommand{\CCC}{\mathord{\mathcal C}}
\newcommand{\DDD}{\mathord{\mathcal D}}
\newcommand{\FFF}{\mathord{\mathcal F}}
\newcommand{\HHH}{\mathord{\mathcal H}}
\newcommand{\LLL}{\mathord{\mathcal L}}

\newcommand{\PPP}{\mathord{\mathcal P}}
\newcommand{\RRR}{\mathord{\mathcal R}}
\newcommand{\TTT}{\mathord{\mathcal T}}
\newcommand{\VVV}{\mathord{\mathcal V}}
\newcommand{\WWW}{\mathord{\mathcal W}}

\newcommand{\SSSS}{\mathord{\mathfrak S}}
\newcommand{\AAAA}{\mathord{\mathfrak A}}

\newcommand{\maprightsp}[1]{\; \smash{\mathop{\; \longrightarrow \; }\limits\sp{#1}}\; }
\newcommand{\maprightsb}[1]{\; \smash{\mathop{\; \longrightarrow \; }\limits\sb{#1}}\; }

\newcommand{\mapdown}{\phantom{\Big\downarrow}\hskip -8pt \downarrow}

\newcommand{\inj}{\hookrightarrow}
\newcommand{\surj}{\mathbin{\to \hskip -7pt \to}}

\newcommand{\isom}{\mathbin{\,\raise -.6pt\rlap{$\to$}\raise 3.5pt \hbox{\hskip .3pt$\mathord{\sim}$}\,\;}}

\newcommand{\set}[2]{\{\; {#1} \; \mid \; {#2} \;  \}}
\newcommand{\shortset}[2]{\{ {#1} \,|\, {#2}   \}}

\newcommand{\gen}[1]{\langle {#1}  \rangle}

\newcommand{\wt}{\widetilde}

\newcommand{\tensor}{\otimes}

\newcommand{\sprime}{\sp\prime}

\newcommand{\spar}[1]{\sp{(#1)}}
\newcommand{\spprime}{\sp{\prime\prime}}

\newcommand{\sptimes}{\sp{\times}}
\newcommand{\sperp}{\sp{\perp}}

\newcommand{\dual}{\sp{\vee}}

\newcommand{\inv}{\sp{-1}}

\newcommand{\Hom}{\mathord{\mathrm {Hom}}}

\newcommand{\GL}{\mathord{\mathrm {GL}}}
\newcommand{\PGL}{\mathord{\mathrm {PGL}}}
\newcommand{\PSL}{\mathord{\mathrm {PSL}}}
\newcommand{\OG}{\mathord{\mathrm {O}}}

\newcommand{\id}{\mathord{\mathrm {id}}}

\newcommand{\Ker}{\operatorname{\mathrm {Ker}}\nolimits}

\newcommand{\Aut}{\operatorname{\mathrm {Aut}}\nolimits}
\newcommand{\Gal}{\operatorname{\mathrm {Gal}}\nolimits}

\newcommand{\pr}{\mathord{\mathrm {pr}}}
\newcommand{\Sing}{\operatorname{\mathrm {Sing}}\nolimits}

\newcommand{\disc}{\operatorname{\mathrm {disc}}\nolimits}

\newcommand{\mystrutd}[1]{\phantom{\hbox{\vrule depth #1}}}

\newcommand{\invol}{\iota}

\newcommand{\tlgammaki}[2]{\tilde{\gamma\,}_{#1}\spar{#2}}
\newcommand{\dpinvol}[1]{\tau(#1)}

\newcommand{\Involsb}[1]{\mathord{\rm Invols\hskip 1pt}_{#1}}

\newcommand{\diag}{\mathord{\rm diag}}
\newcommand{\MW}{\mathord{\rm MW}}

\newcommand{\typeI}{\mathord{\rm I}}
\newcommand{\typeII}{\mathord{\rm II}}
\newcommand{\typeIII}{\mathord{\rm III}}
\newcommand{\Triv}{\mathord{\rm Triv}}
\newcommand{\fphi}{f_{\phi}}
\newcommand{\zphi}{z_{\phi}}

\newcommand{\Xk}{X_k}
\newcommand{\Sk}{S_k}
\newcommand{\Tk}{T_k}
\newcommand{\Rk}{R_k}
\newcommand{\ak}{a_k}
\newcommand{\Gk}{G_k}
\newcommand{\CCCk}{\CCC_{k}}

\newcommand{\Xsb}[1]{X_{#1}}
\newcommand{\asb}[1]{a_{#1}}
\newcommand{\Ssb}[1]{S_{#1}}

\renewcommand{\L}{L_{26}}

\newcommand{\Ue}{\mathord{{\rm U}_{{\rm ell}}}}
\newcommand{\GEeightminus}{\mathord{{\rm E}_{8}^{-}}}

\newcommand{\Gram}{\mathord{\rm G}}

\newcommand{\orbstablerow}[6]{ &{#1} & {#2} & {#3} & {#4} & {#5} & {#6} \\}

\newcommand{\ve}{\mathord{\mathbf{e}}}

\newcommand{\emb}{\varepsilon}
\newcommand{\embk}{\emb_k}

\newcommand{\intf}[1]{\langle #1\rangle}
\newcommand{\intfS}[1]{\langle #1\rangle_{S}}
\newcommand{\intfZ}[1]{\langle #1\rangle_{Z}}
\newcommand{\intfL}[1]{\langle #1\rangle_{L}}
\newcommand{\intfV}[1]{\langle #1\rangle_{V}}
\newcommand{\intfT}[1]{\langle #1\rangle_{T}}

\newcommand{\prS}{\pr_{S}}
\newcommand{\rS}{r_{S}}
\newcommand{\SX}{S_X}

\newcommand{\pnsymp}{\rho_0\spar{0}}
\newcommand{\specialsymp}{\sigma_1\spar{4}}
\newcommand{\Enrinvol}{\varepsilon_0\spar{0}}

\newcommand{\hki}[2]{h_{#1}\sp{[#2]}}
\newcommand{\tlhki}[2]{\tilde{h}_{#1}\spar{#2}}

\newcommand{\aut}{\mathord{\it Aut}}
\newcommand{\SZ}{S_{Z}}
\newcommand{\Zzero}{Z_{0}}

\newcommand{\Dzero}{D\spar{0}}
\newcommand{\Dzerog}[1]{D\sp{(0){#1}}}
\newcommand{\DDDzero}{\DDD\spar{0}}
\newcommand{\DDDzerog}[1]{\DDD\sp{(0){#1}}}

\newcommand{\DZzero}{D_Z\spar{0}}

\newcommand{\hrho}{h_{\rho}}
\newcommand{\hsigma}{h_{\sigma}}

\newcommand{\sppprime}{\sp{\prime\prime\prime}}

\newcommand{\cyc}{\mathord{\rm cyc}}
\newcommand{\tl}[1]{\tilde{#1}}

\newcommand{\hzi}{\hki{0}{i}}
\newcommand{\tlhzone}{\tlhki{0}{1}}

\newcommand{\Cen}{\mathord{{\rm Cen}}(\Enrinvol)}

\begin{document}

\title[The automorphism groups]
{The automorphism groups of certain singular $K3$ surfaces and an Enriques surface}

\author{Ichiro Shimada}
\address{
Department of Mathematics, 
Graduate School of Science, 
Hiroshima University,
1-3-1 Kagamiyama, 
Higashi-Hiroshima, 
739-8526 JAPAN
}
\email{shimada@math.sci.hiroshima-u.ac.jp}

\thanks{Partially supported by
JSPS Grants-in-Aid for Scientific Research (C) No.25400042 
}

\begin{abstract}
We present 
finite sets of generators of the full automorphism groups
of three singular $K3$ surfaces,
on which the alternating group of degree $6$ acts
symplectically.
We also present 
a finite set of generators of the full automorphism group
of an associated Enriques surface,
on which the Mathieu group $M_{10}$   acts.
\end{abstract}


\maketitle


\section{Introduction}\label{sec:Intro}
For a complex $K3$ surface $X$,
we denote by $\SX$ the N\'eron-Severi lattice of $X$
with the intersection form $\intfS{\phantom{\cdot}, \phantom{\cdot}}\colon  \SX\times \SX\to \Z$,
and by $T_X$ the orthogonal complement of $\SX$ in $H^2(X, \Z)$ with respect to the cup-product.
We call $T_X$
 the \emph{transcendental lattice} of $X$.
A complex $K3$ surface is said to be \emph{singular} 
if the rank of $S_X$
attains the possible maximum $20$.
By the result of Shioda and Inose~\cite{MR0441982},
the isomorphism class of a singular $K3$ surface $X$ is determined uniquely by 
its transcendental lattice $T_X$ with the orientation 
given by the class $[\omega_X]\in  T_X\tensor \C$ of a nowhere-vanishing holomorphic $2$-form $\omega_X$ on $X$.
Shioda and Inose~\cite{MR0441982} also showed that
the automorphism group $\Aut(X)$ of a singular $K3$ surface $X$ is infinite.
It is an important problem to determine the structure of the automorphism groups
of singular $K3$ surfaces.
\par
In this paper, we study the automorphism groups
of the following three singular $K3$ surfaces $\Xsb{0}$, $\Xsb{1}$, $\Xsb{2}$;
the Gram matrices of the transcendental lattice $\Tk:=T_{\Xk}$ of $\Xk$ is
\begin{equation}\label{eq:GramTs}
\left[
\begin{array}{cc}
6 & 0 \\ 0 & 6
\end{array}
\right]\;\;\textrm{for $k=0$},\quad 
\left[
\begin{array}{cc}
2 & 0 \\ 0 & 12
\end{array}
\right]\;\;\textrm{for $k=1$}, \quad 
\left[
\begin{array}{cc}
2 & 1 \\ 1 & 8
\end{array}
\right]\quad\textrm{for $k=2$}.
\end{equation}
(Note that the inversion of the orientation of $\Tk$   does not affect the isomorphism class of the singular $K3$ surface in these  three cases.
See, for example,~\cite{MR1820211}.)
These three $K3$ surfaces 
have a common feature in that they admit a symplectic action by the alternating group $\AAAA_{6}$ of degree $6$.
By the classification  due to  Mukai~\cite{MR958597}, we know 
that $\AAAA_{6}$ is one of the eleven maximal finite groups that act symplectically on  complex $K3$ surfaces.
(See also Kondo~\cite{MR1620514} and Xiao~\cite{MR1385511}.)
It was proved in~\cite{MR2142132} that 
every $K3$ surface with a symplectic action by $\AAAA_6$ is singular.
A characterization of singular $K3$ surfaces with a symplectic action by $\AAAA_6$
is given in~\cite{MR2926486} (see also Remark~\ref{rem:rank3}).
\par
The purpose of this paper is to 
present a finite set of generators of the full automorphism group $\Aut(\Xk)$ of $\Xk$
for $k=0,1,2$.
Moreover, we describe 
the action of $\Aut(\Xk)$ on the N\'eron-Severi lattice $\Sk:=S_{\Xk}$.
Furthermore, we calculate the automorphism group $\Aut(\Zzero)$ of an Enriques surface $\Zzero$  whose universal cover is $\Xsb{0}$.
\par
\medskip
Let $X$ be a $K3$ surface.
We let $\Aut(X)$ act on $X$ from the left, and hence on $\SX$ \emph{from the right} by the pull-back.
We denote by
$$
\varphi_X\colon \Aut(X)\to \OG(\SX)
$$
the natural representation of $\Aut(X)$ on $\SX$, where $\OG(\SX)$ is the orthogonal group of the lattice $\SX$.
Since the action of $\Aut(X)$ on $H^2(X, \C)$ preserves 
the one-dimensional subspace
$H^{2,0}(X)$, we also have
a natural representation
$$
\lambda_X\colon \Aut(X)\to \GL(H^{2,0}(X))=\C\sptimes.
$$
An automorphism $g\in \Aut(X)$  is said to be  \emph{symplectic} if $\lambda_X (g)=1$,
whereas we say that $g$ is \emph{purely non-symplectic} 
if the order of $g$ is $>1$ and  equal to the order of $\lambda_X (g)\in \C\sptimes$.
For a subgroup $G$ of $\Aut(X)$,
the subgroup $\Ker\lambda_X\cap G$ consisting 
of symplectic automorphisms belonging to  $G$
is called the \emph{symplectic subgroup of $G$}.
Let $\invol\in \Aut(X)$ be an involution.
If $\invol$ is symplectic,
then the quotient $X/\gen{\invol}$ is birational to a $K3$ surface.
Otherwise,  $X/\gen{\invol}$  is birational to either an Enriques surface or a rational surface.
According to these cases,
we say that $\invol$ is an \emph{Enriques involution} or a \emph{rational  involution}.
\par
Recall  that the N\'eron-Severi lattice
$\SX$ is canonically isomorphic to the Picard group  
of $X$.
A vector $h\in \SX$ with $n:=\intfS{h, h}>0$ is called a \emph{polarization of degree $n$} 
if the complete linear system $|\LLL_h|$
associated with  a  line bundle $\LLL_h\to X$ whose class is $h$ is non-empty and has no  fixed-components.
For a polarization $h\in \SX$, 
we denote the automorphism group of the projective model of the polarized $K3$ surface $(X, h)$ 
by 
$$
\Aut(X, h):=\set{g\in \Aut(X)}{h^g=h}.
$$
It  is easy to see that $\Aut(X, h)$ is a finite group.
Let $h\in \SX$ be a polarization of degree $2$.
Then the Galois transformation of the generically finite morphism 
$X\to \P^2$ of degree $2$  induced by $|\LLL_h|$ gives rise to a rational involution 
$$
\dpinvol {h}\colon X\to X
$$
of $X$,
which we call  the \emph{double-plane involution} associated with $h$.
\par
\medskip
Let $X_k$ $(k=0,1,2)$ be the three singular $K3$ surfaces defined above.
Recall that  $\Sk$ is the N\'eron-Severi lattice of $\Xk$.
We have the following:
\begin{proposition}\label{prop:faithful}
The action $\varphi_{\Xk}$  of $\Aut(\Xk)$ on $\Sk$ is faithful.
\end{proposition}
Hence $\Aut(\Xk)$ can be regarded as a subgroup of the orthogonal group  $\OG(\Sk)$.
\par
\medskip
Our  main results are as follows:
\begin{theorem}\label{thm:aut}
\par
{\rm (0)}
The group $\Aut(\Xsb{0})$
is generated by a purely non-symplectic automorphism 
$\pnsymp$ of order $4$ and $3+12$ double-plane involutions 
$$
\dpinvol{\hki{0}{1}}, \;\dots,\; \dpinvol{\hki{0}{3}},\;
\dpinvol{\tlhki{0}{1}}, \; \dots, \; \dpinvol{ \tlhki{0}{12}}.
$$
There exists an ample class $\asb{0}\in \Ssb{0}$ with $\intfS{\asb{0}, \asb{0}}=20$ such that 
$\Aut(\Xsb{0}, \asb{0})$ is a finite group of order $1440$.
This group $\Aut(\Xsb{0}, \asb{0})$ is generated by $\pnsymp$ and $ \dpinvol{\hki{0}{1}}, \dots, \dpinvol{\hki{0}{3}}$.
The symplectic subgroup  
of $\Aut(\Xsb{0}, \asb{0})$ is isomorphic to $\AAAA_{6}$.
There exists a unique Enriques involution $\Enrinvol$ in $\Aut(\Xsb{0}, \asb{0})$,
and the center of $\Aut(\Xsb{0}, \asb{0})$ is generated by  $\Enrinvol$.
%
%
\par
{\rm (1)}
The group $\Aut(\Xsb{1})$
is generated by  a symplectic involution $\specialsymp$ and $3+(12-1)$ double-plane involutions 
$$
\dpinvol{\hki{1}{1}}, \;\dots,\; \dpinvol{\hki{1}{3}},\;
\dpinvol{\tlhki{1}{1}}, \; \dots, \; \dpinvol{ \tlhki{1}{3}}, \;\dpinvol{ \tlhki{1}{5}},\;\dots\;  \dpinvol{\tlhki{1}{12}}.
$$
There exists an ample class $\asb{1}\in \Ssb{1}$ with $\intfS{\asb{1}, \asb{1}}=30$ such that 
$\Aut(\Xsb{1}, \asb{1})$ is 
 isomorphic to the group $\PGL_{2}(\F_9)$ of order $720$.
 This group $\Aut(\Xsb{1}, \asb{1})$ is  generated by $ \dpinvol{\hki{1}{1}}, \dots, \dpinvol{\hki{1}{3}}$, and 
its  symplectic subgroup 
is isomorphic to $\PSL_{2}(\F_9)\cong \AAAA_{6}$.
\par
{\rm (2)}
The group $\Aut(\Xsb{2})$
is generated by $3+7$ double-plane involutions 
$$
\dpinvol{\hki{2}{1}}, \;\dots,\; \dpinvol{\hki{2}{3}},\;
\dpinvol{\tlhki{2}{1}}, \; \dots, \; \dpinvol{ \tlhki{2}{7}}.
$$
There exists an ample class $\asb{2}\in \Ssb{2}$ with $\intfS{\asb{2}, \asb{2}}=12$ such that 
$\Aut(\Xsb{2}, \asb{2})$ is 
 isomorphic to the group $\PGL_{2}(\F_9)$.
 This group $\Aut(\Xsb{2}, \asb{2})$ is  generated by $ \dpinvol{\hki{2}{1}}, \dots, \dpinvol{\hki{2}{3}}$, and 
its  symplectic subgroup 
is isomorphic to $\PSL_{2}(\F_9)\cong \AAAA_{6}$.
\end{theorem}
%
%
\begin{remark}
Part of the  assertion on $\Aut(\Xsb{0}, \asb{0})$
in  Theorem~\ref{thm:aut} was proved in~\cite{MR2142132},
and 
the group structure of $\Aut(\Xsb{0}, \asb{0})$ was completely determined in~\cite{MR2287452}.
The problem of determining the full automorphism group $\Aut(\Xsb{0})$ was suggested in~\cite{MR2142132}.
\end{remark}
In fact, 
 in Corollary~\ref{cor:thebasis},
we give an explicit basis of $\Sk$ 
by means of a Shioda-Inose elliptic fibration on $\Xk$ (see Definition~\ref{def:ShiodaInose}).
Using this basis,  we obtain  
automorphisms 
generating   $\Aut(\Xk)$
in the form of $20\times 20$ matrices belonging to $\OG(\Sk)$
by Borcherds method~(\cite{MR913200}, \cite{MR1654763}).
We then extract geometric properties of these automorphisms from their  matrix representations  computationally.
Because of the size of the data, however, 
it is impossible to present all of these matrices in this paper.
Instead,  in Tables~\ref{table:hs0}, ~\ref{table:hs1} and~\ref{table:hs2}, 
we  give  the polarizations $\hki{k}{i}$ and $\tlhki{k}{i}$
of degree $2$ that appear in Theorem~\ref{thm:aut}  in the form of row vectors,
from which we can recover  the matrices of $\dpinvol{\hki{k}{i}}$ and $\dpinvol{\tlhki{k}{i}}$
by the method described in Section~\ref{sec:involutions}.
Moreover,
we present the $ADE$-type of the singularities of the branch curve of the double covering $\Xk\to\P^2$
induced by these polarizations.
The matrices of the  purely non-symplectic automorphism
$\pnsymp \in \Aut(\Xsb{0})$, the Enriques involution $\Enrinvol\in \Aut(\Xsb{0})$,  
and the symplectic involution   $\specialsymp\in \Aut(\Xsb{1})$
are given in Tables~\ref{table:pnsymp},~\ref{table:Enrinvol} and~\ref{table:specialsymp}, respectively. 
We also present 
the ample classes $\ak$ in~Table~\ref{table:aks}.
For the readers' convenience,
we put the  matrices of the generators of $\Aut(\Xk)$
and other computational data in 
the author's web paper~\cite{AutSingK3comp}.
\par
\medskip
Let $X$ be a $K3$ surface, and 
let $\PPP(X)$ denote the connected component of $\shortset{x\in \SX\tensor \R}{\intfS{x, x}>0}$
containing an ample class.
We put
$$
N(X):=\set{x\in \PPP(X)}{\intfS{x, C}\ge 0 \;\;\textrm{for any curve $C$ on $X$}\;\;}.
$$
Then $\Aut(X)$ acts on $N(X)$.
Next we investigate this action
for $X=X_0, X_1, X_2$.
\par
Let $L$ be an even hyperbolic lattice
with the symmetric bilinear form $\intfL{\phantom{a}, \phantom{a}}$, 
and let $\PPP(L)$ be one of the two  connected components of 
$\shortset{x\in L\tensor\R}{\intfL{x, x}>0}$,
which we call a \emph{positive cone} of $L$.
We let the orthogonal group $\OG(L)$ on $L$ \emph{from the right},
and put
$$
\OG^+(L):=\set{g\in \OG(L)}{\PPP(L)^g=\PPP(L)},
$$
which is a subgroup of $\OG(L)$ with  index $2$.
For $v\in L\tensor \R$ with $\intfL{v, v}<0$,
we denote by $(v)\sperp$ the real hyperplane 
$$
(v)\sperp:=\set{x\in \PPP(L)}{\intfL{x, v}=0}
$$
of $\PPP(L)$.
We put
$$
\RRR(L):=\set{r\in L}{\intfL{r, r}=-2}.
$$
Let $W(L)$ denote the subgroup of $\OG^+(L)$ generated by all the reflections
$$
s_r \colon x\mapsto x+\intfL{x, r}\cdot  r
$$
in the mirrors $(r)\sperp$ for $r\in \RRR(L)$.
We call $W(L)$ the \emph{Weyl group} of $L$.
The closure in $\PPP(L)$ of each connected component
of the complement 
$$
\PPP(L)\setminus \bigcup_{r\in \RRR(L)} (r)\sperp
$$
 of the union of the mirrors of $W(L)$ 
is a standard fundamental domain of the action of  $W(L)$ on $\PPP(L)$.
\par
We denote by $L\dual$ the \emph{dual lattice} $\Hom(L, \Z)$ 
of $L$, which contains $L$ as a submodule of finite index 
and hence is canonically embedded into $L\tensor \Q$.
A closed subset $\Sigma$ of $\PPP(L)$ \empty{with non-empty interior} is said to be a \emph{chamber}
if there exists a  set $\Delta$  of $L\dual$ such that
$\intfL{v, v}<0$ for every $v\in \Delta$, 
such that the family  of hyperplanes $\shortset{(v)\sperp}{v\in \Delta}$ is locally finite in $\PPP(L)$,
and such that 
$$
\Sigma=\set{x\in \PPP(L)}{ \intfL{x, v}\ge 0 \;\; \textrm{for any $v\in \Delta$}\;}
$$
holds.
Let $\Sigma$ be a chamber.
A hyperplane $(v)\sperp$ of $\PPP(L)$ is said to be  a \emph{wall} of $\Sigma$
if $(v)\sperp$ is disjoint from the interior of $\Sigma$ and $(v)\sperp\cap \Sigma$
contains a non-empty open subset of $(v)\sperp$.
Then there exists a unique subset $\varDelta(\Sigma)$ of $L\dual$
consisting of all \emph{primitive} vectors $v$ in $L\dual$ such that 
the hyperplane $(v)\sperp$ is a wall of $\Sigma$,
and such that $\intfL{x, v}>0$ holds for an interior point $x$ of $\Sigma$;
that is,
$\varDelta(\Sigma)$ is the \emph{set of primitive outward defining vectors of walls of $\Sigma$}.
We say that $\Sigma$ is \emph{finite} if $\varDelta(\Sigma)$ is finite.
\par
By Riemann-Roch theorem, we know that 
the cone $N(X)$  is a chamber  in the positive cone $\PPP(X)$ containing an ample class of $X$, 
and that $N(X)$ is a standard fundamental domain of the action of the Weyl group $W(\SX)$ on $\PPP(X)$. 
Moreover   $\varDelta(N(X))$ is equal to the set of all primitive vectors $v\in \SX\dual$
such that $nv$ is the class 
of a  smooth rational curve on $X$ for  some positive integer $n$.
(See, for example,~\cite{MR633161}.)
\par
\medskip
The next result  describes the chamber $N(\Xk)$ of the three singular $K3$ surfaces $\Xk$.
\begin{theorem}\label{thm:D0}
Let  $k$ be $0$, $1$ or $2$,
and let $\ak$ be the ample class of $\Xk$ given in Theorem~\ref{thm:aut}. 
Then 
there exists a finite chamber $\Dzero$ in $\PPP(\Xk)$ with the following properties;
\begin{itemize}
\item [(i)] the ample class $\ak$ is in the interior of $\Dzero$,
and the stabilizer subgroup 
$$
\set{g\in \Aut(\Xk)}{\Dzerog{g}=\Dzero}
$$
 of $\Dzero$ in $\Aut(\Xk)$
coincides with  $\Aut(\Xk, \ak)$,
\item[(ii)] $\Dzero$ is contained in $N(\Xk)$,
 and $N(\Xk)$ is the union of all $\Dzerog{g}$, where $g$ ranges through 
$\Aut(\Xk)$, 
\item[(iii)] if $g\in \Aut(\Xk)$ is not contained in $\Aut(\Xk, \ak)$, then $\Dzerog{g}$ is disjoint from
the interior of $\Dzero$, and
\item[(iv)] if $(v)\sperp$ is a wall of $\Dzero$
that is not a wall of $N(\Xk)$, then there exists a unique chamber of the form $\Dzerog{g}$
with $g\in \Aut(\Xk)$ 
such that the intersection $(v)\sperp \cap \Dzero\cap \Dzerog{g}$
contains a non-empty open subset of $(v)\sperp$.
\end{itemize}
\end{theorem}
Therefore $N(\Xk)$ is tessellated by the chambers $\Dzerog{g}$, where $g$
runs through a complete set of representatives of $\Aut(\Xk, \ak)\backslash \Aut(\Xk)$.
In fact, this tessellation extends to a tessellation of 
$\PPP(\Xk)$
by the chambers $\Dzerog{g}$, where $g$
runs through a complete set of representatives of $\aut(\Dzero)\backslash \OG^+(\Sk)$,
where  
$$
\aut(\Dzero):=\set{g\in  \OG^+(\Sk)}{\Dzerog{g}=\Dzero}
$$ 
is the stabilizer subgroup of $\Dzero$ in $\OG^+(\Sk)$.
We call each chamber  $\Dzerog{g}$ in this tessellation an \emph{induced chamber}.
(See Definition~\ref{def:inducedchamber} for a more  general definition.)
For a wall $(v)\sperp$ of $\Dzero$ that is not a wall of $N(\Xk)$,
the induced chamber $\Dzerog{g}$ such that
$(v)\sperp \cap \Dzero\cap \Dzerog{g}$
contains a non-empty open subset of $(v)\sperp$
is called the \emph{induced chamber adjacent to $\Dzero$ across $(v)\sperp$}.
\par
In fact,
we can write all elements of the set $\varDelta(\Dzero)$ explicitly
in terms of the fixed basis of $\Sk$.
Note that $\Aut(\Xk, \ak)$ acts on $\varDelta(\Dzero)$.
We describe this action and clarify the meaning of the generators of $\Aut(\Xk)$ 
given in Theorem~\ref{thm:aut}.
\begin{theorem}\label{thm:orbits}
Let  $\Dzero$ be the finite chamber in $N(\Xk)$
given in Theorem~\ref{thm:D0}.
The set $\varDelta(\Dzero)$ is decomposed into the orbits
$o_i$ in Table~\ref{table:main}
 under the action of $\Aut(\Xk, \ak)$.
\par
If $v\in o_{0}$, then $v$ is the class of a smooth rational curve on $\Xk$,
and hence $(v)\sperp$ is a wall of $N(\Xk)$.
If $k=1$ and $v\in o\sprime_{0}$, then $2v$ is the class of a smooth rational curve on $\Xsb{1}$,
and hence $(v)\sperp$ is a wall of $N(\Xsb{1})$.
\par
Suppose that  $i>0$.
Then there exists a vector $v_i\in o_i$
such that  
the involution $\dpinvol{\tlhki{k}{i}}$,  or $\specialsymp$  in the case $k=1$ and $i=4$,   in Theorem~\ref{thm:aut}
maps $\Dzero$ to the induced chamber $D\spar{i}$ in $N(\Xk)$
adjacent to $\Dzero$ across the wall $(v_i)\sperp$. 
\end{theorem}
\begin{table}
{\small
$$
\begin{array}{ccccccc}
\orbstablerow{\textrm{orbit}}{|o_i|}{\nu}{\alpha}{|\Involsb{k}\spar{i}|}{\intfS{a_k\spar{i}, \ak}}
\hline
k=0 &&&&&&\\
\orbstablerow{o_{0}}{60}{-2}{2}{}{}
\orbstablerow{o_{1}}{40}{-3/2}{3}{10=0+0+10}{32}
\orbstablerow{o_{2}}{180}{-4/3}{4}{4=0+0+4}{44}
\orbstablerow{o_{3}}{10}{-1}{4}{24=12+0+12}{52}
\orbstablerow{o_{4}}{144}{-5/6}{5}{6=0+0+6}{80}
\orbstablerow{o_{5}}{144}{-5/6}{5}{6=0+0+6}{80}
\orbstablerow{o_{6}}{240}{-2/3}{6}{4=0+0+4}{128}
\orbstablerow{o_{7}}{360}{-2/3}{6}{4=0+0+4}{128}
\orbstablerow{o_{8}}{180}{-1/3}{6}{4=0+0+4}{236}
\orbstablerow{o_{9}}{240}{-1/6}{7}{4=0+0+4}{608}
\orbstablerow{o_{10}}{240}{-1/6}{7}{4=0+0+4}{608}
\orbstablerow{o_{11}}{720}{-1/6}{7}{2=0+0+2}{608}
\orbstablerow{o_{12}}{720}{-1/6}{7}{2=0+0+2}{608}
\hline
k=1 &&&&&&\\
\orbstablerow{o_{0}}{45}{-2}{2}{}{}
\orbstablerow{o\sprime_{0}}{45}{-1/2}{7}{}{}
\orbstablerow{o_{1}}{10}{-3/2}{3}{12=0+0+12}{42}
\orbstablerow{o_{2}}{30}{-4/3}{4}{10=0+0+10}{54}
\orbstablerow{o_{3}}{72}{-5/4}{5}{6=0+0+6}{70}
\orbstablerow{o_{4}}{60}{-1}{6}{6=6+0+0}{102}
\orbstablerow{o_{5}}{12}{-5/6}{5}{16=0+1+15}{90}
\orbstablerow{o_{6}}{40}{-3/4}{6}{6=0+0+6}{126}
\orbstablerow{o_{7}}{120}{-7/12}{7}{4=0+0+4}{198}
\orbstablerow{o_{8}}{120}{-7/12}{7}{4=0+0+4}{198}
\orbstablerow{o_{9}}{120}{-1/3}{8}{4=0+0+4}{414}
\orbstablerow{o_{10}}{180}{-1/3}{8}{4=0+0+4}{414}
\orbstablerow{o_{11}}{120}{-1/12}{8}{4=0+0+4}{1566}
\orbstablerow{o_{12}}{120}{-1/12}{8}{4=0+0+4}{1566}
\hline
k=2 &&&&&&\\
\orbstablerow{o_{0}}{36}{-2}{1}{}{}
\orbstablerow{o_{1}}{12}{-4/3}{2}{16=0+1+15}{18}
\orbstablerow{o_{2}}{40}{-6/5}{3}{6=0+0+6}{27}
\orbstablerow{o_{3}}{90}{-4/5}{4}{4=0+0+4}{52}
\orbstablerow{o_{4}}{30}{-8/15}{4}{10=0+1+9}{72}
\orbstablerow{o_{5}}{30}{-8/15}{4}{10=0+1+9}{72}
\orbstablerow{o_{6}}{120}{-2/15}{5}{4=0+0+4}{387}
\orbstablerow{o_{7}}{120}{-2/15}{5}{4=0+0+4}{387}
\end{array}
$$
}
\par
\caption{The orbit decomposition of $\varDelta(\Dzero)$ by $\Aut(\Xk, \ak)$}\label{table:main}
\end{table}
In Table~\ref{table:main},
the cardinality $|o_i|$ of each orbit $o_i$ is presented.
The rational number $\nu$ indicates  the square-norm $\intfS{v, v}$ of the primitive vectors  $v\in o_i$,
and $\alpha$ indicates  $\intfS{\ak, v}$ for $v\in o_i$.
\par
An involution of $\Xk$
that  maps $\Dzero$ to  the adjacent chamber $D\spar{i}$
is not unique.
For $i\ge 0$, 
we put
$$
\Involsb{k}\spar{i}:=\set{\invol\in \Aut(\Xk)}{ \textrm{$\invol$ is of order $2$ and maps $\Dzero$ to $D\spar{i}$}}.
$$
\begin{proposition}\label{prop:InvolsD0}
The set $\Involsb{k}\spar{0}$  of involutions in $\Aut(\Xk, \ak)$
has the cardinality
\begin{eqnarray*}
|\Involsb{0}\spar{0}| &=& 91\;=\;45+1+45, \\
|\Involsb{1}\spar{0}|&=&81\;=\;45+0+36, \\
|\Involsb{2}\spar{0}| &=&81\;=\;45+0+36, 
\end{eqnarray*}
where
the right-hand summation means
\begin{eqnarray*}
&&(\textrm{the number  of symplectic involutions})\\
&\;+\;&
(\textrm{the number  of Enriques involutions})\\
&\;+\;&
(\textrm{the number  of rational involutions})\;\;\;.
\end{eqnarray*}
%
%
%
\end{proposition}
In Table~\ref{table:main},
the cardinality of the set  $\Involsb{k}\spar{i}$ is also presented for $i>0$
in the same manner.
Remark that  $\Involsb{1}\spar{4}$ contains no rational involutions,
and hence we have to put the symplectic involution $\specialsymp$ 
in the set of generators of $\Aut(\Xsb{1})$ in Theorem~\ref{thm:aut}.
Note that, for $\invol \in \Involsb{k}\spar{i}$ with $i>0$, the vector
$$
a_k\spar{i}:=a_k^{\invol}
$$
is an interior point of the adjacent chamber $D\spar{i}$,
and 
does not depend on the choice of $\invol \in \Involsb{k}\spar{i}$.
The column $\intfS{a_k\spar{i}, \ak}$  shows the degree
of $a_k\spar{i}$ with respect to $\ak$.
\par
As  a corollary,  we obtain the following:
%
%
\begin{corollary}\label{cor:smoothrationalcurves}
The action of $\Aut(\Xk)$ on the set of smooth rational curves on $\Xk$ is transitive
for $k=0$ and $k=2$,
whereas this action has  exactly two orbits for $k=1$.
\end{corollary}
\par
\medskip
 Borcherds method (\cite{MR913200}, \cite{MR1654763})
 has been applied 
to the studies of the automorphism groups of $K3$ surfaces
by several authors.
We briefly review  these works.
In~\cite{MR1618132}, Kondo applied it to the Kummer surface associated with  the Jacobian variety of a generic genus $2$ curve.
In~\cite{MR1935564K}, Kondo and Dolgachev  applied it to the supersingular $K3$ surface in characteristic $2$
with the Artin invariant $1$.
In~\cite{MR1897389}, Keum and Dolgachev  applied it to the quartic Hessian surface.
In~\cite{MR1806732_2}, Kondo and Keum applied it to the
Kummer surfaces associated with the product of elliptic curves.
In~\cite{MR3190354}, Kondo and the author applied it to the supersingular $K3$ surface in characteristic $3$
with the Artin invariant $1$. 
In~\cite{MR3113614}, Ujikawa applied it to the singular $K3$ surface whose transcendental lattice is of discriminant $7$.
The singular $K3$ surfaces whose transcendental lattices are of discriminant $3$ and $4$
had been studied by Vinberg~\cite{MR719348} by another method.
On the other hand, in~\cite{MR3286672},
we have shown that, in some cases,
 Borcherds method requires too much computation to be completed.
\par
The complexity of our results suggests that
the computer-aided calculation is indispensable
in the study of automorphism  groups of $K3$ surfaces.
The procedure to execute Borcherds method on a computer
has been already described in~\cite{ShimadaAlgoAut}.
In fact, a part of the result on $\Aut(\Xsb{2})$
has been obtained in~\cite{ShimadaAlgoAut}.
In~\cite{ShimadaAlgoAut},  
however,
we did not discuss 
the problem of converting  a matrix in $\OG(\SX)$
to a geometric automorphism of $X$.
In the present  article,
we give a method to derive geometric information of automorphisms 
from their action on $\SX$.
It turns out that 
the notion of \emph{splitting lines}~(\cite{MR2745755}, \cite{MR3166075}) 
is useful to describe the geometry of 
double plane models of $\Xk$ associated with the double-plane involutions of $\Xk$.
See Section~\ref{sec:examples} for examples.
\par
\medskip
The Enriques involution $\Enrinvol$ in $\Aut(\Xsb{0}, \asb{0})$
has been  detected also by Mukai and Ohashi~\cite{MukaiOhashiprivate}.
The Enriques surface 
$$
\Zzero:=X_0/\gen{\Enrinvol}
$$
plays an important role in their classification of finite semi-symplectic automorphism groups of Enriques surfaces.
\par
By the explicit description of $\Aut(\Xsb{0})$ and the chamber $\Dzero$ in $N(\Xsb{0})$ presented above,
we can calculate the full automorphism group $\Aut(\Zzero)$ of the Enriques surface $\Zzero$.
Let $\SZ$  denote the N\'eron-Severi lattice of $\Zzero$
with the intersection form $\intfZ{\phantom{a}, \phantom{a}}$.
Then $\SZ$ is an even unimodular hyperbolic lattice of rank $10$.
We have the following:
\begin{proposition}\label{prop:injZ}
The natural homomorphism 
$$
\varphi_Z\colon \Aut(\Zzero) \to \OG(\SZ)
$$
is injective.
\end{proposition}
Therefore we can regard $\Aut(\Zzero)$ as a subgroup of $\OG(\SZ)$.
Let $\Cen$ be the centralizer subgroup
$$
\set{g\in \Aut(\Xsb{0})}{g\,\Enrinvol=\Enrinvol g}
$$
of $\Enrinvol$ in $\Aut(\Xsb{0})$.
Since $\Xsb{0}$ is the universal covering of $\Zzero$, we have a natural surjective homomorphism
$$
\zeta\colon \Cen\surj \Aut(\Zzero),
$$
which induces an isomorphism
$\Cen/\gen{\Enrinvol}\isom \Aut(\Zzero)$. 
By Theorem~\ref{thm:aut} (0), 
we have $\Aut(\Xsb{0}, \asb{0})\subset \Cen$.
The subgroup $\zeta(\Aut(\Xsb{0},\asb{0}))$ of $\Aut(\Zzero)$ with order $720$ is generated by
\begin{equation}\label{eq:gensAutZa}
\zeta(\pnsymp),  \;\; \zeta(\dpinvol{\hki{0}{1}}), \;\; \zeta(\dpinvol{\hki{0}{2}}), \;\;  \zeta(\dpinvol{\hki{0}{3}}).
\end{equation}
We have the following:
\begin{theorem}\label{thm:autZ}
The finite subgroup  $\zeta(\Aut(\Xsb{0},\asb{0}))$ of $\Aut(\Zzero)$ is isomorphic to the Mathieu group $M_{10}$.
The double-plane involution $\dpinvol{\tlhki{0}{3}}$ of $X_0$  belongs to $\Cen$.
The automorphism group $\Aut(\Zzero)$ of $\Zzero$ is generated by $\zeta(\Aut(\Xsb{0},\asb{0}))$ and $\zeta(\dpinvol{\tlhki{0}{3}})$.
\end{theorem}
In fact,
we present  the generators~\eqref{eq:gensAutZa}  and $\zeta(\dpinvol{\tlhki{0}{3}})$ of $\Aut(\Zzero)$ 
in the form of $10\times 10$ matrices with respect to a certain basis of $\SZ$
(see Table~\ref{table:gensAutZ}).
Moreover,
we describe a chamber $\DZzero$ of $\SZ$ that
plays the same role to $\Aut(\Zzero)$ as the role $\Dzero$ plays to $\Aut(\Xsb{0})$.
\par
To the best knowledge of the author, 
Theorem~\ref{thm:autZ} is the first example of the application of  Borcherds method to the study of automorphism groups of Enriques surfaces.
\par
\medskip
This paper is organized as follows.
In Section~\ref{sec:comptools},
we
fix notions and notation about lattices, and  present three elementary algorithms that  are used throughout this paper.
In Section~\ref{sec:basis},
we give a basis of $\Sk$ in Corollary~\ref{cor:thebasis},
and a computational criterion for a vector in $\Sk$ to be nef  in Corollary~\ref{cor:nefcriterion}.
In Section~\ref{sec:Torelli},
we give a computational characterization of the image of
the natural homomorphism $\varphi_{\Xk}$ from $\Aut(\Xk)$ to $\OG(\Sk)$
and prove Proposition~\ref{prop:faithful}.
In Section~\ref{sec:Borcherds},
we confirm that the requirements to use Borcherds method given in~\cite{ShimadaAlgoAut}
are fulfilled in the cases of our singular $K3$ surfaces $\Xk$,
obtain a finite set of generators of $\Aut(\Xk)$ in the form of matrices in $\OG(\Sk)$
by this method, and 
prove Theorems~\ref{thm:D0} and~\ref{thm:orbits}. 
The embedding of $\Sk$ into the even unimodular hyperbolic lattice $\L$ of rank $26$ given in Table~\ref{table:embs} 
is the key  of this method.
In Section~\ref{sec:smoothrationalcurves},
we give an algorithm
to calculate the set of classes of smooth rational curves of a fixed degree
on a polarized $K3$ surface.
This algorithm plays  an important role in the study of splitting lines of  double plane models
of $K3$ surfaces.
In Section~\ref{sec:involutions},
we review a general theory of the involutions of  $K3$ surfaces.
In Section~\ref{sec:proof},
we prove Theorem~\ref{thm:aut}.
In Section~\ref{sec:examples},
we investigate some automorphisms on $\Xk$ in details
by means of the notion of splitting lines.
In Section~\ref{sec:Enriques},
we prove Proposition~\ref{prop:injZ} and Theorem~\ref{thm:autZ} on the Enriques surface $\Zzero$.
\par
\medskip
This work was partially completed during the author's stay in 
National University of Singapore in   August 2014.
He express his gratitude to this institution for its
great hospitality.
Thanks are also due to  Professors Shigeru Mukai,  Hisanori Ohashi and De-Qi Zhang 
for discussions.
The author also thanks the referees of the first version of this paper for comments.
\par
\medskip
{\bf Conventions.}
Throughout this paper, we work over $\C$.
Every $K3$ surface is assumed to be algebraic.
The symbol $\Aut$ denotes a geometric automorphism group,
whereas $\aut$ denotes a lattice-theoretic automorphism group.
\section{Computational tools}\label{sec:comptools}
\subsection{Lattices}
A \emph{lattice} is a free $\Z$-module $L$ of finite rank with a non-degenerate symmetric bilinear form
$\intfL{\phantom{\cdot}, \phantom{\cdot}}\colon  L\times L\to \Z$.
Suppose that a basis $e_1, \dots, e_n$  of a lattice $L$ is given.
The $n\times n$ matrix $(\intfL{e_i, e_j})$ 
is called the \emph{Gram matrix} of $L$ with respect to the basis  $e_1, \dots, e_n$.
The \emph{discriminant} $\disc L$ of $L$ is the determinant of a Gram matrix of $L$.
The group of isometries of  a lattice $L$ is denoted by $\OG(L)$.
We let $\OG(L)$ act on $L$ \emph{from the right},
and,
when a basis of $L$ is given,
each vector of $L\tensor \R$ is written as a \emph{row} vector.
A lattice $L$ is \emph{even} if $\intfL{v, v}\in 2\Z$ holds for any $v\in L$.
The \emph{signature}  of a lattice $L$   is the signature of the real quadratic space $L\tensor \R$.
A lattice $L$ of rank $n$ is \emph{hyperbolic} if $n>1$ and its signature is $(1, n-1)$,
whereas $L$  is \emph{negative-definite} if its signature is $(0, n)$.
A negative-definite lattice $L$ is a \emph{root lattice}
if $L$ is generated by the vectors in    $\RRR(L):=\shortset{r \in L}{\intfL{r, r}=-2}$.
The classification of root lattices is well-known  (see, for example,~Ebeling~\cite{MR1938666}). 
The roots in the indecomposable root systems of type $A_l$, $D_m$ and $E_n$ are labelled as in Figure~\ref{fig:ADE}.
We denote by $L(m)$ the lattice obtained from $L$ by multiplying $\intfL{\phantom{a}, \phantom{a}}$  by $m$,
and we put  $L\sp{-}:=L(-1)$.
For a subset $A$ of a lattice $L$, we denote by $\gen{A}$ the $\Z$-submodule of $L$
generated by the elements in $A$. 
\par
For an even lattice $L$,
we denote by $L\dual$ the dual lattice $\Hom (L, \Z)$ of $L$,
and by
$$
q_L\colon L\dual/L \to \Q/2\Z
$$
the \emph{discriminant form} of $L$.
See Nikulin~\cite{MR525944}
for the definition and basic properties of discriminant forms.
The automorphism group of the finite quadratic form $q_L$ is denoted by $\OG(q_L)$.
We have a natural homomorphism $\eta_L\colon \OG(L)\to \OG(q_L)$.
\par
For square matrices $M_1, \dots, M_l$,
let $\diag(M_1, \dots, M_l)$ denote  the square matrix obtained by putting $M_1, \dots, M_l$ diagonally in this order
and putting $0$ on the other part.
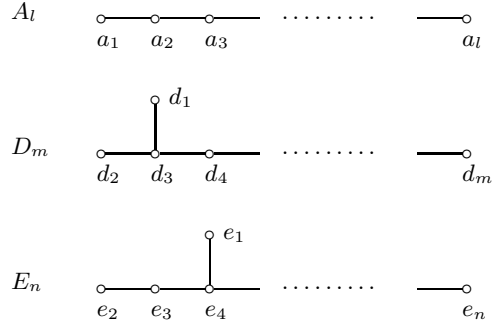
\begin{figure} \label{dynkin}
\def\ha{40}
\def\hav{37}
\def\hd{25}
\def\hdv{22}
\def\he{10}
\def\hev{7}
\setlength{\unitlength}{1.2mm}
\vskip .5cm
\centerline{
{\small
\begin{picture}(100, 37)(-20, 7)
\put(0, \ha){$A\sb l$}
\put(10, \ha){\circle{1}}
\put(9.5, \hav){$a\sb 1$}
\put(10.5, \ha){\line(5, 0){5}}
\put(16, \ha){\circle{1}}
\put(15.5, \hav){$a\sb 2$}
\put(16.5, \ha){\line(5, 0){5}}
\put(22, \ha){\circle{1}}
\put(21.5, \hav){$a\sb 3$}
\put(22.5, \ha){\line(5, 0){5}}
\put(30, \ha){$\dots\dots\dots$}
\put(45, \ha){\line(5, 0){5}}
\put(50.5, \ha){\circle{1}}
\put(50, \hav){$a\sb {l}$}
\put(0, \hd){$D\sb m$}
\put(10, \hd){\circle{1}}
\put(9.5, \hdv){$d\sb 2$}
\put(10.5, \hd){\line(5, 0){5}}
\put(16, 31){\circle{1}}
\put(17.5, 30.5){$d\sb 1$}
\put(16, 25.5){\line(0,1){5}}
\put(16, \hd){\circle{1}}
\put(15.5, \hdv){$d\sb 3$}
\put(16.5, \hd){\line(5, 0){5}}
\put(22, \hd){\circle{1}}
\put(21.5, \hdv){$d\sb 4$}
\put(22.5, \hd){\line(5, 0){5}}
\put(30, \hd){$\dots\dots\dots$}
\put(45, \hd){\line(5, 0){5}}
\put(50.5, \hd){\circle{1}}
\put(50, \hdv){$d\sb {m}$}
\put(0, \he){$E\sb n$}
\put(10, \he){\circle{1}}
\put(9.5, \hev){$e\sb 2$}
\put(10.5, \he){\line(5, 0){5}}
\put(16, \he){\circle{1}}
\put(15.5, \hev){$e\sb 3$}
\put(22, 16){\circle{1}}
\put(23.5, 15.5){$e\sb 1$}
\put(22, 10.5){\line(0,1){5}}
\put(16.5, \he){\line(5, 0){5}}
\put(22, \he){\circle{1}}
\put(21.5, \hev){$e\sb 4$}
\put(22.5, \he){\line(5, 0){5}}
\put(30, \he){$\dots\dots\dots$}
\put(45, \he){\line(5, 0){5}}
\put(50.5, \he){\circle{1}}
\put(50, \hev){$e\sb {n}$}
\end{picture}
}
}
\vskip 10pt
\caption{ Indecomposable root systems}\label{fig:ADE}
\end{figure}
\subsection{Three algorithms}
We use the following algorithms throughout this paper.
See Section 3 of~\cite{MR3166075} for the details.
Let $L$ be a lattice.
We assume that  
the Gram matrix of $L$ with respect to a certain  basis  is given.
\begin{algorithm}\label{algo:homogES}
Suppose that $L$ is negative-definite.
Then, for a negative  integer $d$, the finite set 
$\shortset{v\in L}{\intfL{v, v}=d}$
can be effectively calculated.
\end{algorithm}
\begin{algorithm}\label{algo:affES} 
Suppose that $L$ is hyperbolic, and 
let $a$ be a vector of $L$ with $\intfL{a,a}>0$.
Then, for integers $b$ and $d$,
the finite set
$$
\set{v\in L}{\intfL{a, v}=b, \;\;  \intfL{v, v}=d}
$$
can be effectively calculated.
\end{algorithm}
\begin{algorithm}\label{algo:isnef}
Suppose that $L$ is hyperbolic.
Let $a_1$ and $ a_2$ be vectors of $L$ satisfying  
$\intfL{a_1, a_1}>0$,  $\intfL{a_2, a_2}>0$ and $\intfL{a_1, a_2}>0$.
Then, for a negative integer $d$,
the finite set
$$
\set{v\in L}{\intfL{a_1, v}>0, \;\; \intfL{a_2, v}<0, \;\; \intfL{v, v}=d}
$$
can be effectively calculated.
\end{algorithm}
\section{Bases of the N\'eron-Severi lattices}\label{sec:basis}
In order to express elements  of $\Aut(\Xk)$  in the form of $20\times 20$ matrices in $\OG(\Sk)$,
we have to fix a basis of $\Sk$.
For this purpose,
we review the theory of elliptic fibrations on $K3$ surfaces.
See~\cite{MR1081832} or~\cite{MR1813537} for the details.
\par
\medskip
Let $\phi\colon X\to \P^1$ be an elliptic fibration on a $K3$ surface $X$ with a zero-section 
$\sigma_0\colon \P^1\to X$.
We denote by $\fphi\in \SX$ the class of a fiber of $\phi$,
by $\zphi\in \SX$ the class of the image of  $\sigma_0$,
and by $\MW_{\phi}$ the Mordell-Weil group of $\phi$.
We put
$$
\RRR_{\phi}:=\set{v\in \P^1}{\textrm{$\phi\inv (v)$ is reducible}},
$$
and,
for $v\in \RRR_{\phi}$,
let $\Theta_{\phi, v}\subset \SX$ denote the sublattice 
spanned by the classes of irreducible components of $\phi\inv(v)$
that are disjoint from $\sigma_0$.
Then each $\Theta_{\phi, v}$ is an indecomposable  root lattice.
We put
$$
U_{\phi}:=\gen{\fphi, \zphi},\;\;\; \Theta_{\phi}:=\bigoplus_{v\in \RRR_{\phi}} \Theta_{\phi, v}.
$$
Then $U_{\phi}$  is an even hyperbolic unimodular lattice of rank $2$, and 
we have
\begin{equation}\label{eq:Theta}
\Theta_{\phi}=
\gen{\;r\in \SX \mid \intfS{r, \fphi}=\intfS{r, \zphi}=0,\;\; \intfS{r, r}=-2\;}.
\end{equation}
%
The sublattice
$$
\Triv_{\phi}:=U_{\phi} \oplus  \Theta_{\phi} 
$$
of $\SX$ is called the \emph{trivial sublattice} of $\phi$.
For each element $\sigma\colon \P^1\to X$ of $\MW_{\phi}$,
let $[\sigma]\in \SX$ denote the class of the image of $\sigma$.
Then the mapping $\sigma\mapsto [\sigma] \bmod \Triv_{\phi}$
induces an isomorphism
\begin{equation}\label{eq:MW}
\MW_{\phi} \;\cong \SX/ \Triv_{\phi}.
\end{equation}
Recall that a reducible fiber $\phi\inv(v)$ is of type $\typeII^*$
if and only if $\Theta_{\phi, v}$ is the root lattice of type $E_8$.
\begin{definition}\label{def:ShiodaInose}
An elliptic fibration on a $K3$ surface 
is called a \emph{Shioda-Inose elliptic fibration}
if it has a zero-section $\sigma_0$ and  two singular fibers of type $\typeII^*$.
\end{definition}
Shioda and Inose~\cite{MR0441982} showed that
every singular $K3$ surface has a Shioda-Inose elliptic fibration.
Let $X$ be a singular $K3$ surface with
a  Shioda-Inose elliptic fibration $\phi\colon X\to \P^1$.
Let $v$ and $v\sprime$ be the two points in $\RRR_{\phi}$ such that
$\phi\inv (v)$ and $\phi\inv(v\sprime)$ are  of type $\typeII^*$,
and let
$e_1, \dots, e_8$ (resp.~$e_1\sprime, \dots, e_8\sprime$)
be 
 the classes of the irreducible components of  
$\phi\inv (v)$ (resp.~$\phi\inv (v\sprime)$)
disjoint from $\sigma_0$ numbered in such a way that
their dual graph is as in Figure~\ref{fig:ADE}.
Then 
 the $18$ vectors
 \begin{equation}\label{eq:basis1}
\fphi, \zphi,  e_1, \dots, e_8, e_1\sprime, \dots, e_8\sprime
\end{equation}
  span a hyperbolic \emph{unimodular} sublattice 
  $\Triv\sprime_{\phi}:=U_{\phi}\oplus \Theta_{\phi, v}\oplus \Theta_{\phi, v\sprime}$ of $\Triv_{\phi}$.
  Let $V_{\phi}$ denote the orthogonal complement of $\Triv\sprime_{\phi}$ in $\SX$,
  so that we have an orthogonal direct-sum decomposition 
\begin{equation}\label{eq:SThetaV}
  \SX=\Triv\sprime_{\phi}\oplus V_{\phi}.
 \end{equation}
 Let $V\sprime_{\phi}$ be 
 the sublattice  of $V_{\phi}$
 generated by the vectors  $r\in V_{\phi}$ with $\intfV{r, r}=-2$,
 where $\intfV{\phantom{\cdot}, \phantom{\cdot}}$ is the  symmetric bilinear form of the sublattice $V_{\phi}$ of $\SX$.
  By~\eqref{eq:Theta} and~\eqref{eq:MW}, 
  we obtain  
  \begin{equation}\label{eq:ThetaVMV}
  \Theta_{\phi}=\Theta_{\phi, v}\oplus \Theta_{\phi, v\sprime}\oplus V\sprime_{\phi},
  \quad
  \MW_{\phi}\cong V_{\phi}/V\sprime_{\phi}.
  \end{equation}
\par
We apply these results to our three singular $K3$ surfaces $\Xk$.
\begin{proposition}\label{prop:Vphi}
Let $\phi\colon \Xk\to \P^1$ be a Shioda-Inose elliptic fibration on $\Xk$.
Then $V_{\phi}\cong \Tk^{-}$.
\end{proposition}
\begin{proof}
By~\eqref{eq:SThetaV} and the fact that $\Triv_{\phi}\sprime$ is unimodular,
we have $q_{\Sk}\cong q_{V_{\phi}}$.
Since $H^2(\Xk, \Z)$ with the cup-product is an even unimodular overlattice of $\Sk\oplus \Tk$,
we have $q_{\Sk}\cong -q_{\Tk}$
by Proposition~1.6.1~of~\cite{MR525944}.
Hence we have $q_{V_{\phi}}\cong q_{\Tk^{-}}$.
Note that $V_{\phi}$ is an even negative-definite lattice of rank $2$
with discriminant $36$ (resp. $24$, resp. $15$)
if $k=0$ (resp. $k=1$, resp. $k=2$).
We can make 
a  complete list of isomorphism classes of negative-definite lattices of rank $2$ with a fixed discriminant  $d$
by the classical method of Gauss~(see~Chapter 15 of~\cite{MR1662447}, for example).
Looking at  this list for $d=36, 24$ and $15$, 
we conclude that $V_{\phi}\cong \Tk^{-}$ for $k=0, 1, 2$.
\end{proof}
\begin{remark}
In general,
the isomorphism class of the lattice $V_{\phi}$ depends on the choice of 
the Shioda-Inose elliptic fibration $\phi$.
See, for example, ~\cite{MR2346573} or~\cite{MR2452829}.
\end{remark}
\begin{proposition}\label{prop:Vphi2}
Let $\phi\colon \Xk\to \P^1$ be a Shioda-Inose elliptic fibration on $\Xk$,
and let $v$ and $v\sprime$ be as above.
\par
{\rm (0)}
Suppose that $k=0$.
Then we have $\RRR_{\phi}=\{v, v\sprime\}$,
and $\MW_{\phi}$ is a free $\Z$-module of rank $2$ 
generated by elements $\sigma_{1}, \sigma_{2}$ such that
the vectors 
\begin{equation}\label{eq:ss0}
s_1:=[\sigma_1]-3\fphi-\zphi,
\quad
s_2:=[\sigma_2]-3\fphi-\zphi
\end{equation}
form a basis of $V_{\phi}$ with the Gram matrix 
$$
M_0:=
\left[
\begin{array}{cc}
-6 & 0 \\ 0 & -6
\end{array}
\right].
$$
\par
{\rm (1)}
Suppose that $k=1$.
Then there exists a point $v\spprime\in \P^1$ such that $\RRR_{\phi}$ is equal to $\{v, v\sprime, v\spprime\}$,
and such that $\phi\inv(v\spprime)$ is of type $\typeI_{2}$ or $\typeIII$.
Let $C_1$ be the irreducible component of $\phi\inv(v\spprime)$
disjoint from the zero-section $\sigma_0$.
Then  $\MW_{\phi}$ is a free $\Z$-module of rank $1$  
generated by an element $\sigma_{2}$ such that the vectors
\begin{equation}\label{eq:ss1}
s_1:=[C_1],
\quad
s_2:=[\sigma_2]-6\fphi-\zphi
\end{equation}
form a basis of $V_{\phi}$ with the Gram matrix 
$$
M_1:=
\left[
\begin{array}{cc}
-2 & 0 \\ 0 & -12
\end{array}
\right].
$$
\par
{\rm (2)}
Suppose that $k=2$.
Then there exists a point $v\spprime\in \P^1$ such that $\RRR_{\phi}$ is equal to $\{v, v\sprime, v\spprime\}$,
and such that $\phi\inv(v\spprime)$ is of type $\typeI_{2}$ or $\typeIII$.
Let $C_1$ be the irreducible component of $\phi\inv(v\spprime)$
disjoint from $\sigma_0$.
Then  $\MW_{\phi}$ is a free $\Z$-module of rank $1$  
generated by an element $\sigma_{2}$ such that the vectors
\begin{equation}\label{eq:ss2}
s_1:=[C_1],
\quad
s_2:=-([\sigma_2]-4\fphi-\zphi)
\end{equation}
form a basis of $V_{\phi}$ with the Gram matrix 
$$
M_2:=
\left[
\begin{array}{cc}
-2 & -1 \\ -1 & -8
\end{array}
\right].
$$
\end{proposition}
\begin{proof}
By Proposition~\ref{prop:Vphi},
$V_{\phi}$ has a basis $s_1$, $s_2$ 
with respect to which the  Gram matrix of $V_{\phi}$ is $M_k$.
Since 
$$
\set{r\in V_{\phi}}{\intfV{r, r}=-2}=
\begin{cases}
\emptyset & \textrm{if $k=0$}, \\
\{s_1, -s_1\} & \textrm{if $k=1$ or $2$}, 
\end{cases}
$$
we have
$$
\MW_{\phi}\cong V_{\phi}/V_{\phi}\sprime=
\begin{cases}
\Z s_1\oplus \Z s_2 & \textrm{if $k=0$}, \\
\Z \bar{s}_2 & \textrm{if $k=1$ or $2$},
\end{cases}
$$
where $ \bar{s}_2:=s_2 \bmod \gen{s_1}$.
By~\eqref{eq:ThetaVMV},  the assertions on $\RRR_{\phi}$, the type of $\phi\inv(v\spprime)$
for $k=1$ and $ 2$,
and the structure of $\MW_{\phi}$ are proved.
Note that, for an arbitrary element $\sigma\in \MW_{\phi}$, we have
$$
\intfS{[\sigma], [\sigma]}=-2, 
\quad
\intfS{[\sigma], \fphi}=1, 
\quad
[\sigma]\perp \Theta_{\phi, v}, 
\quad
[\sigma]\perp \Theta_{\phi, v\sprime},
$$
and,  when $k=1$ or $2$, we have 
$$
\intfS{[\sigma], C_1}=0\;\;\textrm{or}\;\; 1.
$$
The projection $\pr_V([\sigma])$ of $[\sigma]$ to $V_{\phi}$ with respect to the orthogonal direct-sum decomposition~\eqref{eq:SThetaV}
is
$$
[\sigma]-(2+\intfS{[\sigma], \zphi})\fphi-\zphi,
$$
and its square-norm is $-4-2\intfS{[\sigma], \zphi}$.
\par
Suppose that $k=0$.
Then we have generators $\sigma_1, \sigma_2$  of $\MW_{\phi}$ such that 
$s_1=\pr_V([\sigma_1])$ and $s_2=\pr_V([\sigma_2])$.
From $\intfS{s_1, s_1}=\intfS{s_2, s_2}=-6$, we obtain
$\intfS{[\sigma_1], \zphi}=\intfS{[\sigma_2], \zphi}=1$ 
and the equality~\eqref{eq:ss0} follows.
\par
Suppose that $k=1$ or $2$.
Changing $s_1, s_2$ to $-s_1, -s_2$ if necessary,
we can assume that $s_1=[C_1]$.
Let $\sigma_2$ be a generator of $\MW_{\phi}\cong \Z$.
Then $[C_1]$ and $\pr_V([\sigma_2])$ generate $V_{\phi}$.
In particular, we have  $s_2=x\, \pr_V([\sigma_2])+ y[C_1]$
for some $x, y\in \Z$.
We put
$$
t:=\intfS{[\sigma_2], \zphi}, \quad u:=\intfS{[\sigma_2], [C_1]}=\intfS{[\sigma_2], s_1}.
$$
Note that $t\in \Z_{\ge 0}$ and $u \in \{0, 1\}$.
Then we have
\begin{eqnarray}
\intfS{s_1, s_2} &=& xu -2y, \label{eq:s1s2}\\
\intfS{s_2, s_2} &=& x^2(-4-2t)+2xyu - 2y^2. \label{eq:s2s2}
\end{eqnarray}
Suppose that $k=1$.
If $u=1$, then we obtain $x=2y$ from~\eqref{eq:s1s2} and $\intfS{s_1, s_2}=0$,
and hence $x^2(-7/2-2t)=-12$ holds from ~\eqref{eq:s2s2} and $\intfS{s_2, s_2}=-12$.
Since the equation $x^2(-7/2-2t)=-12$ has no integer solutions,
we have $u=0$.
Then $y=0$ and $x^2(-4-2t)=-12$ hold.
The only  integer solution of $x^2(-4-2t)=-12$ is $t=4$ and $x=\pm 1$.
Therefore, 
changing $s_2$ to $-s_2$ if necessary,  we obtain~\eqref{eq:ss1}.
Suppose that $k=2$.
Since  $\intfS{s_1, s_2}=-1$, we obtain $u=1$ and $x=2y-1$  from~\eqref{eq:s1s2}.
Substituting $x=2y-1$ in $\intfS{s_2, s_2}=-8$,
we obtain a quadratic equation
$$
(7+4t) y^2-(7+4t) y + t-2=0,
$$
which has an integer solution only when $t=2$.
When $t=2$,
we have  $(x, y)=(-1, 0)$ or $(1,1)$.
Changing $s_2$ to $-s_2+s_1$ if necessary,   we obtain~\eqref{eq:ss2}.
\end{proof}
\begin{corollary}\label{cor:thebasis}
The N\'eron-Severi lattice $\Sk$ of $\Xk$ has  a basis
\begin{equation}\label{eq:thebasis}
\fphi, \zphi,  s_1, s_2, e_1, \dots, e_8, e_1\sprime, \dots, e_8\sprime,
\end{equation}
where  $s_1$, $s_2$ are obtained in Proposition~\ref{prop:Vphi2}.
The Gram matrix of $\Sk$  with respect to this basis is 
\begin{equation*}\label{eq:GramSk}
\Gram_{k}:=\diag (\Ue, M_k, \GEeightminus, \GEeightminus),
\end{equation*}
where $\Ue:=\left[\begin{array}{cc} 0 & 1 \\ 1 & -2\end{array}\right]$, 
$M_k$ is defined in Proposition~\ref{prop:Vphi2}, 
and $\GEeightminus$ is the Cartan matrix of type $E_8$ multiplied by $-1$.
\end{corollary}
Throughout this paper,
we use the basis~\eqref{eq:thebasis}
of $\Sk$,
and the Gram matrix $\Gram_k$ of $\Sk$.
Recall that $\OG(\Sk)$ acts on $\Sk$ \emph{from the right}, so that we have
$$
\OG(\Sk)=\set{A\in \GL_{20}(\Z)}{A\; \Gram_{k}\, {}^t \hskip -2pt A=\Gram_k}.
$$
\par
\medskip
Next we investigate the chamber 
\begin{eqnarray*}
N(\Xk)&:=& \set{v\in \PPP(\Xk)}{\intfS{v, C}\ge 0\;\; \textrm{for any curve $C$ on $\Xk$}}\\
&=& \set{v\in \PPP(\Xk)}{\intfS{v, C}\ge 0\;\; \textrm{for any smooth rational  curve $C$ on $\Xk$}}
\end{eqnarray*}
 in the positive cone  $\PPP(\Xk)$ of $\Sk$.
By the definition of $\fphi$ and $\zphi$, we see that 
the vector
$$
a\sprime_k:=2\fphi+\zphi
$$
of square-norm $2$ is nef,  and hence is contained in $N(\Xk)$.
Moreover the set
\begin{equation}\label{eq:BBB}
\BBB_{k}:=\set{[C]}{\textrm{$C$ is a smooth rational curve on $\Xk$ with $\intfS{a\sprime_k, C}=0$}}
\end{equation}
is equal to
\begin{equation}\label{eq:BBB2}
\begin{cases}
\{\zphi,  e_1, \dots, e_8, e_1\sprime, \dots, e_8\sprime\} & \textrm{if $k=0$}, \\
\{\zphi,  s_1, e_1, \dots, e_8, e_1\sprime, \dots, e_8\sprime\} & \textrm{if $k=1$ or $2$}.
\end{cases}
\end{equation}
Therefore 
we have the following criterion:
\begin{corollary}\label{cor:nefcriterion}
A vector $v\in \Sk$ with   $\intfS{v, v}>0$ is nef if and only if the following  conditions 
are satisfied:
\begin{itemize}
\item[(i)] $\intfS{v, a\sprime_k}> 0$,  so that $v\in \PPP(\Xk)$, 
\item[(ii)] 
the set 
$\shortset{r\in S_k}{\intfS{r, r}=-2,\, \intfS{r, a\sprime_k}>0, \,\intfS{r, v}<0}$ is empty, and 
\item[(iii)] $\intfS{v, r}\ge 0$ for all $r \in \BBB_k$.
\end{itemize}
A nef vector $v\in \Sk$ with   $\intfS{v, v}>0$ is ample if and only if 
$$
\set{r\in \Sk}{\intfS{r, r}=-2,\; \intfS{r, v}=0}
$$
is empty. 
\end{corollary}
%
%
%
Using Corollary~\ref{cor:nefcriterion} and Algorithms~\ref{algo:affES} and~\ref{algo:isnef},
we can determine whether a given vector $v\in \Sk$ is nef or not, and ample or not.
%
%
\section{Application of Torelli theorem to $\Xk$}\label{sec:Torelli}
%
%
Let $X$ be a $K3$ surface. 
The second cohomology group $H^2(X, \Z)$
considered as an even  unimodular lattice by the cup-product is denoted by $H_X$.
By Proposition~1.6.1 of~\cite{MR525944},
the even unimodular overlattice $H_X$ of $\SX\oplus T_X$ induces an isomorphism
$$
\delta_H \colon  q_{\SX}\isom -q_{T_X}.
$$
We regard the nowhere-vanishing holomorphic $2$-form $\omega_X$ on $X$ as a vector 
of $T_X\tensor \C$.
If a $\Q$-rational subspace  $T_{\Q}$ of $H_X\tensor \Q$ satisfies $\omega_X\in T_{\Q}\tensor\C$,
then $T_{\Q}$ contains  $T_X$.
From this minimality of $T_X$,
we see that, if 
$\gamma\in \OG(H_X)$ preserves the subspace $H^{2,0}(X)=\C\,\omega_X$ of $H_X\tensor \C$,
then  $\gamma$ preserves  $T_X$.
Moreover $\gamma\in \OG(H_X)$ satisfies $\omega_X^{\gamma}=\omega_X$
if and only if $\gamma$  
acts on  $T_X$ trivially.
We define the subgroup $\CCC_X$ of $\OG(T_X)$ by 
\begin{equation}\label{eq:CCCX}
\CCC_X:=\set{\gamma\in \OG(T_X)}{\omega_X^{\gamma}=\lambda\, \omega_X\;\;\textrm{for some}\;\; \lambda\in \C\sptimes}.
\end{equation}
For positive integers $n$, we define the subgroups $\CCC_X(n)$ of $\CCC_X$ by
$$
\CCC_X(n):=\set{\gamma\in  \OG(T_X)}{\omega_X^{\gamma}=\lambda\, \omega_X\;\;\textrm{for some}\;\; \lambda\in \C\sptimes\;\; \textrm{with}\;\; \lambda^n=1}.
$$
Then we have $\CCC_X(1)=\{\id\}$.
We denote by
$$
\eta_S \colon \OG(\SX)\to \OG(q_{\SX}),
\quad
\eta_T \colon \OG(T_X)\to \OG(q_{T_X})
$$
the natural homomorphisms, 
and by
$$
\delta_H\sp*\colon \OG(q_{T_X})\isom \OG(q_{\SX})
$$
the isomorphism induced by the isomorphism $\delta_H\colon  q_{\SX}\isom -q_{T_X}$. 
By the definition of $\delta_H$,  an isometry $\gamma\in \OG(\SX)$ of $\SX$ extends to an isometry $\tilde{\gamma}$
of $H_X$ that preserves  the subspace $H^{2,0}(X)=\C\omega_X$ of $H_X\tensor \C$
if and only if
$$
\eta_S(\gamma) \in \delta_H\sp*(\eta_T(\CCC_X)).
$$
More precisely,
an isometry $\gamma\in \OG(\SX)$ extends to an isometry $\tilde{\gamma}$
of $H_X$ that satisfies  $\omega_X^{\tilde{\gamma}}=\lambda\,\omega_X$ with $\lambda^n=1$ if and only if 
$$
\eta_S(\gamma)\in \delta_H\sp*(\eta_T(\CCC_X(n))).
$$
\par
By  Torelli theorem for complex algebraic $K3$ surfaces 
due to Piatetski-Shapiro and Shafarevich~\cite{MR0284440},
we have the following.
Recall that we have the natural representations
$\varphi_X\colon \Aut(X)\to \OG(\SX)$ and $\lambda_X\colon \Aut(X)\to \GL(H^{2,0} (X))=\C\sptimes$
of $\Aut(X)$.
\begin{theorem}\label{thm:Torelli}
The kernel of $\varphi_X$
is isomorphic to 
$$
\set{\gamma\in \CCC_X}{\eta_T(\gamma)=\id}.
$$
The image of $\varphi_X$ is equal to
$$
\set{\gamma\in \OG(\SX)}{N(X)^\gamma=N(X)\;\;\textrm{and}\;\; \eta_S(\gamma) \in \delta_H\sp*(\eta_T(\CCC_X))}.
$$
More precisely, 
the image of  the subgroup $\shortset{g\in \Aut(X)}{\lambda_X(g)^n=1}$ 
of $\Aut(X)$ by $\varphi_X$
is equal to
$$
\set{\gamma\in \OG(\SX)}{N(X)^\gamma=N(X)\;\;\textrm{and}\;\; \eta_S(\gamma) \in \delta_H\sp*(\eta_T(\CCC_X(n)))}.
$$
\end{theorem}
We apply Theorem~\ref{thm:Torelli} to our singular $K3$ surfaces $\Xk$.
Let $t_1, t_2$ be the basis of $\Tk$ with the Gram matrix~\eqref{eq:GramTs}.
We denote by $\intfT{\phantom{\cdot}, \phantom{\cdot}}$ the symmetric bilinear form of $\Tk$.
We have
 $$
 |\OG(\Tk)|=\begin{cases} 8 &\textrm{if $k=0$,}\\ 4 &\textrm{if $k=1, 2$,} \end{cases}
 \quad
  |\OG(q_{\Tk})|=\begin{cases} 16 &\textrm{if $k=0$,}\\ 4 &\textrm{if $k=1, 2$.} \end{cases}
$$
Since $\intfT{\omega_{\Xk}, \omega_{\Xk}}=0$,
we see that $\omega_{\Xk}$ is equal to
\begin{equation}\label{eq:omegas}
\begin{cases}
t_1+ \sqrt{-1}\, t_2 \;\;\textrm{or}\;\; t_1- \sqrt{-1}\, t_2& \textrm{if $k=0$}, \\
t_1+ \sqrt{-6}\, t_2 \;\;\textrm{or}\;\; t_1- \sqrt{-6}\, t_2  & \textrm{if $k=1$}, \\
8t_1+(-1+ \sqrt{-15}\,)\, t_2  \;\;\textrm{or}\;\; 8t_1+(-1- \sqrt{-15}\,)\, t_2 & \textrm{if $k=2$}, \\
\end{cases}
\end{equation}
up to multiplicative constants, and 
the subgroup
$$
\CCCk:=\CCC_{\Xk}
$$ 
of $\OG(\Tk)$ defined by~\eqref{eq:CCCX} is equal to 
 \begin{equation}\label{eq:Ck}
\begin{cases}
\left\{
 \renewcommand{\arraystretch}{0.9}
\pm \left[\begin{array}{cc} 1& 0 \\ 0 & 1 \end{array}\right], \;
\pm \left[\begin{array}{cc} 0& 1 \\ -1& 0\end{array}\right]
\right\} 
\mystrutd{17pt}
& \textrm{if $k=0$}, \\
\left\{
 \renewcommand{\arraystretch}{0.9}
\pm\left[\begin{array}{cc} 1& 0 \\ 0 & 1 \end{array}\right]
\right\}
 & \textrm{if $k=1$ or $2$}. 
\end{cases}
\end{equation}
 (Note that $\CCCk$ does not depend on the choice of the two possibilities of $\omega_{\Xk}$ in~\eqref{eq:omegas}.)
%
%
%
\begin{proof}[Proof of Proposition~\ref{prop:faithful}]
By direct calculations,
we see that $\eta_T$ maps $\CCCk$ into $\OG(q_{\Tk})$ injectively.
\end{proof}
The embedding $V_{\phi}=\gen{s_1, s_2}\inj \Sk$ induces an isomorphism 
$q_{\Sk}\cong q_{V_{\phi}}$.
Let $\delta \colon q_{\Sk}\isom -q_{\Tk}$ be the isomorphism
induced from the isomorphism  $V_{\phi}\isom \Tk^-$
given by  $s_1\mapsto t_1$, $s_2 \mapsto t_2$,
and let
$$
\delta\sp*\colon \OG(q_{\Tk})\isom \OG(q_{\Sk})
$$
be the isomorphism induced by $\delta$.
\begin{lemma}
We have $\delta_H\sp*(\eta_T(\CCCk))=\delta\sp*(\eta_T(\CCCk))$.
\end{lemma}
\begin{proof}
By direct calculations,
we see that 
$\eta_T(\CCCk)$ is a normal subgroup of  $\OG(q_{\Tk})$.
Since $\delta_H^*$ and $\delta^*$ are conjugate,
we obtain the proof.
\end{proof}
Therefore we can calculate the subgroups
$$
\CCC\sprime_k:=\delta_H\sp*(\eta_T(\CCCk)), \quad \CCC\sprime_k(n):=\delta_H\sp*(\eta_T(\CCCk(n))),
$$
of $\OG(q_{\SX})$, 
even though we do not know the isomorphism $\delta_H$.
Combining these with  Proposition~\ref{prop:faithful},  we obtain the following computational criterion:
\begin{corollary}\label{cor:gcriterion}
We put
$$
G_k:=\set{\gamma\in \OG(\Sk)}{\eta_S(\gamma)\in \CCC\sprime_k}.
$$
Let $a\in N(\Xk)$ be an ample class.
Then, by the natural  representation  $\varphi_{\Xk}$,
the group $\Aut(\Xk)$ is identified with the subgroup
$\shortset{\gamma\in G_k}{\textrm{$a^\gamma$ is ample}}$
of $\OG(\Sk)$.
Under this identification,
for $g\in \Aut(\Xk)$,
we have $\lambda_{\Xk}(g)^n=1$ if and only if $\eta_{S} (g)\in \CCC\sprime_k(n)$.
\end{corollary}
%
%
\begin{remark}\label{rem:rank3}
Let $A_1$ and $A_2$ be the positive-definite even lattices of rank $3$
with Gram matrices 
$$
\left[
\begin{array}{ccc}
 2 & 1 & 0\\
 1 & 8 & 0\\
 0 & 0 & 12
\end{array}
\right],
\quad
\left[
\begin{array}{ccc}
6 & 0 & 3\\
0 & 6 & 3\\
3 & 3 & 8
\end{array}
\right], 
$$
respectively.
Suppose that  $X$ is a $K3$ surface on which $\AAAA_6$ acts symplectically.
Then 
$$
H^2(X, \Z)^{\AAAA_6}:=\shortset{v\in H^2(X, \Z)}{v^g=v\;\;\textrm{for any}\;\; g \in \AAAA_6}
$$
is isomorphic to $A_1$ or  $A_2$ (see Table 10.3  of~\cite{MR2926486}).
Hence $X$ is singular,
and its transcendental lattice is isomorphic to the orthogonal complement of an invariant polarization 
in $H^2(X, \Z)^{\AAAA_6}\cong A_{i}$.
\end{remark}
\section{Borcherds method}\label{sec:Borcherds}
Let $\L$ be an even unimodular hyperbolic lattice of rank $26$,
which is unique up to isomorphism (see, for example, Chapter V of~\cite{MR0344216}).
We denote by $\intfL{\phantom{\cdot}, \phantom{\cdot}}$ the symmetric bilinear form of $\L$.
We choose a basis 
\begin{equation}\label{eq:thebasisL}
f, z, \ve_1, \dots, \ve_8, \ve_1\sprime, \dots, \ve_8\sprime, \ve_1\spprime, \dots, \ve_8\spprime
\end{equation}
of $\L$ with respect to which  
the Gram matrix of $\L$  is equal to
\begin{equation}\label{eq:GramL}
\diag (\Ue, \GEeightminus, \GEeightminus, \GEeightminus),
\end{equation}
where  $\Ue$ and  
$\GEeightminus$ are given in Corollary~\ref{cor:thebasis}.
We consider the vector $w_0\in \L$ that is written as 
\begin{eqnarray}\label{eq:w0}
w_0&:=&(61, 30,  -68, -46, -91, -135, -110, -84, -57, -29, \\
&&\phantom{\,61, 30,}-68, -46, -91, -135, -110, -84, -57, -29, \nonumber \\
&&\phantom{\,61, 30,}-68, -46, -91, -135, -110, -84, -57, -29) \nonumber
\end{eqnarray}
in terms of the basis~\eqref{eq:thebasisL}.
\begin{remark}
In terms of the  basis of $\L\dual=\L$ dual to~\eqref{eq:thebasisL},
we have
$$
w_0=(30, 1, 1, \dots, 1)\dual.
$$
\end{remark}
Note that we have $\intfL{w_0, w_0}=0$.
Let $\PPP(\L)$ be the positive cone of $\L$
that contains $w_0$ in its closure.
The real hyperplanes 
$$
(r)\sperp=\set{x\in \PPP(\L)}{\intfL{x, r}=0}
$$
of $\PPP(\L)$,
where $r$ ranges through $\RRR(\L)=\shortset{r\in \L}{\intfL{r, r}=-2}$,
decompose $\PPP(\L)$ into the union of chambers,
each of which is a standard fundamental domain of the action of the Weyl group $W(\L)$ on $\PPP(\L)$.
We call these chambers \emph{Conway chambers}.
The action of $\OG^+(\L)$ on $\PPP(\L)$ preserves this tessellation of $\PPP(\L)$
by Conway chambers.
\begin{theorem}
We put
$$
\WWW_0:=\set{r\in \L}{\intfL{r, r}=-2, \;\; \intfL{r, w_0}=1}.
$$
Then the  chamber
$$
\DDDzero:=\set{x\in \PPP(\L)}{\textrm{$\intfL{x, r}\ge 0$ for all $r\in \WWW_0$}}
$$
of $\PPP(\L)$ is a Conway chamber,
and $(r)\sperp$ is a wall of $\DDDzero$ for any $r\in \WWW_0$.
\end{theorem}
\begin{proof}
By~\cite{MR640949} and ~\cite{MR690711},
it is enough to prove that $\gen{w_0}\sperp/\gen{w_0}$
is isomorphic  the negative-definite Leech lattice;
that is, $\gen{w_0}\sperp/\gen{w_0}$
is an even negative-definite unimodular lattice with no vectors of square-norm $-2$.
The vector 
\begin{eqnarray*}\label{eq:w0sprime}
w_0\sprime&:=&(62, 30, -71, -48, -95, -141, -115, -88, -60, -31, \\
&&\phantom{62, 30,} -68, -46, -91, -135, -110, -84, -57, -29,  \nonumber\\
&&\phantom{62, 30,} -68, -46, -91, -135, -110, -84, -57, -29)\nonumber
\end{eqnarray*}
satisfies $\intfL{w_0, w_0\sprime}=1$ and $\intfL{w_0\sprime, w_0\sprime}=0$.
Then the sublattice $\gen{w_0, w_0\sprime} $ 
of $\L$ is an even unimodular  hyperbolic lattice of rank $2$, 
and $\gen{w_0}\sperp/\gen{w_0}$ is isomorphic to the orthogonal complement 
of $\gen{w_0, w_0\sprime} $ in $\L$.
Hence $\gen{w_0}\sperp/\gen{w_0}$ is even, negative-definite and unimodular.
Moreover we can calculate a Gram matrix of $\gen{w_0}\sperp/\gen{w_0}$.
Using Algorithm~\ref{algo:homogES},
we can confirm that 
$\gen{w_0}\sperp/\gen{w_0}$ contains no vectors of square-norm $-2$.
\end{proof}
\begin{corollary}\label{cor:Conwaychamber}
Any Conway chamber is equal to $\DDDzerog{g}$
for some $g\in \OG^+(\L)$.
\end{corollary}
Since the vectors in  $\WWW_0$ span $\L$,
the vector $w_0$ is uniquely determined by the condition $\intfL{w_0, r}=1$ for any $r\in \WWW_0$.
Therefore
 $\DDDzerog{g}=\DDDzerog{g\sprime}$ implies $w_0^g=w_0^{g\sprime}$.
 \begin{definition}
We call the vector $w_0^g$ the \emph{Weyl vector} of the Conway chamber $\DDDzerog{g}$.
\end{definition}
Let $\embk\colon \Sk\inj \L$ be the linear mapping   given by
$$
\embk(\fphi)=f,\quad
\embk(\zphi)=z, \quad
\embk(e_i)=\ve\sprime_i, \quad
\embk(e_i\sprime)=\ve\spprime_i,
$$
and $\embk(s_1), \embk(s_2)$ are given in Table~\ref{table:embs},
in which $[c_1, \dots, c_8]$ denotes 
the vector
$$
c_1\ve_1+\cdots+c_8 \ve_8
$$
of $\L$.
We can easily confirm that $\embk$ is a primitive embedding of the lattice $\Sk$ into $\L$
by using the Gram matrices~\eqref{eq:GramSk}~and~\eqref{eq:GramL}.
From now on,
we consider $\Sk$ as a primitive sublattice of $\L$ by $\embk$.
Let $\Rk$ denote the orthogonal complement of $\Sk$ in $\L$.
It turns out that $\Rk$ is a root lattice of type
$$
\begin{cases}
2A_2+2A_1 & \textrm{if $k=0$,}\\
A_3+A_2+A_1 & \textrm{if $k=1$,}\\
A_4+A_2 & \textrm{if $k=2$}.
\end{cases}
$$
By Proposition~1.6.1 of~\cite{MR525944}, the even unimodular overlattice $\L$ of $\Sk\oplus \Rk$
induces an isomorphism   
$$
\delta_{L}\colon q_{\Rk}\isom -q_{\Sk}.
$$
Then $\delta_{L}$ induces an isomorphism
$$
\delta_L\sp*\colon \OG(q_{\Sk})\isom \OG(q_{\Rk}).
$$
Since $\Rk$ is negative-definite,
we can calculate all elements of $\OG(\Rk)$
and their images by the natural homomorphism
$\eta_R\colon \OG(\Rk)\to \OG(q_{\Rk})$.
We have
$$
|\OG(\Rk)|=
\begin{cases}  2304 &\textrm{if $k=0$, }  \\ 1152 &\textrm{if $k=1$, }  \\ 2880   &\textrm{if $k=2$, }  \end{cases}
$$ 
and see that  $\eta_R$ is surjective.
%
\begin{table}
\begin{eqnarray*}
\left[\begin{array}{c} \emb_0(s_1)\\ \emb_0(s_2) \end{array}\right]
&=&
\left[ 
\begin {array}{cccccccc} 3&2&4&6&6&6&4&2\\
6&4&8&12&9&6&4&2\end {array}
\right] 
\\
\left[\begin{array}{c} \emb_1(s_1)\\ \emb_1(s_2) \end{array}\right]
&=&
\left[ 
\begin {array}{cccccccc} 3&2&4&6&5&4&3&2\\
6&4&8&12&9&6&3&0\end {array}
 \right]
\\
\left[\begin{array}{c} \emb_2(s_1)\\ \emb_2(s_2) \end{array}\right]
&=&
\left[ 
\begin {array}{cccccccc} 
3&2&4&6&5&4&3&2\\
6&4&8&12&10&7&4&1
\end {array} 
\right] 
\end{eqnarray*}
\caption{The embeddings}\label{table:embs}
\end{table}
In particular, by  Proposition~1.4.2 of~\cite{MR525944},  
the subgroup $\Gk$ of $\OG(\Sk)$ defined in Corollary~\ref{cor:gcriterion} satisfies the following:
\begin{proposition}\label{prop:lift}
Every element  $\gamma\in \Gk$ extends to an  isometry $\tilde{\gamma}\in \OG(\L)$.
\qed
\end{proposition}

It is easy to see that $\embk$ maps $\PPP(\Xk)$ into $\PPP(\L)$.
%
\begin{definition}\label{def:inducedchamber}
A chamber $D$ of $\PPP(\Xk)$ is called an \emph{induced chamber}
if there exists a Conway chamber $\DDD$ such that $D=\DDD\cap \PPP(\Xk)$.
In this case,
we say that $D$ is \emph{induced by} $\DDD$.
\end{definition}
As will be seen in the proof of Theorems~\ref{thm:D0} and~\ref{thm:orbits}  below,
this definition coincides with  the definition of induced chambers in Introduction.
\par
By definition,
$\PPP(\Xk)$ is tessellated  by induced chambers,
and for a wall $(v)\sperp$ of an induced chamber $D$,
we can define the induced chamber adjacent to $D$ across the wall $(v)\sperp$.
By Proposition~\ref{prop:lift},
we have the  following:
\begin{corollary}\label{cor:preserve}
The action of $\Gk$ on $\PPP(\Xk)$
preserves the  tessellation of $\PPP(\Xk)$  by induced chambers. 
\end{corollary}
If $r\in \Sk$ satisfies $\intfS{r, r}=-2$, then we obviously have $\intfL{r, r}=-2$.
Therefore a wall of $N(\Xk)$ is the intersection of a wall of a Conway chamber and $\PPP(\Xk)$.
Hence, if $D$ is an induced chamber,
then either $D$  is contained in $N(\Xk)$
or the interior of $D$ is disjoint from $N(\Xk)$.
Therefore $N(\Xk)$ is also tessellated  by induced chambers.
\par
We denote by
$$
\prS\colon  \L\tensor \Q\to \Sk\tensor \Q
$$
the orthogonal projection.
Note that $\prS(\L)$ is contained in $\Sk\dual$.
For $r\in \RRR(\L)$, we put 
$$
r_S:=\prS(r).
$$
Using the fact that $\Rk$ contains a vector of square-norm $-2$ and hence cannot
be embedded into the negative-definite Leech lattice,  
we have the following:
\begin{proposition}[Algorithm 5.8 in~\cite{ShimadaAlgoAut}]\label{prop:Deltaw}
Suppose that the Weyl vector $w$ of a Conway chamber $\DDD$ is given.
Then the set
$$
\Delta_w:=\set{r\in \RRR(\L)}{\intfL{r, w}=1, \;\intfS{\rS, \rS}<0 \;}
$$
is finite and can be effectively calculated.
\end{proposition}
We put
$$
\ak:=\begin{cases}
2 \,\prS(w_0) & \textrm{if $k=0$ or $1$, } \\
\prS(w_0) & \textrm{if $k=2$.}
\end{cases}
$$
Then $a_k$ is a primitive vector of $S_k$
contained in $\PPP(\Xk)$.
Its coordinates with respect to the basis~\eqref{eq:thebasis}
are given in Table~\ref{table:aks}.
The square-norm  $\intfS{\ak, \ak}$ is given in Theorem~\ref{thm:aut}.
\begin{table}
{\small
$$
\setlength{\arraycolsep}{2pt}
\begin{array}{cccccccccccccc}
a_0 &\;=\;&( 122 , & 60 , & -11 , & -17 , & -136 , & -92 , & -182 , & -270 , & -220 , & -168 , & -114 , & -58,  \\
& &&&&& -136 , & -92 , & -182 , & -270 , & -220 , & -168 , & -114 , & -58 )\\
a_1 &\;=\;&( 122 , & 60 , & -29 , & -8 , & -136 , & -92 , & -182 , & -270 , & -220 , & -168 , & -114 , & -58,  \\
& &&&&& -136 , & -92 , & -182 , & -270 , & -220 , & -168 , & -114 , & -58 )\\
a_2 &\;=\;&( 61 , & 30 , & -12 , & -5 , & -68 , & -46 , & -91 , & -135 , & -110 , & -84 , & -57 , & -29,  \\
& &&&&& -68 , & -46 , & -91 , & -135 , & -110 , & -84 , & -57 , & -29 )
\end{array}
$$
}
\caption{The ample vectors $a_k$}\label{table:aks}
\end{table}
\begin{proposition}\label{prop:Dzero}
The closed subset
$$
\Dzero:=\DDDzero\cap \PPP(\Xk)
$$
of $\PPP(\Xk)$ is an induced chamber that contains
$\ak$
in its interior  and is contained in $N(\Xk)$.
In particular,  $\ak\in \Sk$ is ample.
\end{proposition}
\begin{proof}
For a vector  $r\in \L$ with $\intfL{r, r}=-2$, 
the subset 
$(r)\sperp\cap \PPP(\Xk)=\shortset{x\in \PPP(\Xk)}{\intfS{\rS, x}=0}$  of $\PPP(\Xk)$ is equal to
$$
\begin{cases}
\textrm{the real hyperplane $(\rS)\sperp$ of $\PPP(\Xk)$} & \textrm{if $\intfS{\rS, \rS}<0$}, \\
\PPP(\Xk) & \textrm{if $\rS= 0$}, \\
\emptyset & \textrm{if $\rS\ne 0$ and $\intfS{\rS, \rS}\ge 0$}.
\end{cases}
$$
Moreover, because the embedding $\embk$ maps $\PPP(\Xk)$ into $\PPP(\L)$,
if $r\in \WWW_0$ satisfies $\rS\ne 0$ and $\intfS{\rS, \rS}\ge 0$,
then
every point $x$ of $\PPP(\Xk)$ satisfies $\intfS{\rS, x}>0$.
Note that $r\in \WWW_0$ satisfies $\rS=0$ if and only if $r\in \Rk$.
\par
We first show that $\ak$ is an interior point of 
the closed subset $\Dzero$  of $\PPP(\Xk)$.
We calculate the finite set
$\Delta_{w_0}=\shortset{r\in \WWW_0}{ \intfS{\rS, \rS}<0 }$
 by  Proposition~\ref{prop:Deltaw},
and confirm that
$$
\intfL{\ak, r}>0\;\;\textrm{for all}\;\; r\in \Delta_{w_0}.
$$
Therefore,
by the above consideration,
we see that 
 $\intfS{\ak, \rS}=\intfL{\ak, r}> 0$
for any $r\in \WWW_0$ with $\rS\ne 0$.
Hence $\ak$ is an interior point of  $\Dzero$.
Therefore $\Dzero$ is an induced chamber.
\par
Next we show that $\ak$ is ample.
It is easy to see that $\intfS{a\sprime_k, a_k}>0$,
where $\ak\sprime$
is the nef vector $2\fphi+\zphi$.
By Algorithms~\ref{algo:affES} and~\ref{algo:isnef}, we see that
%
\begin{eqnarray*}
&&\set{r\in \Sk}{\intfS{r, \ak\sprime}>0, \; \intfS{r, \ak}<0, \; \intfS{r,r}=-2}=\emptyset,\\
&& \intfS{\ak, r}>0 \;\;\textrm{ for any $r\in \BBB_k$},\\
&&\set{r\in \Sk}{\intfS{r, \ak}=0, \; \intfS{r,r}=-2}=\emptyset, 
\end{eqnarray*}
where   $\BBB_k$ is defined by~\eqref{eq:BBB} and given in~\eqref{eq:BBB2}.
By  Corollary~\ref{cor:nefcriterion},
we see that $\ak$ is ample.
Since $N(X_k)$ and the interior of $\Dzero$ have a common point $\ak$,
we see that $\Dzero$ is contained in $N(\Xk)$.
\end{proof}
\erase{
\begin{remark}
We have
$$
\Dzero=\DDDzero \cap \bigcap_{r\in \WWW_0\cap \Rk}  (r)\sperp,
$$
that is, $\Dzero$ is a face of $\DDDzero$.
\end{remark}
}
\begin{proof}[Proof of Theorems~\ref{thm:D0} and~\ref{thm:orbits}]
By the results proved so far,
the assumptions required to use the main algorithm  (Algorithm 6.1) of~\cite{ShimadaAlgoAut}
are satisfied.
\par
We calculate the set 
$\varDelta(\Dzero)$
of primitive outward defining vectors of walls of $\Dzero$
from the set $\Delta_{w_0}$ above by Algorithm 3.17 of~\cite{ShimadaAlgoAut}.
Since $\varDelta(\Dzero)$ generate $\Sk\tensor\R$, we can calculate the finite group 
\begin{equation}\label{eq:autD0}
\aut (\Dzero):=\set{\gamma\in \OG(\Sk)}{\Dzerog{\gamma}=\Dzero}
\end{equation}
by Algorithm 3.18 of~\cite{ShimadaAlgoAut}.
Since $\ak$ is an interior point of $\Dzero$ and the action of $\Gk$
preserves the decomposition of $\PPP(\Xk)$ into the union of induced chambers by Proposition~\ref{prop:lift}, 
we have
$$
\Aut(\Xk, \ak)=\aut (\Dzero)\cap \Gk.
$$
Indeed, $\ak$ is proportional to the sum of the vectors in the orbit $o_0$ calculated bellow.
Thus we can calculate all elements of the finite group 
$\Aut(\Xk, \ak)$ in the form of matrices.
Thus we obtain  the set  $\Involsb{k}\spar{0}$
of involutions in $\Aut(\Xk, \ak)$.
\par
We then 
calculate the orbits of 
the action of $\Aut(\Xk, \ak)$ on $\varDelta(\Dzero)$.
Let $o_i$ be an orbit.
We choose a  vector $v_i\in o_i$.
Suppose that there exists a positive integer $n$ such that $n v_i\in \Sk$ and 
$n^2\intfS{v_i, v_i}=-2$.
Then $(v_i)\sperp=(nv_i)\sperp$ is a wall of $N(\Xk)$.
This occurs only when $o_i=o_0$ or ($k=1$ and $o_i=o\sprime_0$).
Suppose that there exists   no such positive integer $n$.
Then the induced chamber $D\spar{i}$ adjacent to $\Dzero$ across the wall $(v_i)\sperp$
is contained in $N(\Xk)$.
By  Algorithm 5.14 of~\cite{ShimadaAlgoAut}, 
we calculate the Weyl vector $w_i\in \L$
such that the corresponding Conway chamber $\DDD\spar{i}$
induces $D\spar{i}$.
From $w_i$, we calculate 
$\varDelta(D\spar{i})$ by Algorithms 3.17 and 5.8 of~\cite{ShimadaAlgoAut}. 
We then use Algorithm 3.19 of~\cite{ShimadaAlgoAut}
to search for an element $\tlgammaki{k}{i}\in \Gk$
such that $\Dzerog{\tlgammaki{k}{i}}=D\spar{i}$.
It turns out that there \emph{does} 
exist such an isometry $\tlgammaki{k}{i}$.
Hence all induced chambers in $N(\Xk)$ are congruent under the action of $\Gk$,
and $\Aut(\Xk)$ is generated by $\Aut(\Xk,\ak)$ and the isometries $\tlgammaki{k}{i}$.
Finally, we calculate 
the set 
$$
\Involsb{k}\spar{i}=\set{g\cdot \tlgammaki{k}{i}}{ g\in \Aut(\Xk,\ak),\;\textrm{and}\;\; \textrm{$g\cdot \tlgammaki{k}{i}$ is of order $2$}}
$$
of involutions  in $\Gk$ that map $\Dzero$ to $D\spar{i}$.
\end{proof}
\begin{remark}\label{rem:autD0}
The index of $\Aut(\Xk, \ak)$ in the stabilizer subgroup  $\aut (\Dzero)$ 
of $\Dzero$ in $\OG(\Sk)$ 
is $4$ for $k=0$,  and is $2$ for $k=1$ and $k=2$.
Under the action of this larger group $\aut (\Dzero)$,
the orbits $o_i$ fuse as follows:
\begin{eqnarray*}
\textrm{for $k=0$} :&& o_0,\; o_1,\; o_2,\; o_3,\; o_4\cup o_5,\; o_6,\; o_7,\; o_8,\; o_9\cup o_{10},\; o_{11}\cup o_{12},\\
\textrm{for $k=1$} :&& o_0,\;o\sprime_0,\; o_1,\; o_2,\; o_3,\; o_4,\; o_5,\; o_6,\; o_7\cup o_8,\; o_9,\;  o_{10},\; o_{11}\cup o_{12},\\
\textrm{for $k=2$}: && o_0,\; o_1,\; o_2,\; o_3,\; o_4\cup o_5,\; o_6\cup o_7. 
\end{eqnarray*}
\end{remark}
By the work in this section,  we have obtained a finite set of generators of $\Aut(\Xk)$ in the form of matrices in $\OG(\Sk)$.
Our next task is to realize them geometrically.
\section{Smooth rational curves on a $K3$ surface}\label{sec:smoothrationalcurves}
From now on to Section~\ref{sec:examples}, 
we omit the subscript $S$ in $\intfS{\phantom{\cdot}, \phantom{\cdot}}$.
\subsection{An algorithm to calculate the classes of smooth rational curves}
In order to obtain geometric information
of an automorphism  $g$ of a $K3$ surface $X$
from its action $v\mapsto v^g$
on the N\'eron-Severi lattice $\SX$ of $X$,
we introduce the following computational tool.
\begin{proposition}\label{prop:tool}
Let $h\in \SX$ be a polarization of degree $n:=\intf{h, h}>0$.
Suppose that an ample class $a\in \SX$ is given.
Then, 
for each non-negative integer $d$,
we can calculate effectively the set
$$
\CCC_d (h):=\set{[\Gamma]\in S_X}{\textrm{\rm $\Gamma$ is a  smooth rational curve   on $X$ such that $\intf{h, \Gamma}=d$}}.
$$
\end{proposition}
First we prove two lemmas. 
In the following, we fix a nef class $h\in \SX$  and an ample class $a\in \SX$.
\begin{lemma}\label{lem:D}
Let $D$ be an effective divisor on $X$ with  $\intf{D, D}<0$,  and
let
$$
D=\Gamma_{0}+\dots+\Gamma_m +M
$$ 
be  a decomposition of $D$
such that $\Gamma_{0}, \dots, \Gamma_m$ are smooth rational curves
and either $M=0$ or $M$ is effective with no fixed components in $|M|$.
Then  there exists a
smooth rational curve  $\Gamma_i$
among  $\Gamma_{0}, \dots, \Gamma_m$  such that $\intf{D, \Gamma_i}<0$.
\end{lemma}
\begin{proof}
If $\intf{D, \Gamma_i}\ge 0$ for $i=0, \dots, m$,
then $\intf{D, D}=\sum \intf{D, \Gamma_i}+\intf{D, M}\ge 0$.
\end{proof}
\begin{lemma}\label{lem:Gammas}
Suppose that  $v\in \SX$ satisfies $\intf{v, v}=-2$ and $\intf{a, v}>0$.
Then the following conditions are equivalent:
\begin{itemize}
\item[(i)]
The vector $v$ is not the class of a smooth rational curve.
\item[(ii)]
There exists a smooth rational curve $\Gamma$
satisfying the following:
$$
\intf{a, \Gamma}<\intf{a, v},
\quad
\intf{h, \Gamma}\le \intf{h, v},
\quad
\intf{v, \Gamma}<0.
$$
\end{itemize}
Suppose further  that $h$ is a polarization of degree $n:=\intf{h, h}>0$  and  that $\intf{h, v}>0$.
Then the above two conditions are equivalent to 
the following:
\begin{itemize}
\item[(iii)]
There exists a smooth rational curve $\Gamma$
satisfying the following:
$$
\intf{a, \Gamma}<\intf{a, v},
\quad
\intf{h, \Gamma}< \intf{h, v},
\quad
\intf{v, \Gamma}<0.
$$
\end{itemize}
\end{lemma}
\begin{proof}
By $\intf{v, v}=-2$ and $\intf{a, v}>0$,
there exists an effective divisor $D$ such that $v$ is the class of $D$.
Let
$D=\Gamma_{0}+\dots+\Gamma_m +M$ 
be a decomposition of $D$
such that $\Gamma_{0}, \dots, \Gamma_m$ are smooth rational curves
and either $M=0$ or $M$ is effective with no fixed components in $|M|$.
By Lemma~\ref{lem:D},
we can assume  that $\intf{v, \Gamma_{0}}=\intf{D, \Gamma_{0}}<0$.
Since $h$ is nef, we have  $\intf{h, \Gamma_{0}}\le \intf{h, D}$.
\par
Suppose that $D$ is not irreducible.
Then $m>0$ or $M\ne 0$.
In either case, we have
$\intf{a, \Gamma_{0}}<\intf{a, D}$.
Hence (ii) holds  by taking $\Gamma_0$ as $\Gamma$.
Suppose that (ii) holds.
Since $\intf{D, \Gamma}<0$,
$\Gamma$ is one of $\Gamma_{0}, \dots, \Gamma_m$.
Since $\intf{a, \Gamma}<\intf{a, D}$,
we have $D\ne \Gamma$, and hence $D$ is not irreducible.
Thus the first part of Lemma~\ref{lem:Gammas} is proved.
\par
Suppose that $h$ is a polarization and that $d:=\intf{h, v}>0$.
The implication (iii) $\Longrightarrow$ (ii) is obvious.
We assume (i) and prove that (iii) holds.
If $M\ne 0$, then $\intf{h, M}>0$.
Hence we have $\intf{h, \Gamma_{0}}<d$,
and  (iii) holds by taking $\Gamma_0$ as $\Gamma$.
Therefore we can assume that $M=0$ and $m>0$.
If $\intf{h, \Gamma_{0}}<d$, then 
(iii) holds  by taking $\Gamma_0$ as $\Gamma$. 
Therefore we further assume that $\intf{h, \Gamma_{0}}=d$.
Then we have 
\begin{equation}\label{eq:hperp}
\intf{h, \Gamma_{i}}=0<d\;\;\textrm{for\; $i=1, \dots, m$}.
\end{equation}
If $\intf{v, \Gamma_{i}}<0$ for some $i>0$,
then (iii) holds  by taking $\Gamma_i$ as $\Gamma$.
Therefore we assume
\begin{equation}\label{eq:contraassump}
\intf{v, \Gamma_i}\ge 0\;\;\textrm{for\;$i=1, \dots, m$}, 
\end{equation}
and  derive a contradiction.
For simplicity, 
we put
$$
\Sigma_j:=\sum_{i=0}^j \Gamma_i, \quad \Xi_j:=\sum_{i=j+1}^m  \Gamma_i.
$$
Note that $\Gamma_{0}$ is distinct from any of $\Gamma_{1}, \dots, \Gamma_m$.
Since $\intf{\Gamma_i, \Gamma_{0}}\ge 0$ for $i> 0$ and 
$$
\intf{v, \Gamma_{0}}=\intf{D, \Gamma_{0}}=-2+\intf{\Xi_0, \Gamma_{0}} <0,
$$
we have 
$\intf{\Xi_0, \Gamma_{0}}=0$ or $1$.
If $\intf{\Xi_0, \Gamma_{0}}=0$, then $\intf{D, D}=-2$ implies that 
$\intf{\Xi_0, \Xi_0}=0$.
Since the class of $\Xi_0$ belongs to the orthogonal complement $[h]\sperp$  of $h$ in $\SX$ by~\eqref{eq:hperp},
and $[h]\sperp$
is negative-definite because $\intf{h, h}>0$, we obtain  $\Xi_0=0$,
which contradicts the assumption (i).
Hence $\intf{\Xi_0, \Gamma_{0}}=1$,
and therefore there exists  a curve $\Gamma_i$ among $\Gamma_{1}, \dots, \Gamma_m$,
say $\Gamma_{1}$, 
such that
$$
\intf{\Gamma_{0}, \Gamma_{1}}=1,\qquad
\intf{\Gamma_{0}, \Gamma_i}= 0\quad(i=2, \dots, m).
$$
 \par
 We consider the following property $P_k$:
 \begin{itemize}
 \item[(a)] $\{\Gamma_{0}, \dots, \Gamma_k\} \cap \{\Gamma_{k+1}, \dots, \Gamma_m\}=\emptyset$,
 \item[(b)] $\Gamma_{0}, \dots, \Gamma_k$ form an $A_{k+1}$-configuration of smooth rational curves.
 \item[(c)] $\intf{\Gamma_i, \Gamma_j}=0$ if $i<k$ and $j>k$.
 \item[(d)] $\intf{\Xi_k, \Gamma_k}=1$.
 \end{itemize}
 We have shown that the property $P_0$ holds.
 (The property (c) is vacuous for $P_0$.)
 \begin{claim}\label{claim:Pk}
 Suppose that the property $P_k$ holds.
 Then,  after renumbering of $\Gamma_{k+1}, \dots, \Gamma_m$,
the property $P_{k+1}$ holds.
 \end{claim}
 \begin{proof}[Proof of Claim~\ref{claim:Pk}]
Since  $\intf{\Xi_k, \Gamma_k}=1$
and $\Gamma_k\notin \{\Gamma_{k+1}, \dots, \Gamma_m\}$,
there exists a unique element, say $\Gamma_{k+1}$,
in the set $\{\Gamma_{k+1}, \dots, \Gamma_m\}$
such that
$\intf{\Gamma_{k}, \Gamma_{k+1}}=1$ and 
$\intf{\Gamma_{k}, \Gamma_{j}}=0$ for $j>k+1$.
Then we have that 
\begin{equation}\label{eq:distinct}
\Gamma_{k+1}\notin \{\Gamma_{k+2}, \dots, \Gamma_m\},
\end{equation}  
that $\Gamma_{0}, \dots, \Gamma_{k+1}$ form an $A_{k+2}$-configuration of smooth rational curves,
and that  $\intf{\Gamma_i, \Gamma_j}=0$ if $i<k+1$ and $j>k+1$.
Therefore it is enough to show that $\intf{\Xi_{k+1}, \Gamma_{k+1}}=1$.
We have $\intf{\Sigma_k, \Sigma_k}=-2$ by (b) for the property $P_k$,
and $\intf{\Sigma_k, \Xi_k}=1$ by (c) and (d) for $P_k$.
From $D^2=(\Sigma_k+\Xi_k)^2=-2$,
we obtain $\Xi_k^2=-2$.
By Lemma~\ref{lem:D},
there exists an irreducible component $\Gamma_{l}$ of $\Xi_k$
such that $\intf{\Xi_k, \Gamma_{l}}<0$.
If $l>k+1$,
then we have  $\intf{\Gamma_i, \Gamma_{l}}=0$ for $i\le k$,
and hence $\intf{D, \Gamma_{l}}=\intf{\Xi_k, \Gamma_{l}}<0$,
which contradicts the assumption~\eqref{eq:contraassump}.
Hence we have $l=k+1$.
From
$$
\intf{\Xi_k, \Gamma_{k+1}}=-2+\intf{\Xi_{k+1}, \Gamma_{k+1}}<0
$$ 
and 
$\intf{\Xi_{k+1}, \Gamma_{k+1}}\ge 0$ by~\eqref{eq:distinct},
we see that $\intf{\Xi_{k+1}, \Gamma_{k+1}}=0$ or $1$.
If $\intf{\Xi_{k+1}, \Gamma_{k+1}}=0$, then $\intf{\Xi_{k+1}, \Sigma_{k+1}}=0$ by (c) for $P_{k+1}$
and, from $D^2=(\Xi_{k+1}+\Sigma_{k+1})^2=-2$ and $\Sigma_{k+1}^2=-2$ by (b) for $P_{k+1}$,
we have $\Xi_{k+1}^2=0$.
Since the class of $\Xi_{k+1}$ belongs to  the negative-definite lattice $[h]\sperp$,
we have $\Xi_{k+1}=0$,
and hence $D=\Sigma_{k+1}$.
Then $\intf{D, \Gamma_{k+1}}<0$, 
which contradicts the assumption~\eqref{eq:contraassump}.
Therefore $\intf{\Xi_{k+1}, \Gamma_{k+1}}=1$.
\end{proof}
Since the property $P_0$ holds, the property $P_{m}$ holds by Claim~\ref{claim:Pk}, 
which says that $\Gamma_0, \dots, \Gamma_m$ form an $A_{m+1}$-configuration.
This contradicts~\eqref{eq:contraassump} for $i=m$.
\end{proof}
\begin{proof}[Proof of Proposition~\ref{prop:tool}]
Since $\intf{h,h}>0$, 
we can calculate the finite set
$$
\VVV_d:=\set{v\in \SX}{\intf{h, v}=d, \; \intf{a, v}>0, \; \intf{v, v}=-2}
$$
by Algorithm~\ref{algo:affES}.
Suppose that $d=0$.
We decompose $\VVV_0$ into the disjoint union of subsets
$$
\VVV_{0}[ \alpha_i]:=\set{v\in \VVV_d}{\intf{a, v}=\alpha_i}
$$
with  $0<\alpha_{0}<\dots<\alpha_N$.
We calculate $\CCC_0 [\alpha_i]$ inductively on $i$ by
setting $\CCC_0 [\alpha_{0}]:=\VVV_0 [\alpha_{0}]$,
and 
$$
\CCC_0 [\alpha_i]:=\set{v\in \VVV_{0}[ \alpha_i]}{
\text{there exist no vectors $\gamma$ in $\bigcup_{j<i} \CCC_0 [\alpha_j]$ such that $\intf{v, \gamma}<0$}}.
$$
Then the union of $\CCC_0 [\alpha_{0}], \dots, \CCC_0 [\alpha_N]$
is the set $\CCC_0 (h)$.
Suppose that $d>0$ and that the set $\CCC_{d\sprime}(h)$ is calculated for every $d\sprime<d$.
Then 
$$
\set{v\in \VVV_{d}}{
\text{there exist no vectors $\gamma$ in $\bigcup_{d\sprime<d} \CCC_{d\sprime}(h)$ such that $\intf{v, \gamma}<0$}}
$$
is the set $\CCC_d (h)$.
\end{proof}
Suppose that $h\in \SX$ is a polarization of degree  $n:=\intf{h, h}>0$.
Let
$$
\Phi_h\colon X\maprightsp{\rho_h} X_h \maprightsp{} \P^{1+n/2}
$$
be the Stein factorization of the morphism $\Phi_h$ 
induced by the complete linear system  $|\LLL_h|$
associated with  a  line bundle $\LLL_h\to X$ whose class is $h$.
Then $X_h$ has only rational double points as its singularities,
and $\rho_h$ is the minimal resolution of singularities.
The set $\CCC_0(h)$
is equal to the set of classes of smooth rational curves contracted by $\rho_h$.
In particular, the dual graph of 
$\CCC_0(h)$ is a disjoint union of indecomposable root systems of
type $A_l$, $D_m$ or $E_n$ (see Figure~\ref{fig:ADE}).
We can calculate the $ADE$-type of the singular points $\Sing (X_h)$ of  $X_h$
from $\CCC_0(h)$.
\par
The set $\CCC_1(h)$ is the set of classes of smooth rational curves  that are
mapped to  lines in $\P^{1+n/2}$ isomorphically by $\Phi_h$;
that is,   $\CCC_1(h)$ is the set of classes of lines of the polarized $K3$ surface $(X, h)$.
\subsection{Application to  projective models}
\begin{definition}
Let $(X, h)$ and $(X\sprime, h\sprime)$ be  polarized $K3$ surfaces.
We say that $(X, h)$ and $(X\sprime, h\sprime)$ have the \emph{same line configuration}
if there exists a bijection
$$
\alpha\colon \CCC_0(h)\cup \CCC_1(h) \isom \CCC_0(h\sprime)\cup \CCC_1(h\sprime)
$$
such that we have 
\begin{equation}\label{eq:alphacond}
\intf{\alpha(r), h\sprime}=\intf{r, h}\quad \textrm{for any $r \in \CCC_0(h)\cup \CCC_1(h)$}, 
\end{equation}
(that is, $\alpha(\CCC_0(h))=\CCC_0(h\sprime)$ and $\alpha(\CCC_1(h))=\CCC_1(h\sprime)$ hold),  and
\begin{equation}\label{eq:alphacond2}
\intf{\alpha(r), \alpha(r\sprime)}=\intf{r, r\sprime} \quad \textrm{for any $r, r\sprime\in \CCC_0(h)\cup \CCC_1(h)$}.
\end{equation}
We say that the line configuration on $(X, h)$ is \emph{full} if  the union of $\CCC_0(h)$ and $\CCC_1(h)$ generates $\SX$.
\end{definition}
\begin{proposition}\label{prop:uptpfinite}
Suppose that $X$ is singular and that the line configuration on $(X, h)$ is full.
Then,
up to isomorphism,
there exist only a finite number of polarized $K3$ surfaces
$(X\sprime, h\sprime)$ that have the same line configuration as $(X, h)$.
Moreover all such polarized $K3$ surfaces
$(X\sprime, h\sprime)$ satisfy $\intf{h\sprime, h\sprime}=\intf{h, h}$.
\end{proposition}
\begin{proof}
Suppose that $(X\sprime, h\sprime)$ has the same line configuration as $(X, h)$,
and let $\alpha$ be a bijection from $\CCC_0(h)\cup \CCC_1(h)$ to $\CCC_0(h\sprime)\cup \CCC_1(h\sprime)$
satisfying~\eqref{eq:alphacond}~and~\eqref{eq:alphacond2}.
Let $S\spprime$ be 
the sublattice  of $S_{X\sprime}$ generated by 
the union of $\CCC_0(h\sprime)$ and $\CCC_1(h\sprime)$.
Then $\alpha$ induces an isometry $\tl{\alpha}$ from $S_X$ to $S\spprime$.
Therefore $X\sprime$ is singular
and
$$
\disc T_{X\sprime}=-\disc S_{X\sprime}=-\disc \SX/m^2=\disc T_{X}/m^2,
$$
where $m$ is the index of $S\spprime$ in $S_{X\sprime}$.
Since the number of isomorphism classes of definite lattices of a fixed discriminant is finite,
the number of isomorphism classes of singular $K3$ surfaces $X\sprime$
that admit a polarization $h\sprime$
with the same line configuration as $(X, h)$ is finite.
Note that the isometry  $\tl{\alpha}\colon \SX\isom S\sprime$ maps $h$ to $h\sprime$,
because $h$ is uniquely determined by $\CCC_0(h)$ and $\CCC_1(h)$ as a unique vector 
satisfying $\intf{r, h}=0$ for any $r\in \CCC_0(h)$ and  $\intf{r, h}=1$ for any $r\in \CCC_1(h)$.
In particular, we have $\intf{h, h}=\intf{h\sprime, h\sprime}$.
For a fixed $K3$ surface $X\sprime$,
the number of polarizations $h\sprime$ with a fixed degree  is finite up to $\Aut(X\sprime)$
by Sterk~\cite{MR786280}.
\end{proof}
We apply this consideration to our singular $K3$ surfaces $\Xk$.
Recall that the inversion  of the orientation of $\Tk$ yields a singular $K3$ surface  isomorphic to $\Xk$.
\begin{proposition}
Let $h$ be a polarization on $\Xk$ of degree  $n:=\intf{h, h}>0$ such that 
the line configuration on $(\Xk, h)$ is full.
Suppose that $(X\sprime, h\sprime)$ has the same line configuration as $(\Xk, h)$.
Then either
 $X\sprime$ is isomorphic to $\Xk$, or $k=0$ and $X\sprime$
is the  singular $K3$ surface with
$T_{X\sprime}=\left[\begin{array}{cc} 2 & 0 \\ 0 & 2\end{array}\right]$.
\end{proposition}
\begin{proof}
Since $X\sprime$ is a singular $K3$ surface by Proposition~\ref{prop:uptpfinite},  we have $\disc T_{X\sprime}\equiv 0\;\textrm{or}\; 3\bmod 4$.
By the proof of Proposition~\ref{prop:uptpfinite},
we see that $\disc T_{X\sprime}=\disc \Tk /m^2$,
and if $m=1$, then $T_{X\sprime}\cong \Tk$ by the proof of Proposition~\ref{prop:Vphi}.
\end{proof}
Therefore, if 
the line configuration of $(\Xk, h)$ is full,
then we can determine  
the projective model of the polarized $K3$ surface $(\Xk, h)$
up to finite possibilities.
\section{Involutions of $K3$ surfaces}\label{sec:involutions}
Let $X$ be a $K3$ surface
such that the representation $\varphi_X\colon \Aut(X) \to \OG(\SX)$ is injective.
Suppose that we are given the action  of  an involution $\iota\in \Aut(X)$ on $\SX$
as a matrix.
In this section, we discuss a method to obtain geometric properties
of $\invol$ from this matrix.
\subsection{Types of the involution}\label{subset:type}
Note that we have $\lambda_X(\invol)=\pm 1$, 
where $\lambda_X$ is the natural representation  of $\Aut(X)$ on $H^{2,0}(X)$.
Since we have assumed that $\varphi_X$ is injective, 
we can determine, by Theorem~\ref{thm:Torelli}, 
whether $\invol$ is symplectic or not by seeing whether $\eta_S(\varphi_X(\invol))\in \OG(q_{\SX})$ is the identity  or not. 
\par
Suppose that $\invol$ is not symplectic.
Then we can determine whether $\invol$ is Enriques or rational by the following:
\begin{proposition}[Keum~\cite{MR1060704}]\label{prop:EnriquesInvol}
Let $\invol\colon X\to X$ be an involution.
We put
$$
\SX^{+}:=\shortset{v\in \SX}{v^{\invol}=v},
\quad
\SX^{-}:=\shortset{v\in \SX}{v^{\invol}=-v}.
$$
Let $\SX^{+}(1/2)$  denote the $\Q$-lattice obtained from 
the lattice $\SX^{+}$
by multiplying the symmetric bilinear form with $1/2$.
Then $\invol$ is an Enriques involution if and only if
$\SX^{+}(1/2)$ is an even unimodular hyperbolic lattice of rank $10$
and $\SX^{-}$ contains no vectors $r$ with $\intf{r, r}=-2$.
\end{proposition}
\begin{remark}
Since $\SX^{+}$ contains an ample class, its orthogonal complement $\SX^{-}$ is negative-definite.
Therefore we can calculate $\shortset{r\in \SX^{-}}{\intf{r,r}=-2}$ by Algorithm~\ref{algo:homogES}.
\end{remark}
\subsection{Polarizations of degree $2$}\label{subsec:poldeg2}
We have the following:
\begin{proposition}[Theorem 5 of~\cite{MR771073}, Proposition 0.1 of~\cite{MR1260944}]\label{prop:Nikulinpol}
Let $h\in \SX$ be a nef class with $n:=\intf{h, h}>0$,
and let $\LLL_h\to X$ be a line bundle whose class is $h$.
Let
$$
|\LLL_h|=|M|+Z
$$
be the decomposition of the complete linear system $|\LLL_h|$
into the movable part $|M|$ and the sum $Z$ of the fixed components.
Then either one of the following holds:
\begin{itemize}
\item[(i)]  $Z$ is empty, 
 and $|\LLL_h|$ defines a morphism $\Phi_h\colon X\to \P^{1+n/2}$.
In other words, $h$ is a polarization of degree $n$.
\item[(ii)]  $Z$ is a smooth rational curve,
and $|M|$ contains a member $mE$,
where $m=1+n/2$ and $E$ is a smooth curve of genus $1$
satisfying $\intf{E, Z}=1$.
The complete linear system $|E|$ defines 
an elliptic fibration $\phi\colon X\to \P^1$
with a zero-section $Z$.
In other words, we have $h=m\fphi+\zphi$,
where $\fphi$ and $\zphi$ are defined in Section~\ref{sec:basis}. 
\end{itemize}
\end{proposition}
\begin{corollary}\label{cor:Nikulinpol}
Let $h\in \SX$ be a nef class with $n:=\intf{h, h}>0$.
Then $h$ is a polarization of degree $n$ if and only if 
the set 
$$
\FFF_{h}:=\set{f\in \SX}{\intf{f, h}=1,\; \intf{f, f}=0}
$$
is empty.
\end{corollary}
\begin{proof}
If the case (ii) of Proposition~\ref{prop:Nikulinpol} holds,
then the class $\fphi$ of $E$ is an element of $\FFF_h$.
Suppose that the case (i) of Proposition~\ref{prop:Nikulinpol} holds
and that $\FFF_h$ contains an element $f$.
Then $\dim |\LLL_f|>0$ and the movable part of  $|\LLL_f|$ contains a curve  that is mapped to a line in $\P^{1+n/2}$ by $\Phi_h$
isomorphically, which is absurd.
\end{proof}
\begin{remark}
Since $\intf{h, h}>0$, we can calculate $\FFF_{h}$ by Algorithm~\ref{algo:affES}.
\end{remark}
\par
Suppose that a polarization $h\in \SX$ of degree $2$ is given,
and let $\dpinvol{h}\in \Aut(X)$ be the associated double-plane involution.
We can calculate the matrix  of the action of $\dpinvol{h}$ on $\SX$
by the following method,
provided that we have an ample class $a\in \SX$. 
Let
$$
\Phi_h\colon  X\maprightsp{\rho_h} X_h \maprightsp{\pi_h} \P^2
$$
be the Stein factorization of the morphism $\Phi_h$ 
induced by the complete linear system  $|\LLL_h|$,
and let $B_h$ be the branch curve of $\pi_h\colon X_h\to \P^2$,
which is a plane curve of degree $6$ with only simple singularities.
Recall that the dual graph of 
the set $\CCC_0(h)$
of classes of smooth rational curves contracted by  the minimal resolution of singularities $\rho_h$
is a disjoint union of indecomposable root systems of
type $A_l$, $D_m$ or $E_n$ in  Figure~\ref{fig:ADE}.
The action of $\dpinvol{h}$ on each  indecomposable root system $R$ is given as follows.
\begin{itemize}
\item  If $R$ is of  type $A_l$, then  $\dpinvol{h}$ maps 
$a_i$ to $a_{l+1-i}$.
\item    If $R$ is of  type  $D_{2k}$, then  $\dpinvol{h}$ acts on $R$ as the identity,
whereas if  $R$ is of type $D_{2k+1}$, then $\dpinvol{h}$ 
interchanges $d_1$ and $d_2$ and fixes  $d_3, \dots, d_{2k+1}$.
\item If  $R$ is of type $E_6$, then $\dpinvol{h}$  fixes $e_1$, $e_4$, and interchanges   $e_i$ and  $e_{8-i}$  for $i=2, 3$.
 If  $R$ is of type $E_7$ or $E_8$, then $\dpinvol{h}$ acts on $R$ as the identity.
\end{itemize}
The eigenspace $(\SX\tensor\Q)^+$
of the action of $\dpinvol{h}$ on $\SX\tensor\Q$ with the eigenvalue $1$
is generated over $\Q$ by the class $h$ and the classes in the set 
$$
\set{r+r^{\dpinvol{h}}}{r\in \CCC_0(h)},
$$
and the eigenspace $(\SX\tensor\Q)^-$ with the eigenvalue $-1$
is orthogonal to $(\SX\tensor\Q)^+$.
Therefore we can determine the action of $\dpinvol{h}$ on $\SX\tensor\Q$ and hence on $\SX$
from the set $\CCC_0(h)$.
\par
Conversely,
suppose that  the matrix $\varphi_X(\invol) \in \OG(\SX)$ of an involution 
$\invol\in \Aut(X)$ 
is given.
We search for a polarization $h$ of degree $2$ such that $\dpinvol{h}=\invol$.
Such a polarization does not necessarily exist.
If it exists,
however,
we can detect it by the following method,
with the help of an ample class  $a\in \SX$.
Let $d$ be a positive integer.
We calculate the finite set
$\shortset{v\in\SX}{\intf{v,v}=2, \intf{v, a}=d}$
by Algorithm~\ref{algo:affES},
and its subset
$$
\HHH_d:=\set{v\in\SX}{\intf{v,v}=2, \;\;\intf{v, a}=d, \;\;v^{\invol}=v}.
$$
For each $h\in \HHH_d$,
we see whether $h$ is nef or not by Corollary~\ref{cor:nefcriterion}.
If $h$ is nef,
then we see whether $h$ is a polarization of degree $2$ or not by Corollary~\ref{cor:Nikulinpol}.
If $h$ is a polarization of degree $2$,
then we calculate the matrix $\varphi_X(\dpinvol{h})$ by the method described above.
If $\varphi_X(\dpinvol{h})$ is equal to $\varphi_X(\invol) $,
then we have $\dpinvol{h}=\invol$.
(Recall that we have assumed that $\varphi_X$ is injective.)
We start from $d=1$ and repeat this process until we find the desired polarization $h$.
\begin{remark}\label{rem:notunique}
It often happens that
two different polarizations of degree $2$ yield
the same double-plane involution.
Let $h\in \SX$ be a polarization of degree $2$.
The morphism $\Phi_h\colon X\to\P^2$
factors as
$$ 
X  \maprightsp{q} F \maprightsp{\beta} \P^2,
$$
where $q$ is the quotient morphism
by $\dpinvol{h}$.
Then $F$ is a smooth rational surface and $\beta$ is a succession of blowing-downs of $(-1)$-curves.
There can exist a birational morphism $\beta\sprime\colon F\to \P^2$
other than $\beta$.
Let $h\sprime\in\SX$ be the class of the pull-back of a line on $\P^2$ by $\beta\sprime\circ q$.
Then $h\sprime$ is a polarization of degree $2$ with $\dpinvol{h}=\dpinvol{h\sprime}$.
See Section~\ref{subsec:2A9} for a concrete  example.
\end{remark}
\subsection{Splitting lines}
\begin{definition}\label{def:splittingXh}
Let $(X, h)$ be a polarized $K3$ surface  of degree $2$.
A line $\ell$ on $\P^2$ is  a \emph{splitting line for $(X, h)$}
if the strict transform of $\ell$ by $\Phi_h$ has two irreducible components.
\par
Let $B$ be a reduced projective plane curve of degree $6$.
A line $\ell$  is a \emph{splitting line for $B$}
if $\ell$ is not an irreducible component of $B$
and the intersection multiplicity of $\ell$ and $B$
at each intersection point
is even.
\end{definition}
By definition, 
a line $\ell$ is splitting for $(X, h)$
if and only if $\ell$ is splitting for the branch curve $B_h$ of $\pi_h\colon X_h\to \P^2$.
Let $\Gamma$ be a smooth rational curve on $X$ such that
$[\Gamma]\in \CCC_1(h)$.
If $[\Gamma]^{\dpinvol{h}}=[\Gamma]$,
then $\Phi_h$ maps $\Gamma$ to a line component  of $B_h$  isomorphically.
If $[\Gamma]^{\dpinvol{h}}\ne [\Gamma]$,
then $\Phi_h$ maps $\Gamma$ to a splitting line for $B_h$  isomorphically.
%
%
%
\section{Proof of Theorem~\ref{thm:aut}, Proposition~\ref{prop:InvolsD0} and Table~\ref{table:main}}\label{sec:proof}
In the proof of Theorems~\ref{thm:D0} and~\ref{thm:orbits} in Section~\ref{sec:Borcherds},
we have already calculated,
in the form of matrices, 
 all the elements of the finite group $\Aut(\Xk, \ak)$,
the set $\Involsb{k}\spar{0}$ of involutions 
in $\Aut(\Xk, \ak)$,
and  the set $\Involsb{k}\spar{i}$ of  involutions 
that 
map the induced chamber $\Dzero$ to  the adjacent induced chamber $D\spar{i}$ for $i>0$.
By the method described in Section~\ref{sec:involutions},
we determine the types of the involutions in  $\Involsb{k}\spar{i}$.
Thus we prove Proposition~\ref{prop:InvolsD0}
and complete Table~\ref{table:main}.
\par
%
%
%
%
%
We prove the assertions on $\Xsb{0}$ in Theorem~\ref{thm:aut}.  
The cardinalities of the  conjugacy classes of $\Aut(\Xsb{0}, \asb{0})$ are as follows:
{\small
$$
\begin{array}{c|cccccccccccccccc}
\textrm{order} & 1& 2& 2& 2& 3& 4& 4& 4& 4& 5& 6& 8& 8& 8& 8& 10\phantom{.}  \\
 \hline
\textrm{card.} & 1& 45& 45& 1& 80& 180& 180& 90& 90& 144& 80& 90& 90& 90& 90& 144.
\end{array}
$$}%
The center of $\Aut(\Xsb{0}, \asb{0})$ is therefore 
a cyclic group of order $2$ generated by $\Enrinvol$
given in Table~\ref{table:Enrinvol}. 
By Proposition~\ref{prop:EnriquesInvol},
we see that 
$\Enrinvol$ is an Enriques involution.
One of the two conjugacy classes of order $2$ with cardinality $45$
consists of symplectic involutions,
and the other consists of rational involutions.
The matrix $\pnsymp$ in Table~\ref{table:pnsymp} is an element of $\Aut(\Xsb{0}, \asb{0})$
with order $4$.
Since $\eta_S(\pnsymp)\in \OG(q_{\Ssb{0}})$ is of order $4$,
we see that $\pnsymp$ is purely non-symplectic.
There exist three double-plane involutions
$\dpinvol{\hki{0}{1}}$, $\dpinvol{\hki{0}{2}}$, $\dpinvol{\hki{0}{3}}$ in $\Aut(\Xsb{0}, \asb{0})$, 
where the polarizations $\hki{0}{i}$   of degree $2$ are given in Table~\ref{table:hs0},
such that  $\dpinvol{\hki{0}{1}}$, $\dpinvol{\hki{0}{2}}$, $\dpinvol{\hki{0}{3}}$ and $\pnsymp$
generate $\Aut(\Xsb{0}, \asb{0})$.
The subgroup
$$
Aut(\Xsb{0}, \asb{0})\sprime:=\gen{\dpinvol{\hki{0}{1}}, \dpinvol{\hki{0}{2}}, \dpinvol{\hki{0}{3}}}
$$
of  $\Aut(\Xsb{0}, \asb{0})$ is of index $2$ and consists of 
elements $g\in \Aut(\Xsb{0}, \asb{0})$
with $\lambda_{\Xsb{0}} (g)^2=1$.
The mapping 
$$
\dpinvol{\hki{0}{1}}\mapsto ((12)(34), -1),
\;\;
\dpinvol{\hki{0}{2}}\mapsto ((35)(46), -1),
\;\;
\dpinvol{\hki{0}{3}}\mapsto  ((23)(56), -1))
$$
induces an isomorphism
 from $\Aut(\Xsb{0}, \asb{0})\sprime$ to $\AAAA_6\times \{\pm 1\}$.
By this isomorphism,
the Enriques involution $\Enrinvol$
is mapped to $(\id, -1)$,
and the symplectic subgroup of $\Aut(\Xsb{0}, \asb{0})\sprime$
is mapped to $\AAAA_6 \times \{1\}$.
For $i=1, \dots, 12$,
the set $\Involsb{0}\spar{i}$ contains 
 a double-plane involution
$\dpinvol{\tlhki{0}{i}}$, 
where the polarization $\tlhki{0}{i}$   of degree $2$ is  given in Table~\ref{table:hs0}.
\begin{table}
{\tiny
\setlength{\arraycolsep}{1pt}
$$
\left[\begin{array}{cccccccccccccccccccc}
405  & 200  & -35  & -57  & -451  & -303  & -606  & -902  & -736  & -560  & -377  & -194  & -451  & -303  & -606  & -902  & -736  & -560  & -377  & -194 \\
10  & 5  & -1  & -1  & -12  & -8  & -16  & -24  & -20  & -15  & -10  & -5  & -12  & -8  & -16  & -24  & -20  & -15  & -10  & -5 \\
426  & 210  & -37  & -60  & -474  & -318  & -636  & -948  & -774  & -588  & -396  & -204  & -474  & -318  & -636  & -948  & -774  & -588  & -396  & -204 \\
690  & 342  & -60  & -97  & -768  & -516  & -1032  & -1536  & -1254  & -954  & -642  & -330  & -768  & -516  & -1032  & -1536  & -1254  & -954  & -642  & -330 \\
0  & 0  & 0  & 0  & 0  & 0  & 0  & 0  & 0  & 0  & 0  & 0  & 0  & 0  & 0  & 0  & 0  & 0  & 1  & 0 \\
0  & 0  & 0  & 0  & 0  & 0  & 0  & 0  & 0  & 0  & 0  & 0  & 0  & 1  & 0  & 0  & 0  & 0  & 0  & 0 \\
14  & 7  & -1  & -2  & -15  & -10  & -20  & -30  & -24  & -18  & -12  & -6  & -17  & -12  & -23  & -34  & -28  & -21  & -14  & -7 \\
22  & 11  & -2  & -3  & -25  & -17  & -34  & -50  & -41  & -31  & -21  & -11  & -24  & -16  & -32  & -48  & -39  & -30  & -21  & -11 \\
21  & 10  & -2  & -3  & -24  & -16  & -32  & -48  & -39  & -30  & -20  & -10  & -22  & -15  & -30  & -44  & -36  & -27  & -18  & -9 \\
14  & 7  & -1  & -2  & -15  & -10  & -20  & -30  & -25  & -19  & -13  & -7  & -17  & -11  & -22  & -33  & -27  & -21  & -14  & -7 \\
0  & 0  & 0  & 0  & 0  & 0  & 0  & 0  & 0  & 0  & 0  & 0  & 1  & 0  & 0  & 0  & 0  & 0  & 0  & 0 \\
22  & 11  & -2  & -3  & -24  & -16  & -32  & -48  & -39  & -30  & -20  & -10  & -26  & -17  & -34  & -51  & -41  & -31  & -21  & -11 \\
0  & 0  & 0  & 0  & 0  & 0  & 0  & 0  & 0  & 0  & 1  & 0  & 0  & 0  & 0  & 0  & 0  & 0  & 0  & 0 \\
0  & 0  & 0  & 0  & 0  & 1  & 0  & 0  & 0  & 0  & 0  & 0  & 0  & 0  & 0  & 0  & 0  & 0  & 0  & 0 \\
14  & 7  & -1  & -2  & -17  & -12  & -23  & -34  & -28  & -21  & -14  & -7  & -15  & -10  & -20  & -30  & -24  & -18  & -12  & -6 \\
22  & 11  & -2  & -3  & -24  & -16  & -32  & -48  & -39  & -30  & -21  & -11  & -25  & -17  & -34  & -50  & -41  & -31  & -21  & -11 \\
21  & 10  & -2  & -3  & -22  & -15  & -30  & -44  & -36  & -27  & -18  & -9  & -24  & -16  & -32  & -48  & -39  & -30  & -20  & -10 \\
14  & 7  & -1  & -2  & -17  & -11  & -22  & -33  & -27  & -21  & -14  & -7  & -15  & -10  & -20  & -30  & -25  & -19  & -13  & -7 \\
0  & 0  & 0  & 0  & 1  & 0  & 0  & 0  & 0  & 0  & 0  & 0  & 0  & 0  & 0  & 0  & 0  & 0  & 0  & 0 \\
22  & 11  & -2  & -3  & -26  & -17  & -34  & -51  & -41  & -31  & -21  & -11  & -24  & -16  & -32  & -48  & -39  & -30  & -20  & -10 \end{array}\right]
$$
}
\caption{The Enriques involution $\Enrinvol$}\label{table:Enrinvol}
\end{table}
\begin{table}
{\tiny
\setlength{\arraycolsep}{1.5pt}
$$
\left[\begin{array}{cccccccccccccccccccc}
318  & 159  & -29  & -46  & -342  & -236  & -465  & -684  & -560  & -430  & -289  & -148  & -350  & -237  & -474  & -700  & -569  & -438  & -296  & -148 \\
0  & 0  & 0  & 0  & 0  & 0  & 0  & 1  & 0  & 0  & 0  & 0  & 0  & 0  & 0  & 0  & 0  & 0  & 0  & 0 \\
324  & 162  & -30  & -47  & -348  & -240  & -474  & -696  & -570  & -438  & -294  & -150  & -354  & -240  & -480  & -708  & -576  & -444  & -300  & -150 \\
540  & 270  & -49  & -78  & -582  & -402  & -792  & -1164  & -954  & -732  & -492  & -252  & -594  & -402  & -804  & -1188  & -966  & -744  & -504  & -252 \\
21  & 10  & -2  & -3  & -22  & -15  & -30  & -44  & -36  & -27  & -18  & -9  & -24  & -16  & -32  & -48  & -39  & -30  & -20  & -10 \\
0  & 0  & 0  & 0  & 0  & 0  & 0  & 0  & 0  & 0  & 0  & 0  & 1  & 0  & 0  & 0  & 0  & 0  & 0  & 0 \\
21  & 10  & -2  & -3  & -22  & -15  & -30  & -44  & -36  & -28  & -19  & -10  & -24  & -16  & -32  & -47  & -38  & -29  & -20  & -10 \\
0  & 1  & 0  & 0  & 0  & 0  & 0  & 0  & 0  & 0  & 0  & 0  & 0  & 0  & 0  & 0  & 0  & 0  & 0  & 0 \\
13  & 6  & -1  & -2  & -14  & -10  & -19  & -28  & -23  & -18  & -12  & -6  & -14  & -10  & -19  & -28  & -23  & -18  & -12  & -6 \\
0  & 0  & 0  & 0  & 0  & 0  & 0  & 0  & 0  & 0  & 0  & 0  & 0  & 1  & 0  & 0  & 0  & 0  & 0  & 0 \\
22  & 11  & -2  & -3  & -25  & -17  & -34  & -50  & -41  & -31  & -21  & -11  & -24  & -17  & -33  & -48  & -39  & -30  & -20  & -10 \\
20  & 10  & -2  & -3  & -22  & -15  & -29  & -43  & -35  & -27  & -18  & -9  & -21  & -14  & -28  & -42  & -34  & -26  & -18  & -9 \\
8  & 4  & -1  & -1  & -8  & -6  & -11  & -16  & -13  & -10  & -7  & -4  & -9  & -6  & -12  & -18  & -15  & -12  & -8  & -4 \\
0  & 0  & 0  & 0  & 0  & 0  & 0  & 0  & 0  & 0  & 1  & 0  & 0  & 0  & 0  & 0  & 0  & 0  & 0  & 0 \\
22  & 11  & -2  & -3  & -25  & -17  & -34  & -50  & -41  & -32  & -22  & -11  & -24  & -16  & -32  & -48  & -39  & -30  & -20  & -10 \\
12  & 6  & -1  & -2  & -12  & -8  & -16  & -24  & -20  & -15  & -10  & -5  & -13  & -9  & -18  & -26  & -21  & -16  & -11  & -6 \\
0  & 0  & 0  & 0  & 0  & 0  & 0  & 0  & 0  & 0  & 0  & 0  & 0  & 0  & 0  & 0  & 0  & 0  & 0  & 1 \\
22  & 11  & -2  & -3  & -24  & -17  & -33  & -48  & -39  & -30  & -20  & -10  & -25  & -17  & -34  & -50  & -41  & -31  & -21  & -11 \\
0  & 0  & 0  & 0  & 0  & 0  & 0  & 0  & 0  & 0  & 0  & 0  & 0  & 0  & 0  & 0  & 1  & 0  & 0  & 0 \\
12  & 6  & -1  & -2  & -13  & -9  & -18  & -26  & -21  & -16  & -11  & -6  & -12  & -8  & -16  & -24  & -20  & -15  & -10  & -5 \end{array}\right]
$$
}
\caption{The purely non-symplectic automorphism  $\pnsymp$ of order $4$}\label{table:pnsymp}
\end{table}
\begin{table}
\setlength{\arraycolsep}{.1pt}
\newcommand{\tempstrut}{\mystrutd{6pt}}
{\tiny
$$
\begin{array}{ccccccccccccccccc}
h & & &&&&&&&&&&&& \phantom{a}& \Sing (X_h) & \intf{h, a_{0}}\\ 
\hline
&&&\\
\hki{0}{1} &=&(43,&21,&-4,&-6,&-47,&-32,&-63,&-93,&-76,&-58,&-39,&-20,&&2A_{2} + 7A_{1} &10 \\
&&    &   &   &   &-48,&-33,&-65,&-96,&-78,&-60,&-40,&-20)&&\tempstrut&   \\
\hki{0}{2} &=&(64,&32,&-6,&-9,&-71,&-47,&-94,&-140,&-114,&-87,&-59,&-30,&&2A_{2} + 7A_{1} &10 \\
&&    &   &   &   &-71,&-48,&-95,&-141,&-115,&-88,&-60,&-30)&&\tempstrut&   \\
\hki{0}{3} &=&(49,&24,&-4,&-7,&-56,&-38,&-75,&-111,&-90,&-69,&-47,&-24,&&2A_{2} + 7A_{1} &10 \\
&&    &   &   &   &-54,&-36,&-72,&-107,&-87,&-66,&-45,&-23)&&\tempstrut&   \\
&&&\\
\hline
&&&\\
\tlhki{0}{1} &=&(64,&32,&-6,&-9,&-71,&-48,&-95,&-140,&-114,&-87,&-59,&-30,&&A_{2} + 8A_{1} &10 \\
&&    &   &   &   &-71,&-48,&-95,&-141,&-115,&-87,&-59,&-30)&&\tempstrut&   \\
\tlhki{0}{2} &=&(57,&28,&-5,&-8,&-64,&-43,&-86,&-127,&-103,&-78,&-52,&-26,&&4A_{2} + 4A_{1} &12 \\
&&    &   &   &   &-64,&-43,&-85,&-127,&-103,&-78,&-53,&-27)&&\tempstrut&   \\
\tlhki{0}{3} &=&(64,&32,&-6,&-9,&-72,&-48,&-96,&-142,&-116,&-89,&-60,&-31,&&3A_{2} + 6A_{1} &12 \\
&&    &   &   &   &-69,&-47,&-93,&-138,&-113,&-86,&-59,&-30)&&\tempstrut&   \\
\tlhki{0}{4} &=&(74,&37,&-7,&-10,&-83,&-56,&-111,&-164,&-134,&-103,&-69,&-35,&&5A_{2} + 4A_{1} &14 \\
&&    &   &   &   &-82,&-55,&-110,&-164,&-134,&-103,&-70,&-36)&&\tempstrut&   \\
\tlhki{0}{5} &=&(80,&40,&-7,&-11,&-91,&-61,&-122,&-181,&-147,&-112,&-75,&-38,&&5A_{2} + 4A_{1} &14 \\
&&    &   &   &   &-89,&-60,&-119,&-178,&-145,&-110,&-75,&-38)&&\tempstrut&   \\
\tlhki{0}{6} &=&(176,&88,&-16,&-25,&-193,&-130,&-260,&-383,&-312,&-238,&-161,&-81,&&3A_{3} + 6A_{1} &22 \\
&&    &   &   &   &-197,&-134,&-264,&-391,&-318,&-243,&-165,&-84)&&\tempstrut&   \\
\tlhki{0}{7} &=&(140,&70,&-13,&-20,&-153,&-102,&-204,&-303,&-245,&-187,&-127,&-64,&&4A_{3} + 4A_{1} &22 \\
&&    &   &   &   &-155,&-105,&-209,&-310,&-254,&-194,&-131,&-67)&&\tempstrut&   \\
\tlhki{0}{8} &=&(152,&76,&-14,&-21,&-173,&-115,&-230,&-342,&-277,&-212,&-144,&-72,&&3A_{4} + A_{2} + A_{1} &24 \\
&&    &   &   &   &-167,&-113,&-222,&-331,&-270,&-208,&-142,&-73)&&\tempstrut&   \\
\tlhki{0}{9} &=&(252,&126,&-22,&-35,&-284,&-191,&-382,&-563,&-456,&-349,&-237,&-121,&&3A_{5} + 3A_{1} &34 \\
&&    &   &   &   &-280,&-191,&-378,&-560,&-457,&-350,&-238,&-121)&&\tempstrut&   \\
\tlhki{0}{10} &=&(148,&74,&-13,&-21,&-171,&-114,&-228,&-338,&-272,&-206,&-140,&-70,&&3A_{5} + 3A_{1} &34 \\
&&    &   &   &   &-160,&-108,&-212,&-316,&-260,&-199,&-134,&-69)&&\tempstrut&   \\
\tlhki{0}{11} &=&(304,&152,&-27,&-42,&-341,&-231,&-456,&-677,&-551,&-420,&-284,&-142,&&D_{4} + 2A_{5} + A_{3} &38 \\
&&    &   &   &   &-340,&-230,&-455,&-680,&-554,&-424,&-288,&-147)&&\tempstrut&   \\
\tlhki{0}{12} &=&(206,&103,&-19,&-29,&-231,&-156,&-312,&-457,&-371,&-285,&-193,&-97,&&D_{4} + 2A_{5} + A_{3} &38 \\
&&    &   &   &   &-224,&-153,&-300,&-447,&-365,&-278,&-191,&-98)&&\tempstrut&   \\
\end{array}
$$
}
\caption{The polarizations  $\hki{0}{i}$ and $\tlhki{0}{i}$ of degree $2$}\label{table:hs0}
\end{table}
%
%
%
%
%
\par
Next 
we prove the assertions on $\Xsb{1}$ and $\Xsb{2}$ in Theorem~\ref{thm:aut}. 
Suppose that $k=1$ or $2$.
Then the 
 cardinalities of the conjugacy classes of $\Aut(\Xsb{k}, \asb{k})$ are as follows:
 {\small
$$
\begin{array}{c|ccccccccccc}
\textrm{order} & 1& 2& 2& 3& 4& 5& 5& 8& 8& 10& 10\phantom{.} \\
 \hline
\textrm{card.} & 1& 45& 36& 80& 90& 72& 72& 90& 90& 72& 72.
\end{array}
$$}%
The conjugacy class of order $2$ with cardinality $45$ consists of  symplectic involutions,
and the class of order $2$ with cardinality $36$ 
consists of rational involutions.
There exist three double-plane involutions
$\dpinvol{\hki{k}{1}}$, $\dpinvol{\hki{k}{2}}$, $\dpinvol{\hki{k}{3}}$ in  $\Aut(\Xk, \ak)$, 
where the polarizations $\hki{k}{i}$   of degree $2$ are given in Tables~\ref{table:hs1} and~\ref{table:hs2}.
These three involutions  
generate $\Aut(\Xsb{k}, \asb{k})$, and 
the mapping
$$
\dpinvol{\hki{k}{1}}\mapsto \left[\begin{array}{cc} 0 & 1+\sqrt{2} \\ 1 & 0 \end{array}\right],
\;\; 
\dpinvol{\hki{k}{2}}\mapsto \left[\begin{array}{cc} 0 & 2+\sqrt{2} \\ 1 & 0 \end{array}\right],
\;\; 
\dpinvol{\hki{k}{3}}\mapsto \left[\begin{array}{cc} 2 &     \sqrt{2} \\ 1 & 1 \end{array}\right]
$$
induces an isomorphism from $\Aut(\Xsb{k}, \asb{k})$ to $\PGL_2(\F_9)$.
Except for the case $k=1$ and $i=4$, 
the set $\Involsb{k}\spar{i}$ contains 
 a double-plane involution
$\dpinvol{\tlhki{k}{i}}$, 
where the polarization $\tlhki{k}{i}$   of degree $2$ is  given in Tables~\ref{table:hs1} and~\ref{table:hs2}.
The set $\Involsb{1}\spar{4}$
consists of $6$ symplectic involutions,
one of which is 
the matrix $\specialsymp$ given in Table~\ref{table:specialsymp}.
\begin{table}
\setlength{\arraycolsep}{.1pt}
\newcommand{\tempstrut}{\mystrutd{6pt}}
{\tiny
$$
\begin{array}{ccccccccccccccccc}
h & & &&&&&&&&&&&& \phantom{a}& \Sing (X_h) & \intf{h, a_{1}}\\ 
\hline
&&&\\
\hki{1}{1} &=&(30,&15,&-7,&-2,&-33,&-22,&-44,&-66,&-54,&-41,&-28,&-14,&&4A_{2} + 5A_{1} &12 \\
&&    &   &   &   &-34,&-23,&-45,&-67,&-55,&-42,&-28,&-14)&&\tempstrut&   \\
\hki{1}{2} &=&(30,&15,&-7,&-2,&-34,&-23,&-45,&-67,&-55,&-42,&-28,&-14,&&4A_{2} + 5A_{1} &12 \\
&&    &   &   &   &-33,&-22,&-44,&-66,&-54,&-41,&-28,&-14)&&\tempstrut&   \\
\hki{1}{3} &=&(43,&21,&-10,&-3,&-46,&-31,&-62,&-92,&-75,&-57,&-39,&-20,&&4A_{2} + 5A_{1} &12 \\
&&    &   &   &   &-49,&-33,&-66,&-98,&-80,&-61,&-41,&-21)&&\tempstrut&   \\
&&&\\
\hline
&&&\\
\tlhki{1}{1} &=&(45,&22,&-11,&-3,&-50,&-34,&-67,&-99,&-81,&-62,&-42,&-21,&&3A_{2} + 6A_{1} &12 \\
&&    &   &   &   &-49,&-33,&-65,&-97,&-79,&-61,&-42,&-21)&&\tempstrut&   \\
\tlhki{1}{2} &=&(43,&21,&-10,&-3,&-48,&-33,&-65,&-96,&-79,&-60,&-40,&-20,&&A_{3} + 4A_{2} + 2A_{1} &14 \\
&&    &   &   &   &-47,&-32,&-63,&-93,&-76,&-58,&-40,&-20)&&\tempstrut&   \\
\tlhki{1}{3} &=&(46,&23,&-11,&-3,&-50,&-34,&-68,&-100,&-81,&-62,&-42,&-21,&&5A_{2} + 4A_{1} &14 \\
&&    &   &   &   &-52,&-36,&-70,&-103,&-84,&-64,&-44,&-22)&&\tempstrut&   \\
&&&&&&&&&&&&&&&&\\
\tlhki{1}{5} &=&(46,&23,&-11,&-3,&-52,&-36,&-70,&-104,&-85,&-65,&-44,&-23,&&2A_{3} + 3A_{2} + 2A_{1} &16 \\
&&    &   &   &   &-49,&-34,&-67,&-98,&-80,&-62,&-42,&-21)&&\tempstrut&   \\
\tlhki{1}{6} &=&(76,&38,&-18,&-5,&-84,&-57,&-112,&-167,&-136,&-103,&-70,&-35,&&3A_{3} + 3A_{2} &18 \\
&&    &   &   &   &-86,&-59,&-116,&-170,&-138,&-106,&-72,&-36)&&\tempstrut&   \\
\tlhki{1}{7} &=&(106,&53,&-25,&-7,&-119,&-81,&-159,&-235,&-192,&-146,&-99,&-50,&&2A_{4} + 2A_{3} + A_{2} &22 \\
&&    &   &   &   &-117,&-81,&-159,&-234,&-192,&-147,&-99,&-51)&&\tempstrut&   \\
\tlhki{1}{8} &=&(94,&47,&-22,&-6,&-104,&-71,&-140,&-208,&-169,&-130,&-88,&-44,&&2A_{4} + 2A_{3} + A_{2} &22 \\
&&    &   &   &   &-106,&-73,&-143,&-211,&-173,&-132,&-91,&-47)&&\tempstrut&   \\
\tlhki{1}{9} &=&(110,&55,&-26,&-8,&-121,&-84,&-164,&-241,&-197,&-150,&-102,&-51,&&2A_{5} + 2A_{3} &30 \\
&&    &   &   &   &-120,&-80,&-160,&-237,&-193,&-149,&-102,&-51)&&\tempstrut&   \\
\tlhki{1}{10} &=&(124,&62,&-29,&-8,&-138,&-95,&-186,&-276,&-225,&-171,&-116,&-58,&&2A_{5} + 2A_{3} &30 \\
&&    &   &   &   &-139,&-95,&-190,&-278,&-227,&-172,&-117,&-59)&&\tempstrut&   \\
\tlhki{1}{11} &=&(217,&108,&-51,&-15,&-239,&-166,&-325,&-477,&-390,&-296,&-202,&-101,&&2A_{9} &54 \\
&&    &   &   &   &-239,&-166,&-325,&-477,&-390,&-296,&-202,&-101)&&\tempstrut&   \\
\tlhki{1}{12} &=&(250,&125,&-59,&-17,&-277,&-185,&-370,&-548,&-449,&-343,&-231,&-119,&&2A_{9} &54 \\
&&    &   &   &   &-276,&-191,&-375,&-552,&-453,&-348,&-236,&-118)&&\tempstrut&   \\
\end{array}
$$
}
\caption{The polarizations   $\hki{1}{i}$ and $\tlhki{1}{i}$ of degree $2$}\label{table:hs1}
\end{table}

\begin{table}
\setlength{\arraycolsep}{.1pt}
\newcommand{\tempstrut}{\mystrutd{6pt}}
{\tiny
$$
\begin{array}{ccccccccccccccccc}
h & & &&&&&&&&&&&& \phantom{a}& \Sing (X_h) & \intf{h, a_{2}}\\ 
\hline
&&&\\
\hki{2}{1} &=&(37,&18,&-7,&-3,&-41,&-28,&-55,&-82,&-67,&-51,&-35,&-18,&&5A_{2} + 5A_{1} &7 \\
&&    &   &   &   &-42,&-29,&-57,&-84,&-68,&-52,&-35,&-18)&&\tempstrut&   \\
\hki{2}{2} &=&(25,&12,&-5,&-2,&-27,&-18,&-36,&-54,&-44,&-34,&-23,&-12,&&5A_{2} + 5A_{1} &7 \\
&&    &   &   &   &-28,&-19,&-38,&-56,&-46,&-35,&-24,&-12)&&\tempstrut&   \\
\hki{2}{3} &=&(36,&18,&-7,&-3,&-40,&-27,&-54,&-80,&-65,&-50,&-34,&-17,&&5A_{2} + 5A_{1} &7 \\
&&    &   &   &   &-40,&-27,&-54,&-80,&-65,&-49,&-33,&-17)&&\tempstrut&   \\
&&&\\
\hline
&&&\\
\tlhki{2}{1} &=&(24,&12,&-5,&-2,&-25,&-17,&-33,&-49,&-40,&-31,&-21,&-11,&&2A_{3} + 3A_{2} + 2A_{1} &8 \\
&&    &   &   &   &-27,&-18,&-36,&-54,&-44,&-34,&-23,&-12)&&\tempstrut&   \\
\tlhki{2}{2} &=&(34,&17,&-7,&-3,&-37,&-25,&-49,&-73,&-60,&-46,&-32,&-16,&&3A_{3} + 3A_{2} &9 \\
&&    &   &   &   &-36,&-24,&-48,&-72,&-59,&-45,&-31,&-16)&&\tempstrut&   \\
\tlhki{2}{3} &=&(65,&32,&-12,&-6,&-70,&-48,&-94,&-140,&-114,&-87,&-60,&-30,&&3A_{4} + A_{2} + A_{1} &12 \\
&&    &   &   &   &-73,&-49,&-97,&-145,&-118,&-91,&-62,&-32)&&\tempstrut&   \\
\tlhki{2}{4} &=&(56,&28,&-11,&-5,&-61,&-41,&-81,&-121,&-98,&-74,&-50,&-25,&&A_{5} + 2A_{4} + A_{3} &13 \\
&&    &   &   &   &-61,&-41,&-82,&-122,&-101,&-77,&-53,&-28)&&\tempstrut&   \\
\tlhki{2}{5} &=&(44,&22,&-9,&-4,&-46,&-31,&-61,&-91,&-75,&-57,&-38,&-19,&&A_{5} + 2A_{4} + A_{3} &13 \\
&&    &   &   &   &-49,&-32,&-64,&-96,&-78,&-60,&-42,&-21)&&\tempstrut&   \\
\tlhki{2}{6} &=&(126,&63,&-26,&-11,&-138,&-95,&-185,&-275,&-222,&-169,&-116,&-58,&&2A_{9} &27 \\
&&    &   &   &   &-136,&-89,&-178,&-267,&-220,&-168,&-116,&-58)&&\tempstrut&   \\
\tlhki{2}{7} &=&(145,&70,&-28,&-13,&-160,&-110,&-215,&-320,&-260,&-200,&-135,&-70,&&2A_{9} &27 \\
&&    &   &   &   &-160,&-105,&-210,&-315,&-255,&-195,&-130,&-65)&&\tempstrut&   \\
\end{array}
$$
}
\caption{The polarizations   $\hki{2}{i}$ and $\tlhki{2}{i}$ of degree $2$}\label{table:hs2}
\end{table}
\begin{table}
{\tiny
\setlength{\arraycolsep}{.6pt}
$$
\left[\begin{array}{cccccccccccccccccccc}
624  & 312  & -145  & -43  & -720  & -495  & -968  & -1440  & -1164  & -888  & -612  & -306  & -658  & -456  & -890  & -1316  & -1084  & -824  & -564  & -282 \\
0  & 0  & 0  & 0  & 0  & 0  & 1  & 0  & 0  & 0  & 0  & 0  & 0  & 0  & 0  & 0  & 0  & 0  & 0  & 0 \\
580  & 290  & -135  & -40  & -668  & -460  & -898  & -1336  & -1080  & -824  & -568  & -284  & -612  & -424  & -828  & -1224  & -1008  & -766  & -524  & -262 \\
1032  & 516  & -240  & -71  & -1188  & -816  & -1596  & -2376  & -1920  & -1464  & -1008  & -504  & -1092  & -756  & -1476  & -2184  & -1800  & -1368  & -936  & -468 \\
0  & 0  & 0  & 0  & 1  & 0  & 0  & 0  & 0  & 0  & 0  & 0  & 0  & 0  & 0  & 0  & 0  & 0  & 0  & 0 \\
45  & 22  & -11  & -3  & -52  & -36  & -70  & -104  & -84  & -64  & -44  & -22  & -46  & -32  & -62  & -92  & -76  & -58  & -40  & -20 \\
0  & 1  & 0  & 0  & 0  & 0  & 0  & 0  & 0  & 0  & 0  & 0  & 0  & 0  & 0  & 0  & 0  & 0  & 0  & 0 \\
57  & 28  & -13  & -4  & -66  & -45  & -88  & -131  & -106  & -81  & -56  & -28  & -61  & -42  & -82  & -122  & -100  & -76  & -52  & -26 \\
0  & 0  & 0  & 0  & 0  & 0  & 0  & 0  & 0  & 0  & 0  & 0  & 0  & 0  & 0  & 1  & 0  & 0  & 0  & 0 \\
0  & 0  & 0  & 0  & 0  & 0  & 0  & 0  & 0  & 0  & 0  & 0  & 1  & 0  & 0  & 0  & 0  & 0  & 0  & 0 \\
60  & 30  & -14  & -4  & -70  & -48  & -94  & -140  & -113  & -86  & -59  & -30  & -64  & -44  & -86  & -127  & -104  & -79  & -54  & -27 \\
0  & 0  & 0  & 0  & 0  & 0  & 0  & 0  & 0  & 0  & 0  & 1  & 0  & 0  & 0  & 0  & 0  & 0  & 0  & 0 \\
0  & 0  & 0  & 0  & 0  & 0  & 0  & 0  & 0  & 1  & 0  & 0  & 0  & 0  & 0  & 0  & 0  & 0  & 0  & 0 \\
44  & 22  & -10  & -3  & -52  & -36  & -70  & -104  & -84  & -64  & -44  & -22  & -46  & -31  & -62  & -92  & -76  & -58  & -40  & -20 \\
16  & 8  & -4  & -1  & -18  & -12  & -24  & -36  & -30  & -23  & -16  & -8  & -17  & -12  & -23  & -34  & -28  & -21  & -14  & -7 \\
0  & 0  & 0  & 0  & 0  & 0  & 0  & 0  & 1  & 0  & 0  & 0  & 0  & 0  & 0  & 0  & 0  & 0  & 0  & 0 \\
56  & 28  & -13  & -4  & -64  & -44  & -86  & -128  & -104  & -79  & -54  & -27  & -59  & -41  & -80  & -118  & -97  & -74  & -51  & -26 \\
0  & 0  & 0  & 0  & 0  & 0  & 0  & 0  & 0  & 0  & 0  & 0  & 0  & 0  & 0  & 0  & 0  & 0  & 0  & 1 \\
44  & 22  & -10  & -3  & -52  & -36  & -70  & -104  & -84  & -64  & -44  & -22  & -46  & -32  & -62  & -92  & -76  & -58  & -39  & -20 \\
0  & 0  & 0  & 0  & 0  & 0  & 0  & 0  & 0  & 0  & 0  & 0  & 0  & 0  & 0  & 0  & 0  & 1  & 0  & 0 \end{array}\right]
$$
}
\caption{The symplectic involution $\specialsymp$}\label{table:specialsymp}
\end{table}
\begin{remark}\label{rem:isomPGL2F9}
According to~\cite{MR827219}, 
there exist exactly three non-splitting  extensions of the cyclic group of order $2$ by $\AAAA_6$;
namely, the symmetric group $\SSSS_6$, the Mathieu group $M_{10}$, and the projective general linear group $\PGL_{2}(\F_9)$.
In~\cite[Chapter 10, Section 1.5]{MR1662447},
these three groups 
are distinguished by 
the numbers of conjugacy classes of elements of order $3$ and $5$:
$\SSSS_6$ has two classes of order $3$ and one of order $5$,
 $M_{10}$  has one of each,
and $\PGL_{2}(\F_9)$  has one of order $3$ and two of order $5$.
\end{remark}

\section{Examples}\label{sec:examples}
In this section, we investigate projective geometry of some of the automorphisms 
that appear in Theorem~\ref{thm:aut}.
\subsection{The purely non-symplectic automorphism $\pnsymp$}\label{subsec:pnsmp}
We investigate  the purely non-symplectic automorphism $\pnsymp$  of order $4$ in $\Aut(\Xsb{0}, a\sb{0})$.
The vector
$$
\setlength{\arraycolsep}{2pt}
\begin{array}{cccccccccccccc}
h_{\rho} &\;:=\;&( 88 , & 43 , & -8 , & -12 , & -98 , & -66 , & -131 , & -195 , & -159 , & -121 , & -82 , & -42, \\
& &&&&& -99 , & -67 , & -133 , & -197 , & -161 , & -123 , & -84 , & -43)
\end{array}
$$
of $S_0$ with $\intf{\hrho, \hrho}=4$ is invariant under the action of  $\pnsymp$.
By Corollary~\ref{cor:nefcriterion}, we see that $\hrho$ is nef, and 
by Corollary~\ref{cor:Nikulinpol}, we see that $\hrho$ is a polarization of degree $4$.
Moreover, by Algorithm~\ref{algo:affES}, we have
$$
\set{v\in \Ssb{0}}{\intf{v, v}=0, \;\; \intf{\hrho, v}=2\;}=\emptyset.
$$ 
Hence, by Theorem 5.2 of Saint-Donat~\cite{MR0364263},
the polarization $\hrho$ is not hyperelliptic;
that is, $\hrho$ is  the class of the pull-back of a hyperplane section by
 a birational morphism
 from $\Xsb{0}$ to a normal quartic surface $Y\subset \P^3$ given by $|\LLL_{h_{\rho}}|$.
 Since $\hrho$ is invariant under the action of  $\pnsymp$,
 we conclude  that $\pnsymp$ is induced 
by a projective linear automorphism  of $\P^3$ that leaves $Y$ invariant.
By a direct calculation, 
we see that the line configuration of $(\Xsb{0}, \hrho)$ is full, 
and hence, up to finite possibilities,
the projective equivalence class of the quartic surface $Y$
is determined by the line configuration of $(\Xsb{0}, \hrho)$.
We describe  this line configuration in details,
hoping that we can obtain a defining equation of  $Y$ in future.
Let $S$ be a set on which the group $\gen{\pnsymp}$ of order $4$
acts transitively.
By $S=[s_0, s_1, s_2, s_3]$,
we mean 
that $|S|=4$ and that $\pnsymp$ maps $s_i$ to $s_{i+1}$ for $i=0,1,2$ and $s_3$ to $s_0$,
and by $S=[s_0, s_1]$,
we mean that   $|S|=2$ and that $\pnsymp$ interchanges $s_0$ and $s_1$.
We denote by $\cyc(a,b,c,d)$ the cyclic matrix
$$
\left[\begin{array}{cccc}
a & b & c & d \\
d & a & b & c\\
c & d & a & b\\
b & c & d & a
\end{array}
\right].
$$
From the set $\CCC_0(\hrho)$,
we see that $\Sing (Y)$ consists of $6$ ordinary nodes,
and the group $\gen{\pnsymp}$ decomposes $\Sing(Y)$ into two orbits
$[p_0, p_1, p_2, p_3]$ and  $[q_0, q_1]$.
From the set $\CCC_1(\hrho)$,
we see that $Y$ contains exactly $36$ lines,  and they are decomposed 
into $9$ orbits
$$
l_i:=[\ell_i, \ell_i\sprime, \ell_i\spprime, \ell_i\sppprime]\quad (i=0, \dots, 8)
$$
of length $4$ by $\gen{\pnsymp}$.
We can  choose the element $\ell_i\in l_i$  in such a way that 
$$
\Sing (Y) \cap \ell_i=
\begin{cases}
\emptyset & \textrm{if $i=0, 1$,} \\
\{ q_0\}& \textrm{if $i=2, 3$,} \\
\{ p_0\} & \textrm{if $i=4, 5, 6, 7$,} \\
\{ p_0, q_1\}& \textrm{if $i=8$.} 
\end{cases}
$$
The intersection pattern  of lines in  the orbits $l_i$
and $l_j$ is given by the cyclic matrix
$$
M_{ij}=\cyc(\intf{\tl{\ell}_i, \tl{\ell}_j},    \intf{\tl{\ell}_i, \tl{\ell}\sprime_j},     \intf{\tl{\ell}_i, \tl{\ell}\spprime_j},     \intf{\tl{\ell}_i, \tl{\ell}\sppprime_j}),
$$
where $\tl{\ell}\subset \Xsb{0}$ is the strict transform of a line $\ell\subset Y$.
We have
%
%
$$
M_{ii}=
\begin{cases}
\cyc(-2,0,1,0)  & \textrm{if $i=0, 1, 4,5,6,7$,} \\
\cyc(-2,1,0,1) & \textrm{if $i=2$,} \\
\cyc(-2,0,0,0) & \textrm{if $i=3, 8$.} 
\end{cases}
$$
We put
%
%
$$
\begin{array}{ccc}
C_1 :=\cyc(0,0,1,0),  & 
C_2 :=\cyc(0,0,0,1),  & 
C_3:=\cyc(1,0,0,0),   \\
C_4:=\cyc(0,0,0,0),  & 
C_5:=\cyc(1,0,0,1),  & 
C_6:=\cyc(0,1,1,0),   \\
C_7:=\cyc(1,1,0,0),  & 
C_8:=\cyc(0,1,0,0),  & 
C_9:=\cyc(0,0,1,1).
\end{array}
$$
Then the matrices $M_{ij}$ for $i\ne j$ are given  in Table~\ref{table:Mij}.
\begin{table}
$$
\begin{array}{cccccccccc}
 i \backslash j &0 & 1 & 2 & 3 & 4 & 5 & 6 & 7 & 8  \\ 
0 &        &C_{1}  &C_{2}  &C_{3}  &C_{4}  &C_{5}  &C_{4}  &C_{6}  &C_{7}  \\ 
1 &        &        &C_{1}  &C_{8}  &C_{5}  &C_{4}  &C_{6}  &C_{4}  &C_{7}  \\ 
2 &        &        &        &C_{4}  &C_{3}  &C_{8}  &C_{2}  &C_{3}  &C_{4}  \\ 
3 &        &        &        &        &C_{5}  &C_{4}  &C_{4}  &C_{7}  &C_{1}  \\ 
4 &        &        &        &        &        &C_{9}  &C_{4}  &C_{2}  &C_{4}  \\ 
5 &        &        &        &        &        &        &C_{8}  &C_{4}  &C_{2}  \\ 
6 &        &        &        &        &        &        &        &C_{9}  &C_{8}  \\ 
7 &        &        &        &        &        &        &        &        &C_{4}  \\
\end{array}
$$
\vskip 7pt
\caption{The intersection of lines on the quartic surface model $Y$ of $\Xsb{0}$}\label{table:Mij}
\end{table}
%
%
%
\subsection{The double-plane involutions  $\dpinvol{\hki{0}{i}}$}\label{subsec:dpinvol01}
The three double-plane involutions $\dpinvol{\hki{0}{1}}$, $\dpinvol{\hki{0}{2}}$, $\dpinvol{\hki{0}{3}}$ of $\Xsb{0}$
are conjugate in $\Aut(\Xsb{0}, \asb{0})$.
Hence there exist a sextic double plane $Y\to \P^2$
and three isomorphisms
$\alpha\sp{[i]}\colon \Xsb{0}\isom \wt{Y}$
for $i=1,2,3$ 
such that $\dpinvol{\hki{0}{i}}= (\alpha\sp{[i]})\inv\circ \tau_{Y}\circ \alpha\sp{[i]}$ holds for $i=1,2,3$,
where $\wt{Y}$ is the minimal resolution of singularities of $Y$ and $\tau_{Y}$ is the involution of $\wt{Y}$ induced by $\Gal(Y/\P^2)$.
By a direct calculation, 
we see that 
the line configuration of $(\Xsb{0}, \hzi)$ is full, 
and hence, up to finite possibilities,
the projective equivalence class of the sextic double plane $Y\to \P^2$
is determined by the line configuration of $(\Xsb{0}, \hzi)$.
Let $B\subset \P^2$ denote the branch curve of  $Y\to \P^2$.
From $\CCC_0(\hzi)$, 
we see that
$\Sing(B)$ consists of two ordinary cusps $q_0, q_1$ and seven  ordinary nodes $n_0, \dots, n_6$.
The set $\CCC_1(\hzi)$ consists of $38$ elements, and 
the action of  $\gen{\dpinvol{\hki{0}{i}}}$ decomposes $\CCC_1(\hzi)$  into the union of $19$ orbits of length $2$.
Hence $B$ does not contain a line as an irreducible component.
Therefore 
$B$ is irreducible, and 
$B$ has $19$ splitting lines.
From the intersection pairing between $\CCC_0(\hzi)$ and $\CCC_1(\hzi)$, 
we see that, under suitable numbering of ordinary nodes $n_0, \dots, n_6$,  these  splitting lines are
$$
\ell_{00},\; \dots, \;  \ell_{06}, \;  \ell_{10}, \;  \dots, \;  \ell_{16}, \;  m_{012},  \;  m_{034}, \;  m_{056}, \;  m_{135},  \; m_{246}, 
$$
where
$\Sing (B)\cap \ell_{ij}=\{q_i, n_j\}$ and 
$\Sing (B)\cap m_{ijk}=\{n_i, n_j, n_k\}$.
%
%
\subsection{The double-plane involution  $\dpinvol{\tlhzone}$}\label{subsec:dpinvoltl01}
Next we examine the double-plane involution $\dpinvol{\tlhzone}$ of $\Xsb{0}$
that maps the induced chamber $\Dzero$  to
the induced chamber $D\spar{1}$ adjacent to $\Dzero$ across the wall $(v_1)\sperp$,
where
$$
\setlength{\arraycolsep}{2pt}
\begin{array}{cccccccccccccc}
2 v_{1} &\;=\;&( 64 , & 32 , & -6 , & -9 , & -72 , & -48 , & -96 , & -142 , & -116 , & -88 , & -60 , & -30, \phantom{).}\\
& &&&&& -70 , & -48 , & -94 , & -140 , & -114 , & -86 , & -58 , & -30 ).
\end{array}
$$
As in the previous subsection,
we denote by $B$ the branch curve of the sextic 
 double plane
$Y\to \P^2$ associated with the polarization $\tlhzone$ of $\Xsb{0}$ given in Table~\ref{table:hs0}.
By a direct calculation, 
we see that 
the line configuration of $(\Xsb{0}, \tlhzone)$ is full, 
and hence, up to finite possibilities,
the projective equivalence class of $Y\to \P^2$
is determined by the line configuration on $(\Xsb{0}, \tlhzone)$.
From $\CCC_0(\tlhzone)$, 
we see that
$\Sing(B)$ consists of one ordinary cusp $q_0$ and eight   ordinary nodes $n_0, \dots, n_7$.
The set $\CCC_1(\tlhzone)$ consists of $48$ elements, and 
the action of  $\gen{\dpinvol{\tlhzone}}$ decomposes $\CCC_1(\tlhzone)$  into the union of $24$ orbits of length $2$.
Hence $B$ does not have a line as an irreducible component, 
and $B$ has $24$ splitting lines.
We put
$$
\TTT:=\{\{0, 1, 5\},\; \{0, 2, 6\},\;\{0, 3, 4\},\; \{1, 2, 4\},\;\{1, 3, 7\}, \;\{2, 5, 7\},\; \{3, 5, 6\}, \; \{4, 6, 7\}\}.
$$
Under suitable numbering of the ordinary nodes $n_0, \dots, n_7$,
the splitting lines are
$$
\ell_{0i}\;\; (i=0, \dots, 7),
\quad
\ell_{i} \;\;(i=0, \dots, 7),
\quad
m_{ijk} \;\; (\{i,j,k\}\in \TTT),
$$
where 
$$
\Sing (B)\cap \ell_{0i}=\{q_0, n_i\},  \quad \Sing (B)\cap \ell_{i}=\{n_i\}, \quad
\Sing (B)\cap m_{ijk}=\{n_i, n_j, n_k\}.
$$
Since a triplet of  ordinary nodes of $B$ is collinear, we conclude that $B$ is irreducible.
Note that,
if three ordinary nodes $n_i, n_j, n_k$ are on a line $\ell\subset \P^2$,
then $\ell$ is splitting for $B$, and hence $\{i, j, k\}\in \TTT$.
Therefore no three of $n_0, n_1, n_2, n_3$ are collinear.
%
%
Choosing homogeneous coordinates of $\P^2$
in such a way that
$$
n_0=[1: 0: 0],
\;\;
n_1=[0: 1: 0],
\;\;
n_2=[0: 0: 1],
\;\;
n_3=[1:1: 1],
$$
we see that
$$
n_4=[0: 1: 1],
\;\;
n_5=[1: \eta: 0],
\;\;
n_6=[1, 0, \bar{\eta}],
\;\;
n_7=[1: \eta: 1],
$$
where $\eta$ is a root of $z^2-z+1=0$.
\subsection{The symplectic involution  $\specialsymp$}\label{subsec:sigma14}
We examine the symplectic involution  $\specialsymp$ on $\Xsb{1}$
that maps 
the induced chamber $\Dzero$  to
the induced chamber $D\spar{4}$ adjacent to $\Dzero$ across the wall $(v_4)\sperp$,
where
$$
\setlength{\arraycolsep}{2pt}
\begin{array}{cccccccccccccc}
2v_4 &\;=\;&( 44 , & 22 , & -10 , & -3 , & -52 , & -36 , & -70 , & -104 , & -84 , & -64 , & -44 , & -22,  \\
& &&&&& -46 , & -32 , & -62 , & -92 , & -76 , & -58 , & -40 , & -20 ).
\end{array} 
$$
Consider the vector
$$
\setlength{\arraycolsep}{2pt}
\begin{array}{cccccccccccccc}
h_{\sigma} &\;:=\;&( 60 , & 30 , & -14 , & -4 , & -69 , & -47 , & -92 , & -137 , & -111 , & -85 , & -58 , & -29,  \\
& &&&&& -65 , & -45 , & -87 , & -129 , & -106 , & -81 , & -55 , & -28 )
\end{array}
$$
of $\Ssb{1}$ with 
$\intf{\hsigma, \hsigma}=2$.
By Corollary~\ref{cor:nefcriterion}, we see that $\hsigma$ is nef, and 
by Corollary~\ref{cor:Nikulinpol}, we see that $\hsigma$ is a polarization of degree $2$.
The polarization  $\hsigma$ is 
 invariant under $\specialsymp$, and hence
 $\dpinvol{\hsigma}$  and $\specialsymp$ commute.
The symplectic involution  $\specialsymp$ induces
a commutative diagram
$$
\begin{array}{ccc}
Y & \maprightsp{} &Y \\
\mapdown & & \mapdown \\
\P^2 &\maprightsb{\bar{\sigma}} &\P^2
\end{array}
$$
on the sextic double plane $Y\to \P^2$ associated with $\hsigma$. 
Let $B$ be the branch curve of $Y\to \P^2$,
which is invariant under the action of $\bar{\sigma}$ on $\P^2$.
 By a direct calculation, 
we see that 
the line configuration of $(\Xsb{1}, \hsigma)$ is full, 
and hence
the projective equivalence class of the double plane  $Y$
is determined by the line configuration of $(\Xsb{1}, \hsigma)$
up to finite possibilities.
From $\CCC_0(\hsigma)$, we see that $\Sing(B)$ consists of seven ordinary cusps
$q_0, q_1, q_1\sprime, q_2, q_2\sprime, q_3, q_3\sprime$.
In particular, $B$ is irreducible.
The involution $\bar{\sigma}$ of $\P^2$ fixes $q_0$ and interchanges $q_i$ and $q_i\sprime$ for $i=1,2,3$.
From $\CCC_1(\hsigma)$, we see that $B$ has $10$ splitting lines $\ell_0, \dots, \ell_{9}$.
Under suitable numbering, we have
$$
\renewcommand{\arraystretch}{1.2}
\begin{array}{ll}
\ell_0\cap \Sing(B)=\{q_0, q_1, q_1\sprime\},  &
\ell_1\cap \Sing(B)=\{q_0, q_2, q_3\},\\
 \ell_2\cap \Sing(B)=\{q_0, q_2\sprime, q_3\sprime\},   &
 \ell_3\cap \Sing(B)=\{q_1, q_2, q_3\sprime\}, \\
 \ell_4\cap \Sing(B)=\{q_1, q_2\sprime\},  &
\ell_5\cap \Sing(B)=\{q_1\sprime, q_2\},\\
 \ell_6\cap \Sing(B)=\{q_1\sprime, q_2\sprime, q_3\},  &
\ell_7\cap \Sing(B)=\{q_3, q_3\sprime\},\\
 \ell_8\cap \Sing(B)=\emptyset,  &
\ell_9\cap \Sing(B)=\emptyset.
\end{array}
$$
%
The involution $\bar{\sigma}$
fixes $\ell_0$ and $\ell_7$,
and interchanges two lines in the pairs $\{\ell_1, \ell_2\}$, $\{\ell_3, \ell_6\}$, $\{\ell_4, \ell_5\}$ and $\{\ell_8, \ell_9\}$.
\subsection{The double-plane involutions $\dpinvol{\tlhki{1}{11}}$, $\dpinvol{\tlhki{1}{12}}$, $\dpinvol{\tlhki{2}{6}}$, $\dpinvol{\tlhki{2}{7}}$}\label{subsec:2A9}
These four double-plane involutions have the following common feature.
We say that a projective plane curve $B$  of degree $6$ is \emph{of type $LQ$}
if the following hold;
\begin{itemize}
\item[(i)] $B$ is the union of a line $L$ and an irreducible quintic curve $Q$,
\item[(ii)] $L$ and $Q$ intersect at a point $P_0$ with intersection multiplicity $5$,
\item[(iii)] $Q$ is smooth at $P_0$, 
\item[(iv)] the singular locus $\Sing (Q)$ of $Q$ consists of a point $P_1$ of type $A_9$, and 
\item[(v)] the line $\ell$ passing through $P_0$ and $P_1$ intersects $Q$ at $P_1$ with intersection multiplicity $4$.
\end{itemize}
If $B$ is of type $LQ$,  then the $ADE$-type of $\Sing(B)$ is $2A_9$, and the line $\ell$ in the condition (v)
is splitting for $B$.
\par
\medskip
Let $h$ be $\tlhki{1}{11}$, $\tlhki{1}{12}$,  $\tlhki{2}{6}$ or \,$ \tlhki{2}{7}$.
We put $k=1$ if $h$ is $\tlhki{1}{11}$ or $\tlhki{1}{12}$, and $k=2$ if $h$ is $\tlhki{2}{6}$ or $\tlhki{2}{7}$,
so that $h\in \Sk$ and $\dpinvol{h}\in \Aut(\Xk)$.
The dual graph of the set $\CCC_0 (h)$ is a root system of type $2A_9$.
The set $\CCC_1(h)$ consists of $3$ elements,
and $\gen{\dpinvol{h}}$ decomposes it into the union of two orbits of length $1$ and $2$.
The union of $\CCC_0(h)$ and $\CCC_1(h)$ 
generates a sublattice of rank $19$ in $\Sk$.
Hence, unfortunately, the line configuration of $(\Xk, h)$ is not full.
The branch curve of $(X_k, h)$ is of type $LQ$.
\par
\medskip
We consider two vectors
$$
\setlength{\arraycolsep}{2pt}
\begin{array}{cccccccccccccc}
h^{\prime} &\;:=\;&( 172 , & 83 , & -34 , & -15 , & -191 , & -131 , & -257 , & -382 , & -310 , & -238 , & -161 , & -83,  \\
& &&&&& -189 , & -124 , & -248 , & -372 , & -301 , & -230 , & -154 , & -77 ), \\
h^{\prime\prime} &\;:=\;&( 183 , & 88 , & -36 , & -16 , & -200 , & -138 , & -269 , & -400 , & -325 , & -250 , & -169 , & -88,  \\
& &&&&& -204 , & -134 , & -268 , & -401 , & -325 , & -249 , & -166 , & -83 )
\end{array}
$$ 
in $\Ssb{2}$ of square-norm $2$.
By Corollaries~\ref{cor:nefcriterion} and~\ref{cor:Nikulinpol}, 
we see that they are  polarizations of degree $2$.
We have $\dpinvol{h\sprime}=\dpinvol{h\spprime}=\dpinvol{\tlhki{2}{7}}$.
Unfortunately again, the line configurations of $(\Xsb{2}, h\sprime)$ and  $(\Xsb{2}, h\spprime)$ are  not full.
The $ADE$-type of the singularities of the branch curve  of $(X_2, h\sprime)$ is $E_6+A_{11}$,
whereas that of $(X_2, h\spprime)$ is $A_{15}+A_{3}$.
\section{The automorphism group of the Enriques surface $\Zzero$}\label{sec:Enriques}
In this section, we compute the automorphism group $\Aut(\Zzero)$ 
of the Enriques surface $\Zzero:=\Xsb{0}/\gen{\Enrinvol}$,
and prove Proposition~\ref{prop:injZ} and Theorem~\ref{thm:autZ}.
\par
\medskip
We put
$$
\Ssb{0}^+:=\shortset{v\in \Ssb{0}}{v^{\Enrinvol}=v},
\quad
\Ssb{0}^-:=\shortset{v\in \Ssb{0}}{v^{\Enrinvol}=-v}.
$$
They are orthogonal complement to each other in $\Ssb{0}$.
Let $\pi\colon \Xsb{0}\to \Zzero$ be the universal covering of $\Zzero$ by $\Xsb{0}$.
Then the pull-back by $\pi$ identifies the primitive sublattice $\Ssb{0}^+$ 
of $\Ssb{0}$ with the lattice $\SZ(2)$.
From the matrix representation~(Table~\ref{table:Enrinvol}) of $\Enrinvol$, 
we see that $\Ssb{0}^+$ is generated by the vectors $f_1, \dots, f_{10}$ given in Table~\ref{table:EmbZ}.
From now on,
we consider $f_1, \dots, f_{10}$ as a basis of $\SZ$ by $\pi^*$.
The Gram matrix
$$
(\intfZ{f_i, f_j})=(\intfS{f_i, f_j}/2)
$$
of $\SZ$ with respect to this basis is given in Table~\ref{table:GramSZ}.
\begin{table}
%
%
{\small
$$
\setlength{\arraycolsep}{2.2pt}
\renewcommand{\arraystretch}{1.6}
\begin{array}{ccccccccccccccccccccccc}
f_{1} &:=(&\hskip -4pt 1, & 0,& 2,& -1,& 0,& 0,& 0,& 0,& -4,& 0,& 0,& 0,& -4,& -1,& -2,& -8,& -6,& -3,& -4,& -5 & \hskip -4pt )\\ 
f_{2} &:=(&\hskip -4pt  0, & 1,& 1,& 0,& 0,& 0,& 0,& 0,& -3,& 0,& 0,& 0,& -3,& -1,& -2,& -6,& -5,& -3,& -3,& -3 &\hskip -4pt )\\ 
f_{3} &:=(&\hskip -4pt  0, & 0,& 3,& -1,& 0,& 0,& 0,& 0,& -6,& 0,& 0,& 0,& -4,& 0,& 0,& -8,& -6,& -2,& -4,& -6 &\hskip -4pt )\\ 
f_{4} &:=(&\hskip -4pt  0, & 0,& 0,& 0,& 1,& 0,& 0,& 0,& 0,& 0,& 0,& 0,& 0,& 0,& 0,& 0,& 0,& 0,& 1,& 0 &\hskip -4pt )\\ 
f_{5} &:=(&\hskip -4pt  0, & 0,& 0,& 0,& 0,& 1,& 0,& 0,& 0,& 0,& 0,& 0,& 0,& 1,& 0,& 0,& 0,& 0,& 0,& 0 &\hskip -4pt )\\ 
f_{6} &:=(&\hskip -4pt  0, & 0,& 0,& 0,& 0,& 0,& 1,& 0,& 1,& 0,& 0,& 0,& 0,& -1,& -1,& 0,& -1,& 0,& 0,& 0 &\hskip -4pt )\\ 
f_{7} &:=(&\hskip -4pt  0, & 0,& 0,& 0,& 0,& 0,& 0,& 1,& 0,& 0,& 0,& 0,& 0,& 0,& 0,& 0,& 0,& 0,& -1,& -1 &\hskip -4pt )\\ 
f_{8} &:=(&\hskip -4pt  0, & 0,& 0,& 0,& 0,& 0,& 0,& 0,& 0,& 1,& 0,& 0,& -1,& 0,& 0,& -1,& 0,& -1,& 0,& 1 &\hskip -4pt )\\ 
f_{9} &:=(&\hskip -4pt  0, & 0,& 0,& 0,& 0,& 0,& 0,& 0,& 0,& 0,& 1,& 0,& 1,& 0,& 0,& 0,& 0,& 0,& 0,& 0 &\hskip -4pt )\\ 
f_{10} &:=(&\hskip -4pt  0, & 0,& 0,& 0,& 0,& 0,& 0,& 0,& 0,& 0,& 0,& 1,& -1,& 0,& 0,& -1,& 0,& 0,& 0,& 0 &\hskip -4pt )\\
\end{array}
$$
}
\caption{A basis of $\SZ$}\label{table:EmbZ}
\end{table}
\begin{table}
%
%
{\small
$$
\left[ \begin {array}{cccccccccc} 
-54&-30&-78&0&0&6&-5&1&0&-2\\ \noalign{\medskip}
-30&-20&-45&0&0&4&-3&0&0&-1\\ \noalign{\medskip}
-78&-45&-114&0&0&9&-7&1&0&-3\\ \noalign{\medskip}
0&0&0&-2&0&0&1&0&0&0\\ \noalign{\medskip}
0&0&0&0&-2&1&0&0&0&0\\ \noalign{\medskip}
6&4&9&0&1&-4&1&2&0&1\\ \noalign{\medskip}
-5&-3&-7&1&0&1&-2&1&0&0\\ \noalign{\medskip}
1&0&1&0&0&2&1&-4&1&-1\\ \noalign{\medskip}
0&0&0&0&0&0&0&1&-2&1\\ \noalign{\medskip}
-2&-1&-3&0&0&1&0&-1&1&-2\end {array}
 \right] 
$$
}
\vskip 2pt
\caption{The Gram matrix of $\SZ$}\label{table:GramSZ}
\end{table}
\par
Note that we have 
$$
\Cen=\set{g\in \Aut(\Xsb{0})}{(\Ssb{0}^+)^g=\Ssb{0}^+}.
$$
Hence we have a natural action
$$
\psi\colon \Cen \to \OG(\Ssb{0}^+)
$$
of $\Cen$ on $\Ssb{0}^+$.
With the identifications $\OG(\Ssb{0}^+)\isom \OG(\SZ)$ by $\pi^*$
and  $\Cen/\gen{\Enrinvol} \isom \Aut(\Zzero)$ by $\zeta$,
we see that Proposition~\ref{prop:injZ} follows from
\begin{equation}\label{eq:Kerpsi}
\Ker \psi=\gen{\Enrinvol}.
\end{equation}
Suppose that $g\in \Ker\psi$ so that $g$ acts on $\Ssb{0}^+$ trivially.
Since $\Enrinvol\in \Aut(\Xsb{0}, \asb{0})$,
we have $\asb{0}\in \Ssb{0}^+$ and hence $\asb{0}^g=\asb{0}$.
Consequently, we have $\Ker\psi\subset \Aut(\Xsb{0}, \asb{0})$.
Calculating $\psi(g)$ for the $1440$ elements of $\Aut(\Xsb{0}, \asb{0})$
by means of their matrix representations,
we prove~\eqref{eq:Kerpsi} and hence Proposition~\ref{prop:injZ}.
\par
By Remark~\ref{rem:isomPGL2F9}, 
in order to prove 
the first assertion of Theorem~\ref{thm:autZ},
it is enough to show that $\zeta(\Aut(\Xsb{0}, \asb{0}))=\Aut(\Xsb{0}, \asb{0})/\gen{\Enrinvol}$
is a non-splitting extension of $\Z/2\Z$ by $\AAAA_6$
and to calculate the conjugacy classes of this group.
Since the symplectic subgroup of $\Aut(\Xsb{0}, \asb{0})$ is isomorphic to $\AAAA_6$,  
we see that $\zeta(\Aut(\Xsb{0}, \asb{0}))$ contains a normal subgroup isomorphic to $\AAAA_6$ as a subgroup of  index $2$. 
By direct calculations, we confirm that every element of order $2$ of $\zeta(\Aut(\Xsb{0}, \asb{0}))$  belongs to this normal subgroup.
Hence the extension is non-splitting.
The conjugacy classes of $\zeta(\Aut(\Xsb{0}, \asb{0}))$ are calculated as follows:
$$
\begin{array}{c|cccccccc}
\textrm{order} & 1& 2 & 3 & 4 & 4 & 5 & 8 & 8 \\
\hline
\textrm{card.} & 1& 45 & 80 & 90 & 180 & 144 & 90 & 90.
\end{array}
$$
Therefore $\zeta(\Aut(\Xsb{0}, \asb{0}))$ is isomorphic to $M_{10}$.
\par
\begin{table}
{\tiny 
$$
\setlength{\arraycolsep}{1pt}
\left[ \begin {array}{cccccccccccccccccccc} 
1076&533&-101&-148&-1217& -817&-1624&-2398&-1955&-1502&-1012&-522&-1176&-802&-1593&-2352&-1924&- 1460&-996&-500\\
 21&10&-2&-3&-23&-15&-30&-45&-37&-28 &-19&-10&-23&-16&-31&-46&-38&-29&-20&-10\\
 1224&606& -115&-168&-1386&-930&-1848&-2730&-2226&-1710&-1152&-594&-1338&-912&- 1812&-2676&-2190&-1662&-1134&-570\\
 1794&888&-168&- 247&-2028&-1362&-2706&-3996&-3258&-2502&-1686&-870&-1962&-1338&-2658&- 3924&-3210&-2436&-1662&-834\\
 73&36&-7&-10&-83&-56&- 111&-164&-133&-102&-69&-36&-79&-54&-107&-158&-129&-98&-67&-34 \\
 20&10&-2&-3&-22&-15&-29&-43&-35&-27&-18&-9&-21&- 14&-28&-42&-34&-26&-18&-9\\
 66&33&-6&-9&-75&-50&-100 &-148&-120&-92&-62&-32&-73&-50&-99&-146&-120&-91&-62&-31 \\
 0&0&0&0&0&0&0&1&0&0&0&0&0&0&0&0&0&0&0&0 \\
 21&10&-2&-3&-24&-16&-32&-48&-39&-30&-20&-10&-22&- 15&-30&-44&-36&-27&-18&-9\\
 74&37&-7&-10&-83&-56&- 111&-164&-134&-103&-70&-36&-82&-56&-111&-164&-134&-102&-70&-35 \\
 0&0&0&0&0&0&0&0&0&0&1&0&0&0&0&0&0&0&0&0 \\
 65&32&-6&-9&-74&-49&-98&-145&-118&-91&-62&-32&-71 &-48&-96&-142&-116&-88&-60&-30\\
 0&0&0&0&0&0&0&0&0&0 &0&0&1&0&0&0&0&0&0&0\\
 23&11&-2&-3&-27&-18&-36&-53&- 43&-33&-22&-11&-26&-18&-35&-52&-42&-32&-22&-11\\
 64& 32&-6&-9&-72&-48&-96&-142&-116&-89&-60&-31&-69&-47&-94&-138&-113&-86&- 59&-30\\
 22&11&-2&-3&-25&-17&-33&-49&-40&-31&-21&-11 &-25&-17&-33&-49&-40&-30&-20&-10\\
 73&36&-7&-10&-82& -55&-110&-162&-132&-101&-68&-35&-80&-54&-108&-160&-131&-100&-68&-34 \\
 0&0&0&0&0&0&0&0&0&0&0&0&0&0&0&0&0&1&0&0 \\
 65&32&-6&-9&-74&-50&-99&-146&-119&-92&-62&-32&-70 &-48&-95&-140&-115&-88&-60&-30\\
 8&4&-1&-1&-8&-5&-10 &-15&-12&-9&-6&-3&-10&-7&-14&-20&-16&-12&-8&-4\end {array} \right] 
$$
}
\caption{The involution $\dpinvol{\tlhki{0}{3}}$}\label{table:dpinvoltlh03}
\end{table}
\begin{table}
{\tiny
$$
\zeta(\pnsymp)= \left[ \begin {array}{cccccccccc} 76&40&-67&-85&-56&-116&-170&-108&-
71&-34\\ 43&23&-38&-49&-32&-67&-97&-62&-41&-20
\\ 110&58&-97&-124&-82&-170&-248&-158&-104&-50
\\ 21&10&-18&-22&-15&-30&-44&-27&-18&-9
\\ 0&0&0&0&0&0&0&0&1&0\\ 12&5&-10&
-11&-8&-15&-22&-14&-10&-5\\ -12&-5&10&13&9&18&26&16&
11&6\\ -30&-15&26&31&22&42&62&39&26&13
\\ 30&15&-26&-33&-23&-45&-66&-41&-28&-15
\\ 0&0&0&-2&-1&-2&-3&-2&-1&0\end {array} \right] 
$$
\vskip 2pt
$$
\zeta(\dpinvol{\hki{0}{1}})= \left[ \begin {array}{cccccccccc} 74&37&-64&-91&-58&-121&-176&-112&-
74&-36\\ 43&21&-37&-54&-35&-73&-105&-67&-45&-23
\\ 96&48&-83&-120&-76&-160&-232&-148&-98&-48
\\ 22&11&-19&-26&-17&-34&-51&-31&-21&-11
\\ 14&7&-12&-17&-12&-23&-34&-21&-14&-7
\\ 0&0&0&0&0&1&0&0&0&0\\ 0&0&0&0&0
&0&1&0&0&0\\ -14&-7&12&17&11&21&32&20&13&7
\\ 0&0&0&0&0&0&0&0&1&0\\ -9&-5&8&
11&8&15&22&14&9&5\end {array} \right] 
$$
\vskip 2pt
$$
\zeta(\dpinvol{\hki{0}{2}})= \left[ \begin {array}{cccccccccc} 109&54&-94&-130&-91&-178&-258&-165&
-112&-59\\ 79&39&-68&-93&-64&-125&-183&-116&-79&-42
\\ 166&82&-143&-198&-138&-270&-392&-250&-170&-90
\\ 34&16&-29&-37&-26&-50&-74&-46&-31&-16
\\ 43&21&-37&-49&-33&-65&-96&-60&-40&-20
\\ -13&-6&11&17&11&22&33&21&14&7
\\ -21&-10&18&23&16&31&46&28&19&10
\\ 0&0&0&-2&-1&-2&-4&-2&-2&-1\\ 0&0
&0&0&0&0&0&0&1&0\\ -1&-1&1&1&1&1&2&1&0&0\end {array}
 \right] 
$$
\vskip 2pt
$$
\zeta(\dpinvol{\hki{0}{3}})= \left[ \begin {array}{cccccccccc} 142&69&-122&-163&-106&-215&-320&-
201&-138&-69\\ 94&46&-81&-106&-70&-141&-209&-132&-90
&-45\\ 206&100&-177&-236&-154&-312&-464&-292&-200&-
100\\ 30&15&-26&-35&-24&-48&-70&-44&-30&-15
\\ 35&17&-30&-39&-27&-53&-78&-49&-34&-17
\\ -21&-10&18&22&15&30&45&28&20&10
\\ -21&-10&18&24&17&33&48&30&20&10
\\ -2&-2&2&5&3&6&9&6&4&2\\ 44&22&-
38&-51&-35&-68&-101&-63&-43&-22\\ -22&-11&19&26&18&
35&51&32&22&12\end {array} \right] 
$$
\vskip 2pt
$$
\zeta(\dpinvol{\tlhki{0}{3}})= \left[ \begin {array}{cccccccccc} 581&290&-502&-666&-446&-888&-1310&-
822&-554&-286\\ 315&157&-272&-360&-241&-479&-707&-
443&-299&-155\\ 830&414&-717&-950&-636&-1266&-1868&-
1172&-790&-408\\ 138&68&-119&-157&-106&-210&-310&-
194&-131&-68\\ 43&21&-37&-49&-33&-65&-96&-60&-40&-20
\\ -73&-36&63&82&55&110&161&101&68&35
\\ -73&-36&63&82&55&109&162&101&68&35
\\ 60&30&-52&-66&-44&-88&-130&-81&-55&-28
\\ 0&0&0&0&0&0&0&0&1&0\\ 43&21&-37
&-49&-32&-65&-96&-60&-41&-21\end {array} \right] 
$$
}
\caption{Generators of $\Aut(\Zzero)$}\label{table:gensAutZ}
\end{table}
\begin{table}
{\small
$$
\begin{array}{crrrrrrrrrrc}
{(}&\hskip -4pt  0, & 0, &0, &0, &0, &0, &0, &0, &1, &0 &\hskip -4pt{)};\\
{(}&\hskip -4pt  0, & 0, &0, &0, &1, &0, &0, &0, &0, &0 &\hskip -4pt{)};\\
{(}&\hskip -4pt  0, & 0, &0, &1, &0, &0, &0, &0, &0, &0 &\hskip -4pt{)};\\
{(}&\hskip -4pt  10, & 6, &-9, &-12, &-8, &-16, &-24, &-15, &-10, &-5 &\hskip -4pt{)};\\
{(}&\hskip -4pt  13, & 6, &-11, &-16, &-11, &-22, &-32, &-20, &-14, &-8 &\hskip -4pt{)};\\
{(}&\hskip -4pt  13, & 6, &-11, &-12, &-8, &-16, &-24, &-15, &-10, &-5 &\hskip -4pt{)};\\
{(}&\hskip -4pt  14, & 7, &-12, &-17, &-12, &-23, &-34, &-21, &-14, &-7 &\hskip -4pt{)};\\
{(}&\hskip -4pt  14, & 7, &-12, &-17, &-11, &-22, &-33, &-21, &-14, &-7 &\hskip -4pt{)};\\
{(}&\hskip -4pt  14, & 7, &-12, &-15, &-10, &-20, &-30, &-18, &-13, &-7 &\hskip -4pt{)};\\
{(}&\hskip -4pt  14, & 7, &-12, &-15, &-10, &-19, &-30, &-18, &-12, &-6 &\hskip -4pt{)};\\
{(}&\hskip -4pt  21, & 10, &-18, &-24, &-16, &-32, &-48, &-30, &-20, &-10 &\hskip -4pt{)};\\
{(}&\hskip -4pt  21, & 10, &-18, &-22, &-15, &-30, &-44, &-27, &-18, &-9 &\hskip -4pt{)};\\
{(}&\hskip -4pt  22, & 11, &-19, &-26, &-17, &-34, &-51, &-31, &-21, &-11 &\hskip -4pt{)};\\
{(}&\hskip -4pt  22, & 11, &-19, &-25, &-17, &-34, &-49, &-31, &-21, &-11 &\hskip -4pt{)};\\
{(}&\hskip -4pt  22, & 11, &-19, &-24, &-16, &-32, &-48, &-30, &-21, &-11 &\hskip -4pt{)};\\
{(}&\hskip -4pt  22, & 11, &-19, &-24, &-16, &-32, &-48, &-30, &-20, &-9 &\hskip -4pt{)};\\
{(}&\hskip -4pt  26, & 12, &-22, &-29, &-19, &-38, &-57, &-35, &-24, &-12 &\hskip -4pt{)};\\
{(}&\hskip -4pt  30, & 15, &-26, &-35, &-24, &-48, &-70, &-44, &-30, &-15 &\hskip -4pt{)};\\
{(}&\hskip -4pt  30, & 15, &-26, &-33, &-23, &-45, &-66, &-42, &-28, &-14 &\hskip -4pt{)};\\
{(}&\hskip -4pt  30, & 15, &-26, &-33, &-23, &-45, &-66, &-41, &-28, &-15 &\hskip -4pt{)};\\
{(}&\hskip -4pt  30, & 15, &-26, &-32, &-21, &-42, &-63, &-39, &-26, &-13 &\hskip -4pt{)};\\
{(}&\hskip -4pt  34, & 16, &-29, &-37, &-26, &-50, &-74, &-46, &-31, &-16 &\hskip -4pt{)};\\
{(}&\hskip -4pt  35, & 17, &-30, &-39, &-27, &-53, &-78, &-49, &-34, &-17 &\hskip -4pt{)};\\
{(}&\hskip -4pt  35, & 17, &-30, &-39, &-26, &-51, &-76, &-47, &-32, &-17 &\hskip -4pt{)};\\
{(}&\hskip -4pt  42, & 20, &-36, &-46, &-31, &-62, &-91, &-57, &-39, &-20 &\hskip -4pt{)};\\
{(}&\hskip -4pt  43, & 21, &-37, &-49, &-33, &-65, &-96, &-60, &-40, &-20 &\hskip -4pt{)};\\
{(}&\hskip -4pt  43, & 21, &-37, &-47, &-32, &-63, &-94, &-58, &-40, &-20 &\hskip -4pt{)};\\
{(}&\hskip -4pt  44, & 22, &-38, &-51, &-35, &-68, &-101, &-63, &-43, &-22 &\hskip -4pt{)};\\
{(}&\hskip -4pt  44, & 22, &-38, &-50, &-33, &-66, &-98, &-61, &-42, &-21 &\hskip -4pt{)};\\
{(}&\hskip -4pt  44, & 22, &-38, &-48, &-33, &-64, &-95, &-59, &-40, &-20 &\hskip -4pt{)};\\
\end{array}
$$
}
\caption{The orbit $\tilde{o}_0$}\label{table: tildeorbit0}
\end{table}
\begin{table}
{\small
$$
\begin{array}{crrrrrrrrrrc}
{(}&\hskip -4pt  34, & 16, &-29, &-38, &-26, &-52, &-76, &-48, &-32, &-16 &\hskip -4pt  {)};\\
{(}&\hskip -4pt  34, & 16, &-29, &-36, &-24, &-48, &-72, &-44, &-30, &-16 &\hskip -4pt  {)};\\
{(}&\hskip -4pt  36, & 18, &-31, &-42, &-28, &-56, &-84, &-52, &-36, &-18 &\hskip -4pt  {)};\\
{(}&\hskip -4pt  36, & 18, &-31, &-40, &-28, &-54, &-80, &-50, &-34, &-18 &\hskip -4pt  {)};\\
{(}&\hskip -4pt  36, & 18, &-31, &-40, &-26, &-52, &-78, &-48, &-32, &-16 &\hskip -4pt  {)};\\
{(}&\hskip -4pt  52, & 26, &-45, &-60, &-40, &-80, &-118, &-74, &-50, &-26 &\hskip -4pt  {)};\\
{(}&\hskip -4pt  52, & 26, &-45, &-58, &-40, &-78, &-116, &-72, &-48, &-24 &\hskip -4pt  {)};\\
{(}&\hskip -4pt  52, & 26, &-45, &-56, &-38, &-76, &-112, &-70, &-48, &-24 &\hskip -4pt  {)};\\
{(}&\hskip -4pt  78, & 38, &-67, &-88, &-60, &-118, &-174, &-108, &-74, &-38 &\hskip -4pt  {)};\\
{(}&\hskip -4pt  78, & 38, &-67, &-86, &-58, &-114, &-170, &-106, &-72, &-36 &\hskip -4pt  {)};\\
\end{array}
$$
}
\caption{The orbit $\tilde{o}_3$}\label{table: tildeorbit3}
\end{table}
The second assertion of Theorem~\ref{thm:autZ} is confirmed by a direct calculation 
from the matrix representation~(Table~\ref{table:Enrinvol}) of $\Enrinvol$
and the matrix representation~(Table~\ref{table:dpinvoltlh03}) of $\dpinvol{\tlhki{0}{3}}$.
In fact, we see that every element of the set $\Involsb{0}\spar{3}$ commutes with $\Enrinvol$.
\par
In order to prove the third assertion of Theorem~\ref{thm:autZ},
we consider the positive cone
$$
\PPP(\Zzero):=(\SZ\tensor\R)\cap \PPP(\Xsb{0})
$$
of $\SZ$ that contains an ample class.
(Recall that we consider $\SZ$ as a $\Z$-submodule of $\Ssb{0}$ by $\pi^*$.)
We put
$$
\DZzero:= \PPP(\Zzero)\cap \Dzero,
$$
where $\Dzero$ is the induced chamber in  $N(\Xsb{0})$ given in Theorem~\ref{thm:D0}.
Let 
$$
\pr_Z\colon \Ssb{0}\tensor\R \;\to \;\SZ\tensor\R
$$
be the orthogonal projection.
Then we have
$$
\DZzero=\set{x\in \PPP(\Zzero)}{\intfZ{u, x}\ge 0 \;\;\textrm{for any}\;\; u\in \pr_Z(\Delta(\Dzero))}
$$
%
%
%
Since the interior point $\asb{0}$ of $\Dzero$ belongs to $\SZ$, 
the closed subset  $\DZzero$ 
of $\PPP(\Zzero)$ also contains $\asb{0}$ in its interior,
and hence $\DZzero$  is a chamber of $\PPP(\Zzero)$.
Moreover the finite group $\zeta(\Aut(\Xsb{0}, \asb{0}))$ acts on $\DZzero$.
For $v\in \Delta(\Dzero)$, the hyperplane 
$$
( \pr_Z(v))\sperp=(v)\sperp \cap \PPP(\Zzero)
$$
of $\PPP(\Zzero)$ is a wall of  $\DZzero$  if and only if 
the solution of the linear programing
to minimize $ \intfZ{\pr_Z(v), x}$  under the condition
$$
 \intfZ{u\sprime, x}{}\ge 0 \;\; \textrm{for all}\;\; u\sprime \in \pr_Z(\Delta(\Dzero)) \;\; \textrm{not proporsional to $\pr_Z(v)$}
$$
is unbounded to $-\infty$, 
where  the variable $x$ ranges through $\SZ\tensor\Q$.
(See Section~3 of~\cite{ShimadaAlgoAut}).
By this method,
we see that  the set of primitive outward defining vectors of walls of $\DZzero$
consists of $40$ vectors,
and they are decomposed into the two orbits $\tilde{o}_0$ and $\tilde{o}_3$
of cardinalities $30$ and $10$
under the action of $\zeta(\Aut(\Xsb{0}, \asb{0}))$,
where
$$
\tilde{o}_0=\set{2\,\pr_Z(r)}{r\in o_0},
\quad
\tilde{o}_3=\set{2\,\pr_Z(v)}{v\in o_3}.
$$
Here we use the dual basis of $\SZ$ not with respect to $\intfS{\phantom{a}, \phantom{a}} |_{\SZ}$ but
with respect to $\intfZ{\phantom{a}, \phantom{a}}$ .
(Recall that we have $|o_0|=60$ and $|o_{3}|=10$.)
The involution $\zeta(\dpinvol{\tlhki{0}{3}})$ maps $\DZzero$ to a chamber of $\PPP(\Zzero)$ adjacent to $\DZzero$ across the wall defined by a vector 
$$
(52, 26, -45, -60, -40, -80, -118, -74, -50, -26)
$$
in $\tilde{o}_3$
isomorphically.
In particular, the cone $N(\Zzero):=\PPP(\Zzero)\cap N(\Xsb{0})$ in $\PPP(\Zzero)$ is tessellated by chambers congruent to $\DZzero$
under the action of $\Aut(\Zzero)$.
Thus Theorem~\ref{thm:autZ} is proved.
\begin{remark}
The matrix representations of the generators 
$$
\zeta(\pnsymp), \;\; \zeta(\dpinvol{\hki{0}{1}}), \;\; \zeta(\dpinvol{\hki{0}{2}}), \;\; \zeta(\dpinvol{\hki{0}{3}}), \;\;\zeta(\dpinvol{\tlhki{0}{3}})
$$
of $\Aut(\Zzero)$ with respect to the basis $f_1, \dots, f_{10}$ of $\SZ$ are given Table~\ref{table:gensAutZ}.
\end{remark}
\begin{remark}
The interior point $\asb{0}$ of $\DZzero$ is written as 
$$
(122, 60, -105, -136, -92, -182, -270, -168, -114, -58)
$$
with respect to the basis $f_1, \dots, f_{10}$ of $\SZ$.
The elements of the orbits $\tilde{o}_0$ and $\tilde{o}_3$ are given in Tables~\ref{table: tildeorbit0}~and~\ref{table: tildeorbit3}.
By these data and the Gram matrix~(Table~\ref{table:GramSZ}) of $\SZ$,
we can completely determine the shape of the chamber $\DZzero$.

\end{remark}
\bibliographystyle{plain}

\def\cftil#1{\ifmmode\setbox7\hbox{$\accent"5E#1$}\else
  \setbox7\hbox{\accent"5E#1}\penalty 10000\relax\fi\raise 1\ht7
  \hbox{\lower1.15ex\hbox to 1\wd7{\hss\accent"7E\hss}}\penalty 10000
  \hskip-1\wd7\penalty 10000\box7} \def\cprime{$'$} \def\cprime{$'$}
  \def\cprime{$'$} \def\cprime{$'$}

\end{document}